\documentclass[a4paper]{amsart}

\usepackage{amssymb}
\usepackage{xypic}
\usepackage[only,mapsfrom]{stmaryrd}
\usepackage{graphicx}

\newtheorem{theorem}{Theorem}[section]
\newtheorem{lemma}[theorem]{Lemma}
\newtheorem{Lemma}[theorem]{Lemma}
\newtheorem{cor}[theorem]{Corollary}

\theoremstyle{definition}
\newtheorem{definition}[theorem]{Definition}

\theoremstyle{remark}
\newtheorem{remark}[theorem]{Remark}
\newtheorem{convention}[theorem]{Convention}

\numberwithin{equation}{section}
\newcommand{\Q}{\mathbb{Q}}
\newcommand{\C}{\mathbb{C}}

\newcommand{\R}{\mathbb{R}}

\DeclareMathOperator{\image}{im}

\DeclareMathOperator{\Hom}{Hom}

\DeclareMathOperator{\conv}{conv}
\DeclareMathOperator{\trace}{trace}

\DeclareMathOperator{\codim}{codim}
\DeclareMathOperator{\rank}{rank}
\DeclareMathOperator{\id}{Id}
\DeclareMathOperator{\ord}{ord}

\begin{document}

\title{Torus manifolds with non-abelian symmetries}


\author{Michael Wiemeler}
\address{ Department of Mathematics,
University of Fribourg, Chemin du Mus\'ee 23, CH-1700 Fribourg, Switzerland}
\email{michwiem@web.de}
\thanks{Part of the research was supported by SNF Grant No. 200021-117701}


\subjclass[2000]{Primary 57S15, 57S25}

\keywords{quasitoric manifolds, blow ups, compact non-abelian Lie-groups}



\begin{abstract}
Let \(G\) be a connected compact non-abelian Lie-group and \(T\) a maximal torus of \(G\).
A torus manifold with \(G\)-action is defined to be a smooth connected closed oriented manifold of dimension \(2\dim T\) with an almost effective action of \(G\) such that \(M^T\neq \emptyset\).
We show that if there is a torus manifold \(M\) with \(G\)-action, then the action of a finite covering group of \(G\) factors through \(\tilde{G}=\prod SU(l_i+1)\times\prod SO(2l_i+1)\times \prod SO(2l_i)\times T^{l_0}\).
The action of \(\tilde{G}\) on \(M\) restricts to an action of \(\tilde{G}'=\prod SU(l_i+1)\times\prod SO(2l_i+1)\times \prod U(l_i)\times T^{l_0}\) which has the same orbits as the \(\tilde{G}\)-action.

We define invariants of torus manifolds with \(G\)-action which determine their \(\tilde{G}'\)-equivariant diffeomorphism type.
We call these invariants admissible 5-tuples.
A simply connected torus manifold with \(G\)-action is determined by its admissible 5-tuple up to \(\tilde{G}\)-equivariant diffeomorphism.
Furthermore, we prove that all admissible 5-tuples may be realised by torus manifolds with \(\tilde{G}''\)-action, where \(\tilde{G}''\) is a finite covering group of \(\tilde{G}'\). 
\end{abstract}

\maketitle

\section{Introduction}
\label{sec:results}

A \(2n\)-dimensional smooth connected closed oriented manifold \(M\) with an almost effective action of an \(n\)-dimensional torus \(T\) is called \emph{torus manifold} if \(M^T\neq\emptyset\).
If each point of \(M\) has an invariant open neighborhood, which is weakly equivariantly diffeomorphic to an open subset of the standard action of \(T\) on \(\mathbb{C}^n\), then the orbit space \(M/T\) is an \(n\)-dimensional manifold with corners \cite[p.720-721]{1111.57019}.
In this case \(M\) is said to be \emph{quasitoric} if \(M/T\) is face preserving homeomorphic to a simple polytope \(P\).
In that case there are strong relations between the topology of \(M\) and the combinatorics of \(P\) \cite{davis91:_convex_coxet, 1012.52021}.

In this article we study torus manifolds, for which the \(T\)-action may be extended by an action of a connected compact non-abelian Lie-group \(G\).
To state our results, we introduce a bit more notations, which are used to describe the structure of torus manifolds. 

A closed, connected submanifold \(M_i\) of codimension two of a torus manifold \(M\), which is pointwise fixed by a one dimensional subtorus \(\lambda(M_i)\) of \(T\) and which contains a \(T\)-fixed point, is called \emph{characteristic} submanifold of \(M\).

All characteristic submanifolds \(M_i\) are orientable and an orientation of \(M_i\) determines a complex structure on the normal bundle \(N(M_i,M)\) of \(M_i\).

We denote the set of unoriented characteristic submanifolds of \(M\) by \(\mathfrak{F}\). 
If \(M\) is quasitoric the characteristic submanifolds of  \(M\) are given by the preimages of the facets of \(P\). In this case we identify \(\mathfrak{F}\) with the set of facets of \(P\).

Let \(G\) be a connected compact non-abelian Lie-group.
We call a smooth connected closed oriented \(G\)-manifold \(M\) a \emph{torus manifold with \(G\)-action} if \(G\) acts almost effectively on \(M\), \(\dim M=2\rank G\) and \(M^T\neq \emptyset\) for a maximal torus \(T\) of \(G\).
That means that \(M\) with the action of \(T\) is a torus manifold.
Because all maximal tori of \(G\) are conjugated, \(M\) together with the action of any other maximal torus \(T'\) is also a torus manifold.
Moreover, for all choices of a maximal torus of \(G\), we get up to weakly equivariant diffeomorphism the same torus manifold.
The \(G\)-action on \(M\) induces an action of the Weyl-group \(W(G)\) of \(G\) on \(\mathfrak{F}\) and the \(T\)-equivariant cohomology of \(M\). Results of Masuda~\cite{0940.57037} and Davis-Januszkiewicz~\cite{davis91:_convex_coxet} make a comparison of these actions possible.
From this comparison we get a description of the action on \(\mathfrak{F}\) and the isomorphism type of \(W(G)\).
Namely there is a partition of \(\mathfrak{F}=\mathfrak{F}_0\amalg\dots\amalg \mathfrak{F}_k\) and a finite covering group \(\tilde{G}=\prod_{j=1}^k G_j\times T^{l_0}\) of \(G\) such that
each \(G_{j_0}\) is non-abelian and \(W(G_{j_0})\) acts transitively on \(\mathfrak{F}_{j_0}\) and trivially on \(\mathfrak{F}_j\), \(j\neq j_{0}\), and the orientation of each \(M_i \in \mathfrak{F}_j\), \(j\neq j_{0}\), is preserved by \(W(G_{j_0})\) (see section~\ref{sec:charman}).

We call such \(G_i\) the \emph{elementary factors} of \(\tilde{G}\).

By looking at the orbits of the \(T\)-fixed points, we find that we may assume without loss of generality that all elementary factors are isomorphic to \(SU(l_i+1)\), \(SO(2l_i)\) or \(SO(2l_i+1)\) (see section~\ref{sec:Gact}). If \(M\) is quasitoric then all elementary factors are isomorphic to \(SU(l_i+1)\).

Now assume \(\tilde{G}=G_1\times G_2\) with \(G_1=SO(2l_1)\) elementary.
Then the restriction of the action of \(G_1\) to \(U(l_1)\) has the same orbits as the \(G_1\)-action (see section~\ref{sec:so2l}).
The following theorem shows that the classification of simply connected torus manifolds with \(\tilde{G}\)-action reduces to the classification of torus manifolds with \(U(l_1)\times G_2\)-action.

\begin{theorem}[Theorem~\ref{sec:case-g_1=so2l_1-2}]
\label{thm:intro2}
  Let \(M,M'\) be two simply connected torus manifolds with \(\tilde{G}\)-action, \(\tilde{G}=G_1\times G_2\) with \(G_1=SO(2l_1)\) elementary.
  Then \(M\) and \(M'\) are \(\tilde{G}\)-equivariantly diffeomorphic if and only if they are \(U(l_1)\times G_2\)-equivariantly diffeomorphic.
\end{theorem}

By applying a blow up construction along the fixed points of an elementary factor of \(\tilde{G}\) isomorphic to \(SU(l_i+1)\) or \(SO(2l_i+1)\), we get a fiber bundle over a complex or real projective space with some torus manifold as fiber.

This construction may be reversed and we call the inverse construction a blow down.
With this notation we get:

\begin{theorem}[Corollaries~\ref{cor:hhh},~\ref{sec:case-g_1-=-9},~\ref{cor:so3}, Theorem~\ref{sec:case-g_1=so2l_1-+1-2}]
\label{thm:intro}
  Let \(\tilde{G}=G_1\times G_2\), \(M\) a torus manifold with \(G\)-action such that \(G_1\) is elementary and \(l_2=\rank G_2\).
  \begin{itemize}
  \item If \(G_1=SU(l_1+1)\) and \(\#\mathfrak{F}_1=2\) in the case \(l_1=1\), then \(M\) is the blow down of a fiber bundle \(\tilde{M}\) over \(\C P^{l_1}\) with fiber some \(2l_2\)-dimensional torus manifold with \(G_2\)-action along an invariant submanifold of codimension two.
    Here the \(G_1\)-action on \(\tilde{M}\) covers the standard action of \(SU(l_i+1)\) on \(\C P^{l_1}\).
  \item If \(G_1=SO(2l_1+1)\) and \(\#\mathfrak{F}_1=1\) in the case \(l_1=1\), then \(M\) is a blow down of a fiber bundle \(\tilde{M}\) over \(\R P^{2l_1}\) with fiber some \(2l_2\)-dimensional torus manifold with \(G_2\)-action along an invariant submanifold of codimension one or a Cartesian product of a \(2l_1\)-dimensional sphere and a \(2l_2\)-dimensional torus manifold with \(G_2\)-action.
    In the first case the \(G_1\)-action on \(\tilde{M}\) covers the standard action of \(SO(2l_1+1)\) on \(\R P^{2l_1}\).
    In the second case \(G_1\) acts  in the usual way on \(S^{2l_1}\).
  \end{itemize}
\end{theorem}

If all elementary factors of \(\tilde{G}\) are isomorphic to \(SO(2l_i+1)\) or \(SU(l_i+1)\), then we may iterate this construction.
By this iteration we get a complete classification of torus manifolds with \(\tilde{G}\)-action up to \(\tilde{G}\)-equivariant diffeomorphism in terms of admissible 5-tuples (Theorem~\ref{thm:class1}).
For general \(G\) we have \(\tilde{G}=\prod SU(l_i+1)\times\prod SO(2l_i+1)\times SO(2l_i)\times T^{l_0}\).
We may restrict the action of \(\tilde{G}\) to \(\prod SU(l_i+1)\times \prod SO(2l_i+1)\times \prod U(l_i) \times T^{l_0}\).
Therefore we get invariants for torus manifolds with \(G\)-action from the above classification.
With Theorem~\ref{thm:intro2}, we see that these invariants determine the \(G\)-equivariant diffeomorphism type of simply connected torus manifolds with \(G\)-action. 

At the end we apply our classification to get more explicit results in special cases. These are:

For the special case \(G_2=\{1\}\) we get:
\begin{cor}[Corollary \ref{sec:g-action-m-2}]
  Assume that \(G\) is elementary and \(M\) a torus manifold with \(G\)-action.
  Then \(M\) is equivariantly diffeomorphic to \(S^{2l}\) or \(\mathbb{C}P^l\) if \(G=SO(2l+1),SO(2l)\) or \(G=SU(l+1)\), respectively.
\end{cor}

We recover certain results of Kuroki~\cite{kuroki_pre_1_2009,kuroki_pre_2_2009,kuroki_pre_3_2009} who gave a classification of torus manifolds with \(G\)-action and \(\dim M/G\leq 1\) (see Corollaries~\ref{sec:classification-4} and~\ref{sec:classification-8}).

For quasitoric manifolds we have the following result.
\begin{theorem}[Corollary~\ref{sec:classification-5}]
  If \(G\) is semi-simple and \(M\) a quasitoric manifold with \(G\)-action, then
  \begin{equation*}
    \tilde{G}=\prod_{i=1}^kSU(l_i+1)
  \end{equation*}
  and \(M\) is equivariantly diffeomorphic to a product of complex projective spaces.
\end{theorem}

Furthermore, we give an explicit classification of simply connected torus manifolds with \(G\)-action such that \(\tilde{G}\) is semi-simple and has two simple factors.

\begin{theorem}[Corollaries~\ref{sec:g-action-m-2},~\ref{sec:classification-7},~\ref{sec:classification-6}]
  Let \(\tilde{G}=G_1\times G_2\) with \(G_i\) simple and \(M\) a simply connected torus manifold with \(G\)-action.
  Then \(M\) is one of the following:
  \begin{equation*}
    \mathbb{C}P^{l_1}\times\mathbb{C}P^{l_2},\;\; \mathbb{C}P^{l_1}\times S^{2l_2},\;\; \#_i(S^{2l_1}\times S^{2l_2})_i,\;\; S^{2l_1+2l_2}
  \end{equation*}
The \(\tilde{G}\)-actions on these spaces is unique up to equivariant diffeomorphism.
\end{theorem}

The paper is organized as follows.
In section~\ref{sec:charman} we investigate the action of the Weyl-group of \(G\) on \(\mathfrak{F}\) and  \(H^*_T(M)\).
In section~\ref{sec:Gact} we determine the orbit-types of the \(T\)-fixed points in \(M\) and the isomorphism types of the elementary factors of \(G\).
In section~\ref{sec:blow} the basic properties of the blow up construction are established.
In section~\ref{sec:su} actions with elementary factor \(G_1=SU(l_1+1)\) are studied.
In section~\ref{sec:so2l} we give an argument which reduces the classification problem for actions with an elementary factor \(G_1=SO(2l_1)\) to that with an elementary factor \(SU(l_1)\).
In section~\ref{sec:so} we classify torus manifolds with \(G\)-action with elementary factor \(G_1=SO(2l_1+1)\).
In section~\ref{sec:classif} we iterate the classification results of the previous sections and illustrate them with some applications.
There are two appendices with preliminary facts on Lie-groups and torus manifolds.

I would like to thank Prof. Anand Dessai for helpful discussions.
I would also like to thank Prof. Mikiya Masuda for a simplification of the proof of Lemma~\ref{lem:action-weyl-group}.

\section{The action of the Weyl-group on $\mathfrak{F}$}
\label{sec:charman}

Let \(G\) be a compact connected Lie-group of rank \(n\) and \(T\) a maximal torus of \(G\).
Moreover, let \(M\) be a torus manifold with \(G\)-action.
That means that \(G\) acts almost effectively on the \(2n\)-dimensional smooth closed connected oriented manifold \(M\) such that \(M^T\neq \emptyset\).
We call a closed connected submanifold \(M_i\) of codimension two of \(M\), which is pointwise fixed by a one-dimensional subtorus \(\lambda(M_i)\) of \(T\) and which contains a \(T\)-fixed point, a characteristic submanifold of \(M\). 
If \(g\) is an element of the normalizer \(N_GT\) of \(T\) in \(G\), then, for every characteristic submanifold \(M_i\), \(gM_i\) is also a characteristic submanifold.
Therefore there are  actions of \(N_GT\) and the Weyl-group of \(G\) on \(\mathfrak{F}\).

In this section we describe this action of the Weyl-group of \(G\) on \(\mathfrak{F}\). At first we recall the definition of the equivariant cohomology of a \(G\)-space \(X\).
Let \(EG \rightarrow BG\) be a universal principal \(G\)-bundle.
Then \(EG\) is a contractible free right \(G\)-space.
If \(T\) is a maximal torus of \(G\), then we may identify \(ET=EG\) and \(BT=EG/T\).
The Borel-construction \(X_G\) of \(X\) is the orbit space of the right action \(((e,x),g) \mapsto (eg,g^{-1}x)\) on \(EG \times X\).
The equivariant cohomology \(H^*_G(X)\) of \(X\) is defined as the cohomology of \(X_G\).

In this section we take all cohomology groups with coefficients in \(\Q\).

The \(G\)-action on  \(EG \times X\) induces a right action of the normalizer of \(T\) on \(X_T\).
Therefore it induces a left action of the Weyl-group of \(G\) on the \(T\)-equivariant cohomology of \(X\).

Now let \(X=M\) be a torus manifold with \(G\)-action.
Denote the characteristic submanifolds of \(M\) by \(M_i\), \(i=1,\dots,m\).
Then, for any \(g \in N_GT\), \(M_{g(i)}=gM_i\) is also  a characteristic submanifold which depends only on the class \(w =[g]\in W(G)=N_GT/T\).
Therefore we get an action of the Weyl-group of \(G\) on \(\mathfrak{F}\).
Notice that \(M_i\in\mathfrak{F}\) is a fixed point of the \(W(G)\)-action on \(\mathfrak{F}\) if and only if it is invariant under the action of \(N_GT\) on \(M\).

A choice of an orientation for each characteristic submanifold of \(M\) together with an orientation for \(M\) is called an \emph{omniorientation} of \(M\).
If we fix an omniorientation for \(M\), then the \(T\)-equivariant Poincar\'e-dual \(\tau_i\) of \(M_i\) is well defined.

It is the image of the Thom-class of  \(N(M_i, M)_T\) under the natural map
 \begin{equation*}
   \psi: H^2(N(M_i,M)_T,N(M_i,M)_T-(M_i)_T) \rightarrow H^2(M_{T},M_{T}-(M_i)_T) \rightarrow H_T^2(M).
 \end{equation*}
Because of the uniqueness of the Thom-class \cite[p.110]{0298.57008} and  because \(\psi\) commutes with the action of \(W(G)\), we have
\begin{equation}
  \label{eq:jjkj}
   \tau_{g(i)}=\pm g^*\tau_i.
\end{equation}
Here the minus-sign occurs if and only if \(g|_{M_i}:M_{i}\rightarrow M_{g(i)}\) is orientation reversing.
We say that the class \([g]\in W(G)\) acts orientation preserving at \(M_i\) if this map is orientation preserving.
If \([g]\) acts orientation preserving at all characteristic submanifolds, then we say that \([g]\) preserves the omniorientation of \(M\).

Let \(S= H^{>0}(BT)\) and \(\hat{H}_T^*(M)= H_T^*(M)/S\text{-torsion}\). Because \(M^T\neq \emptyset\), there is an injection \(H^2(BT) \hookrightarrow H^2_T(M)\) and 
\begin{equation}
  \label{eq:lkjlk}
   H^2(BT)\cap S\text{-torsion}=\{0\}.
\end{equation}
By \cite[p. 240-241]{0940.57037}, the \(\tau_i\) are linearly independent in \(\hat{H}_T^*(M)\).
By Lemma 3.2 of \cite[p. 246]{0940.57037}, they form a basis of \(\hat{H}_T^2(M)\).

The Lie-algebra \(LG\) of \(G\) may be endowed with an Euclidean inner product which is invariant for the adjoint representation.
This allows us to identify the Weyl-group \(W(G)\) of \(G\) with a group of orthogonal transformations on the Lie-algebra \(LT\) of \(T\).
It is generated by reflections in the walls of the Weyl-chambers of \(G\) \cite[p. 192-193]{0581.22009}.
In the following we say that an element of \(W(G)\) is a reflection if and only if it is a reflection in a wall of a Weyl-chamber of \(G\).
An element \(w\in W(G)\) is a reflection if and only if it acts as a reflection on \(H^2(BT)\).

Here we say that \(A\in \text{Gl}(L)\) acts as a reflection on the \(\mathbb{Q}\)-vector space \(L\) if there is a decomposition \(L=L_+\oplus L_-\) with \(\dim_\Q L_-=1\) and \(A|_{L_\pm}=\pm \id\).
Notice that \(A\in \text{Gl}(L)\) acts as a reflection on \(L\) if and only if \(\ord A=2\) and \(\trace(A,L)=\dim_\Q L -2\).
 
\begin{lemma}
\label{lem:action-weyl-group}
  Let \(w \in W(G)\) be a reflection. Then there are the following possibilities for the action of \(w\) on \(\mathfrak{F}\):
\begin{enumerate}
\item \(w\) fixes all except exactly two elements of \(\mathfrak{F}\). It acts orientation preserving at all characteristic submanifolds.
\item \(w\) fixes  all except exactly two elements of \(\mathfrak{F}\).
  Denote the elements of \(\mathfrak{F}\) which are not fixed by \(w\) by \(M_1,M_2\).
  The action of \(w\) is orientation preserving at all characteristic submanifolds of \(M\) except \(M_1,M_2\).
  It is orientation reversing at \(M_1,M_2\).
\item \(w\) fixes all elements of \(\mathfrak{F}\).
  It acts orientation reversing at exactly one characteristic submanifold of \(M\).
\end{enumerate}
\end{lemma}
\begin{proof}
  Using the arguments given before Lemma \ref{lem:action-weyl-group}, we have the following commutative diagram of \(W(G)\)-representations with exact rows and columns
\begin{equation*}
  \xymatrix{
    & & S\text{-torsion in }H^2_T(M) \ar[d]\\
    0 \ar[r] & H^2(BT) \ar[r] &H^2_T(M) \ar^\phi[r]\ar[d] & H^2(M)\\
     &  & \hat{H}^2_T(M)\ar[d]\\
    &&0\\
}
\end{equation*}

Here \(\phi\) denotes the natural map \(H^2_T(M)\rightarrow H^2(M)\).

Because \(G\) is connected, the \(W(G)\)-action on \(H^2(M)\) is trivial.
By (\ref{eq:lkjlk}) the \(S\)-torsion in \(H^2_T(M)\) injects into \(H^2(M)\).
Therefore \(W(G)\) acts trivially on the \(S\)-torsion in \(H^2_T(M)\).

Because \(w\) is a reflection, we have \(\trace(w,H^2(BT))=\dim_\Q H^2(BT)-2\).
From the exact row in the diagram we get
\begin{align*}
  \trace(w,H^2_T(M)) &= \trace(w,H^2(BT)) + \trace(w,\image\phi)\\
  &= \dim_\Q H^2(BT)-2 +\dim_\Q \image\phi\\
  &= \dim_\Q H^2_T(M)-2.
\end{align*}

Similarly we get
\begin{align*}
    \trace(w,\hat{H}^2_T(M)) &=\trace(w,H^2_T(M))-\trace(w,S\text{-torsion in }H^2_T(M))\\
    &= \dim_\Q \hat{H}^2_T(M) -2.
\end{align*}
Now the statement follows from (\ref{eq:jjkj}) because the \(\tau_i\) form a basis of \(\hat{H}^2_T(M)\).
\end{proof}

\begin{lemma}
  An element \(w\in W(G)\) acts as a reflection on \(\hat{H}^2_T(M)\) if and only if it is a reflection.
\end{lemma}
\begin{proof}
  Because, by (\ref{eq:lkjlk}), \(H^2(BT)\) injects into \(\hat{H}^2_T(M)\), \(W(G)\) acts effectively on \(\hat{H}^2_T(M)\).
  Therefore we may identify \(W(G)\) with a subgroup of \(\text{Gl}(\hat{H}^2_T(M))\).

  If \(w\in W(G)\), then, as in the proof of Lemma~\ref{lem:action-weyl-group}, we see that
  \begin{equation*}
    \dim_\Q H^2(BT)- \trace(w,H^2(BT)) = \dim_\Q \hat{H}^2_T(M) - \trace(w,\hat{H}^2_T(M)).
  \end{equation*}
  Therefore, by the remark before Lemma~\ref{lem:action-weyl-group}, an element of \(W(G)\) of order two is a reflection if and only if it acts as a reflection on \(\hat{H}^2_T(M)\).
\end{proof}

Let \(\mathfrak{F}_0\) be the set of characteristic submanifolds, which are fixed by the \(W(G)\)-action on \(\mathfrak{F}\) and at which \(W(G)\) acts orientation preserving.
Furthermore let \(\mathfrak{F}_i\), \(i = 1,\dots, k\), be the other orbits of the \(W(G)\)-action on \(\mathfrak{F}\) and
 \(V_i\) the subspace of \(\hat{H}^2_T(M)\) spanned by the \(\tau_j\) with \(M_j\in\mathfrak{F}_i\).
 Then \(W(G)\) acts trivially on \(V_0\).
For \(i>0\), let \(W_i\) be the subgroup of \(W(G)\) which is generated by the reflections which act non-trivially on \(V_i\).
Then, by Lemma~\ref{lem:action-weyl-group}, \(W_i\) acts trivially on \(V_j\), \(j\neq i\).

By (\ref{eq:lkjlk}), \(H^2(BT)\) injects into \(\hat{H}^2_T(M)\).
Therefore \(W(G)\) acts effectively on \(\hat{H}^2_T(M)\).
This fact implies that the subgroups \(W_i\), \(i=1,\dots,k\), of  \(W(G)\) pairwise commute and \(\langle W_1,\dots, W_i\rangle\cap W_{i+1} =\{1\}\) for all \(i=1,\dots, k-1\).
Here \(\langle W_1,\dots, W_i\rangle\) denotes the subgroup of \(W(G)\) which is generated by \(W_1,\dots,W_i\).
Hence, we have an injective group homomorphism \(\prod W_i \rightarrow W(G)\), \((w_1,\dots,w_k)\mapsto w_1\dots w_k\).

\begin{lemma}
  The group homomorphism \(\prod W_i \rightarrow W(G)\), \((w_1,\dots,w_k)\mapsto w_1\dots w_k\) is an isomorphism.
\end{lemma}
\begin{proof}
  Because \(W(G)\) is generated by reflections and each reflection is contained in a \(W_i\), the above homomorphism is surjective.
  As noted before, it is injective.
  Therefore it is an isomorphism.
\end{proof}

\begin{lemma}
\label{sec:action-weyl-group-4}
  For each pair \(M_{j_1},M_{j_2}\in \mathfrak{F}_i\), \(i>0\), with \(M_{j_1}\neq M_{j_2}\) there is a reflection \(w\in W_i\) with \(w(M_{j_1})=M_{j_2}\).
\end{lemma}
\begin{proof}
  Because  \(\mathfrak{F}_i\) is an orbit of the \(W(G)\)-action on \(\mathfrak{F}\) and \(W(G)\) is generated by reflections, there is a \(M_{j_1}'\in \mathfrak{F}_i\) with \(M_{j_1}'\neq M_{j_2}\) and a reflection \(w\in W_i\) with \(w(M_{j_1}')=M_{j_2}\).
  
  Because \(W_i\) is generated by reflections and acts transitively on \(\mathfrak{F}_i\)
  the natural map \(W_i\rightarrow S(\mathfrak{F}_i)\) to the permutation group  \(S(\mathfrak{F}_i)\) of \(\mathfrak{F}_i\) is a surjection  by Lemma~\ref{lem:action-weyl-group} and Lemma 3.10 of \cite[p. 51]{0225.20020}.
  Therefore there is a \(w'\in W_i\) with
  \begin{align*}
    w'(M_{j_1})&=M_{j_1}', &  w'(M_{j_1}')&=M_{j_1}, & w'(M_{j_2})&=M_{j_2}.
  \end{align*}
  
  Now \(w'^{-1}ww'\in W_i\) is a reflection with the required properties.
\end{proof}

It follows from Lemma~\ref{lem:action-weyl-group} that for each pair \(M_{j_1},M_{j_2}\in \mathfrak{F}_i\), \(i>0\), with \(M_{j_1}\neq M_{j_2}\) there are at most two reflections,
which map \(M_{j_1}\) to \(M_{j_2}\).

If \(M_{j_1'}, M_{j_2'}\in \mathfrak{F}_i\) is another pair with \(M_{j_1'}\neq M_{j_2'}\), then one sees as in the proof of Lemma~\ref{sec:action-weyl-group-4} that there is a \(w'\in W_i\) with
\begin{align*}
  w'(M_{j_1'})&=M_{j_1},& w'(M_{j_2'})&=M_{j_2}.
\end{align*}
Therefore there is a bijection
\begin{align*}
  \{ w\in W_i;\; w \text{ reflection, } w(M_{j_1})=M_{j_2}\}&\rightarrow \{ w\in W_i;\; w \text{ reflection, } w(M_{j_1'})=M_{j_2'}\}\\
w &\mapsto w'^{-1}w w'.
\end{align*}
In particular, the number of reflections which map \(M_{j_1}\) to \(M_{j_2}\) does not depend on the choice of \(M_{j_1},M_{j_2}\in \mathfrak{F}_i\).

\begin{lemma}
\label{sec:action-weyl-group-6}
  Assume \(\# \mathfrak{F}_i>1\) and \(i>0\).
  If for each pair  \(M_{j_1},M_{j_2}\in \mathfrak{F}_i\) with \(M_{j_1}\neq M_{j_2}\) there is exactly one reflection in \(W_i\), which maps \(M_{j_1}\) to \(M_{j_2}\),
  then \(W_i\) is isomorphic to \(S(\mathfrak{F}_i)\cong W(SU(l_i+1))\) with \(l_i+1=\#\mathfrak{F}_i\).
\end{lemma}
\begin{proof}
  First we show that there is no reflection of the third type as described in Lemma~\ref{lem:action-weyl-group} in \(W_i\).
  Assume that \(w'\in W_i\) is a reflection of the third type.
  Then let \(M_1\in \mathfrak{F}_i\) be the characteristic submanifold at which \(w'\) acts orientation reversing.
  Furthermore, let \(M_1\neq M_2\in \mathfrak{F}_i\).
  
  Then by Lemma~\ref{sec:action-weyl-group-4} there is a reflection \(w\in W_i\) such that \(wM_1=M_2\).
  Hence, \(w'ww'\) is a reflection with \(w'ww'M_1=M_2\).
  Because \(w\) and \(w'ww'\) have a different orientation behaviour at \(M_1\), we have  \(w\neq w'ww'\), contradicting our assumption.

  To prove the lemma, it is sufficient to show that the kernel of the natural map \(W_i\rightarrow S(\mathfrak{F}_i)\) is trivial.
  Let \(w\) be an element of this kernel.
  Then for each \(\tau_j\in V_i\) we have
  \begin{equation*}
    w\tau_j=\pm \tau_j.
  \end{equation*}
  If we have \( w\tau_j= \tau_j\) for all \(\tau_j\in V_i\), then \(w=\id\).

  Now assume that \(w \tau_{j_0}=-\tau_{j_0}\) for a \(\tau_{j_0}\in V_i\).
  Then there are reflections \(w_1,\dots,w_n\in W_i\), \(n\geq 2\), with \(-\tau_{j_0}=w\tau_{j_0}=w_1\dots w_n\tau_{j_0}\).
  After removing some of the \(w_i\), we may assume that
  \begin{align*}
    w_i\dots w_n\tau_{j_0}&\neq \pm \tau_{j_0} & \text{for all }&  i=2,\dots, n,\\
    w_{i+1}\dots w_n\tau_{j_0}&\neq \pm w_{i}\dots w_n \tau_{j_0} & \text{for all }& i=2,\dots,n.\\
  \end{align*}
  Therefore, by Lemma~\ref{lem:action-weyl-group}, we have \(w_i\tau_{j_0}=\tau_{j_0}\) for \(2\leq i < n\).
  This equation  together with \(w\tau_{j_0}=-\tau_{j_0}\) implies
  \begin{equation*}
    w_n\dots w_2 w_1 w_2 \dots w_n \tau_{j_0}= - w_n \tau_{j_0}.
  \end{equation*}
  Therefore \( w_n\dots w_2 w_1 w_2 \dots w_n M_{j_0}= w_n M_{j_0}\).

  But \( w_n\dots w_2 w_1 w_2 \dots w_n\) is a reflection.
  Therefore, by assumption, we have
  \begin{equation*}
    w_n\dots w_2 w_1 w_2 \dots w_n =  w_n 
  \end{equation*}
  and
  \begin{equation*}
    w_n\tau_{j_0}= w_{n}w_{n-1}\dots w_2w_1w_2\dots w_n \tau_{j_0} = -w_n\tau_{j_0}.
  \end{equation*}
  Because \(w_n\tau_{j_0}\neq 0\), this is impossible.
  Hence, our assumption that \(w\tau_{j_0}=-\tau_{j_0}\) is false.

  Therefore the kernel is trivial.
\end{proof}

To get the isomorphism type of \(W_i\) in the case, where there is a pair  \(M_{j_1},M_{j_2}\in \mathfrak{F}_i\), \(i>0\), with \(M_{j_1}\neq M_{j_2}\) and exactly two reflections in \(W_i\), which map \(M_{j_1}\) to \(M_{j_2}\), we first give a description of the Weyl-groups of some Lie-groups.

Let \(L\) be an \(l\)-dimensional \(\Q\)-vector space with basis \(e_1,\dots,e_l\).
For \(1\leq i<j\leq l\) let \(f_{ij\pm},g_i\in \text{Gl}(L)\) such that
\begin{gather*}
  f_{ij+}e_k=
  \begin{cases}
     e_i & \text{if } k=j\\
     e_j & \text{if } k=i\\
    e_k & \text{else}
  \end{cases}\\
  f_{ij-}e_k=
  \begin{cases}
    - e_i & \text{if } k=j\\
    - e_j & \text{if } k=i\\
    e_k & \text{else}
  \end{cases}\\
  g_ie_k=
  \begin{cases}
    -e_i &\text{if } k=i\\
    e_k &\text{else}.
  \end{cases}
\end{gather*}
Then we have the following isomorphisms of groups \cite[p. 171-172]{0581.22009}:
\begin{align*}
  W(SU(l-1))\cong S(l)&\cong\langle f_{ij+}; 1\leq i<j\leq l\rangle,\\
  W(SO(2l))&\cong \langle f_{ij\pm}; 1\leq i<j\leq l\rangle,\\
  W(SO(2l+1))\cong W(Sp(l)) &\cong\langle f_{ij\pm},g_1; 1\leq i<j\leq l \rangle.
\end{align*}

From this description and Lemma~\ref{lem:action-weyl-group}, we get:
\begin{lemma}
\label{sec:action-weyl-group-5}
  If for each pair \(M_{j_1},M_{j_2}\in \mathfrak{F}_i\), \(i>0\), with \(M_{j_1}\neq M_{j_2}\) there are exactly two reflections in \(W_i\) which map \(M_{j_1}\) to \(M_{j_2}\),
  then  with \(l_i=\# \mathfrak{F}_i\)  we have
  \begin{enumerate}
  \item \(W_i\cong W(SO(2l_i))\) if there is no reflection of the third type as described in Lemma~\ref{lem:action-weyl-group} in \(W_i\).
  \item \(W_i\cong W(SO(2l_i+1)) \cong W(Sp(l_i))\)  if there is a reflection of the third type in \(W_i\).
  \end{enumerate}
\end{lemma}

By \cite[p. 233]{0581.22009}, \(G\) has a finite covering group \(\tilde{G}\) such that \(\tilde{G}=\prod_i G_i \times T^{l_0}\), where the \(G_i\) are simple simply connected compact Lie-groups. 
The Weyl-group of \(G\) is given by \(W(G)=\prod_i W(G_i)\).

We call two reflections \(w,w'\in W(G)\) equivalent if there are reflections \(w_1,\dots,w_k\in W(G)\) such that
\begin{align*}
  w&=w_1,&w'&=w_k,&[w_i,w_{i+1}]&\neq 1.
\end{align*}
Here \([w_i,w_{i+1}]\) denotes the commutator of  \(w_i\) and \(w_{i+1}\).
Because the Dynkin-diagram of a simple Lie-group is connected, each \(W(G_i)\) is generated by equivalent reflections.
Therefore each \(W(G_i)\) is contained in a \(W_j\).
Therefore we get
\(W_i=\prod_{j\in J_i}W(G_j)\).
Using Lemmas~\ref{sec:action-weyl-group-6} and~\ref{sec:action-weyl-group-5}, we deduce:
\begin{equation*}
  W_i=
  \begin{cases}
    W(G_j) & \text{for some } j \text{ if } W_i \not\cong W(SO(4))\\
    W(G_{j_1})\times W(G_{j_2}) & \text{ with } G_{j_1}\cong G_{j_2}\cong SU(2) \text{ if } W_i \cong W(SO(4)).
  \end{cases}
\end{equation*}
Therefore we may write \(\tilde{G}=\prod_i G_i\times T^{l_0}\) with  \(W_i=W(G_i)\) and \(G_i\) simple and simply connected or \(G_i=\text{Spin}(4)\).
In the following we will call these \(G_i\) the \emph{elementary factors} of \(\tilde{G}\).

We summarize the above discussion in the following lemma.
\begin{lemma}
\label{lem:1}
  Let \(M\) be a torus manifold with \(G\)-action and \(\tilde{G}\) as above.
  Then all \(G_i\) are non-exceptional, i.e. \(G_i= SU(l_i +1), \text{Spin}(2l_i), \text{Spin}(2l_i+1), Sp(l_i)\).

  The Weyl-group of an elementary factor \(G_i\) of \(\tilde{G}\) acts transitively on \(\mathfrak{F}_i\) and
  trivially on \(\mathfrak{F}_j\), \(j\neq i\).

For a given isomorphism type of \(G_i\), there are at most two possible values of \(\#\mathfrak{F}_i\).
The possible values of \(\#\mathfrak{F}_i\) are listed in the following table.
  \begin{center}
    \begin{tabular}{|c|c|}
      \(G_i\) &  \(\#\mathfrak{F}_i\)\\\hline\hline
      \(SU(2)=\text{Spin}(3)=Sp(1)\) &  \(1,2\) \\\hline
      \(\text{Spin}(4)\) & \(2\)\\\hline
      \(\text{Spin}(5)=Sp(2)\) & \(2\)\\\hline
      \(SU(4)=\text{Spin}(6)\) & \(3,4\) \\\hline
      \(SU(l_i+1)\), \(l_i \neq 1,3\) & \(l_i+1\)\\\hline
      \(\text{Spin}(2l_i+1)\), \(l_i > 2\) & \(l_i\)\\\hline
      \(\text{Spin}(2l_i)\), \(l_i > 3\) & \(l_i\)\\\hline
      \(Sp(l_i)\), \(l_i > 2\) &\(l_i\)
  \end{tabular}  
\end{center}
\end{lemma}

If we restrict our attention to quasitoric manifolds with \(G\)-action, then we get a much shorter list of possible isomorphism types of the elementary factors.
In fact, if \(M\) is  a quasitoric manifold with \(G\)-action, then, as shown in the next lemma, all elementary factors of \(G\) are isomorphic to \(SU(l_i+1)\) for some \(l_i\geq 1\).

\begin{lemma}
\label{sec:action-weyl-group-1}
  Let \(M\) be a quasitoric manifold with \(G\)-action.
  Then there is a covering group \(\tilde{G}\) of \(G\) with \(\tilde{G}= \prod_{i=1}^{k_1} SU(l_i + 1)\times T^{l_0} \).
\end{lemma}
\begin{proof}
  First we show for \(i>0\):
  \begin{equation}
    \label{eq:8}
    W_i \cong S(\mathfrak{F}_i).
  \end{equation}
  To do so,
  it is sufficient to prove that 
  there is an omniorientation on \(M\) which is preserved by the action of \(W(G)\). This is true if
  for every characteristic submanifold \(M_i\) and \(g \in N_GT\) such that \(gM_i =M_i\), \(g\) preserves the orientation of \(M_i\).
  Since \(G\) is connected, \(g\) preserves the orientation of \(M\) and acts trivially on  \(H^2(M)\).

  Because each vertex of the orbit polytope \(P\) of \(M\) is the intersection of exactly \(n\) facets of \(P\), every fixed point of the \(T\)-action on \(M\) is the transverse intersection of exactly \(n\) characteristic submanifolds.
  Thus, the Poincar\'e-dual \(PD(M_i)\in H^2(M)\) of \(M_i\) is non-zero because \(M_i\cap M^T\neq \emptyset\).
  Therefore \(g\) preserves the orientation of \(M_i\) since otherwise
  \begin{align*}  
    PD(M_i)&=\frac{1}{2}(PD(M_i)+PD(M_i))\\
    &=\frac{1}{2}(PD(M_i)+g^*PD(M_i))&&(g \text{ acts trivially on }H^2(M))\\
    &=\frac{1}{2}(PD(M_i)-PD(M_i))&& (g \text{ reverses the orientation of } M_i)\\
    &=0.
  \end{align*}
  
  This establishes (\ref{eq:8}).
  Recall that all simple compact simply connected Lie-groups having a Weyl-group isomorphic to some symmetric group are isomorphic to some \(SU(l+1)\). 
  Therefore all elementary factors of \(\tilde{G}\) are isomorphic to \(SU(l_i+1)\).
  From this the statement follows.
\end{proof}

\begin{remark}
\label{sec:action-weyl-group-3}
  In \cite{1111.57019} Masuda and Panov show that the cohomology with coefficients in \(\mathbb{Z}\) of a torus manifold \(M\) is generated by its degree-two part if and only if the torus action on \(M\) is locally standard and the orbit space \(M/T\) is a homology polytope.
That means that all faces of \(M/T\) are acyclic and all intersections of facets of \(M/T\) are connected.
In particular, each \(T\)-fixed point is the transverse intersection of \(n\) characteristic submanifolds.
Therefore the above lemma also holds in this case.
\end{remark}

For a characteristic submanifold \(M_i\) of \(M\), let \(\lambda(M_i)\)  denote the one-dimensional subtorus of \(T\) which fixes \(M_i\) pointwise.
The normalizer \(N_GT\) of \(T\) in \(G\) acts by conjugation on the set of one-dimensional subtori of \(T\).
The following lemma shows that 
\begin{equation*}
  \lambda:\mathfrak{F}\rightarrow \{\text{one-dimensional subtori of }T\} 
\end{equation*}
is  \(N_GT\)-equivariant.

\begin{lemma}
\label{sec:action-weyl-group}
  Let \(M\) be a torus manifold with \(G\)-action, \(g\in N_GT\) and \(M_i\subset M\) be a characteristic submanifold. Then we have:
  \begin{enumerate}
  \item\label{item:12} \(\lambda(gM_i) =g\lambda(M_i)g^{-1}\).
  \item\label{item:13} If \(gM_i=M_i\), then \(g\) acts orientation preserving on \(M_i\) if and only if
    \begin{equation*}
      \lambda(M_i)\rightarrow \lambda(M_i)\quad t \mapsto gtg^{-1}
    \end{equation*}
is orientation preserving.
  \end{enumerate}
\end{lemma}
\begin{proof}
  First we prove (\ref{item:12}).
  Let \(x\in M_i\) be a generic point.
    Then the identity component \(T_x^0\) of the stabilizer of \(x\) in \(T\) is given by \(T_x^0=\lambda(M_i)\). Therefore we have
    \begin{equation*}
      \lambda(gM_i)=T_{gx}^0=gT_x^0g^{-1}=  g\lambda(M_i)g^{-1}.
    \end{equation*}

    Now we prove (\ref{item:13}).
 An orientation of \(M_i\) induces a complex structure on \(N(M_i,M)\).
    We fix an isomorphism \(\rho:\lambda(M_i)\rightarrow S^1\) such that the action of \(t\in \lambda(M_i)\) on \(N(M_i,M)\) is given by multiplication with \(\rho(t)^m\), \(m>0\).
    The differential \(Dg:N(M_i,M)\rightarrow N(M_i,M)\) is orientation preserving if and only if it is complex linear.
    Otherwise it is complex anti-linear.
    Therefore for \(v\in N(M_i,M)\) we have
    \begin{align*}
      \rho(gtg^{-1})^m v &= (Dg) (Dt) (Dg)^{-1}v
      =(Dg)\rho(t)^m(Dg)^{-1}v\\
      &=\rho(t)^{\pm m} (Dg)(Dg)^{-1}v
      =\rho(t^{\pm 1})^m v.
    \end{align*}
    This equation implies that \(\rho(gtg^{-1}t^{\mp 1})\in \mathbb{Z} /m\mathbb{Z}\).
    Because \(\lambda(M_i)\) is connected and \(\mathbb{Z}/m\mathbb{Z}\) is discrete, \(gtg^{-1}=t^{\pm 1}\) follows, where the plus-sign arises if and only if \(g\) acts orientation preserving on \(M_i\).
\end{proof}

\section{G-action on M}
\label{sec:Gact}

In this section we consider torus manifolds with \(G\)-action such that \(\tilde{G}\) has only one elementary factor \(G_1\), i.e. \(\tilde{G}=G_1\times T^{l_0}\).
There are two cases:
\begin{enumerate}
\item There is a \(T\)-fixed point, which is not fixed by \(G_1\).
\item There is a \(G\)-fixed point.
\end{enumerate}

We first discuss the case, where there is a \(T\)-fixed point which is not fixed by \(G_1\).

\begin{lemma}
\label{lem:iso}
Let \(\tilde{G}= G_1\times T^{l_0}\) with \(G_1\) elementary, \(\rank G_1=l_1\) and \(M\) a torus manifold with \(G\)-action of dimension \(2n=2(l_0+l_1)\).
If there is an \(x \in M^T\), which is not fixed by the action of \(G_1\), then
\begin{enumerate}
\item\label{item:15} \(G_1=SU(l_1+1)\) or \(G_1=\text{Spin}(2l_1+1)\) and the stabilizer of \(x\) in \(G_1\) is conjugated to \(S(U(l_1)\times U(1))\) or \(\text{Spin}(2l_1)\), respectively.
\item The \(G_1\)-orbit of \(x\) equals the component of \(M^{T^{l_0}}\) which contains \(x\).
\end{enumerate}
Moreover, if \(G_1=SU(4)\), one has \(\#\mathfrak{F}_1 =4\).
\end{lemma}
\begin{proof}
  The \(G_1\)-orbit of \(x\) is contained in the component \(N\) of \(M^{T^{l_0}}\) containing \(x\). Therefore we have
  \begin{equation*}
    \codim G_{1x} = \dim G_1/G_{1x} = \dim G_1x \leq \dim  N \leq 2l_1.
  \end{equation*}
Furthermore the stabilizer \(G_{1x}\) of \(x\) has maximal rank \(l_1\).
In particular, its identity component \(G_{1x}^0\) is a closed connected maximal rank subgroup.

Next we use the theory of Lie-groups to determine the isomorphism types of \(G_1\) and \(G_{1x}\).
At first we consider the case \(G_1\neq \text{Spin}(4)\).
From the classification of closed connected maximal rank subgroups of a compact Lie-group given in \cite[p. 219]{0034.30701}  we get the following connected maximal rank subgroups \(H\) of maximal dimension:
\begin{center}
    \begin{tabular}{|c|c|c|}
      \(G_1\)                              &  \(H\)                                          & \(\codim H\)   \\\hline\hline
      \(SU(2)=\text{Spin}(3)=Sp(1)\)       &  \(S(U(1)\times U(1))\)                         & \(2\)          \\\hline
      \(\text{Spin}(5)=Sp(2)\)             &  \(\text{Spin}(4)\)                             & \(4\) \\\hline
      \(SU(4)=\text{Spin}(6)\)             & \(S(U(3)\times U(1))\)                          & \(6\)          \\\hline
      \(SU(l_1+1)\), \(l_1 \neq 1,3\)      & \(S(U(l_1)\times U(1))\)                        & \(2l_1\)       \\\hline
      \(\text{Spin}(2l_1+1)\), \(l_1 > 2\) & \(\text{Spin}(2l_1)\)                           & \(2l_1\)        \\\hline
      \(\text{Spin}(2l_1)\), \(l_1 > 3\)   & \(\text{Spin}(2l_1-2)\times \text{Spin}(2)\)    & \(4l_1-4\)     \\\hline
      \(Sp(l_1)\), \(l_1 > 2\)             &\(Sp(l_1-1)\times Sp(1)\)                        & \(4l_1-4\) 
  \end{tabular}
\end{center}

Because \(H\) is unique up to conjugation and 
\begin{equation*}
 \codim H \leq \codim G_{1x}^0 =\codim G_{1x}\leq 2l_1, 
\end{equation*}
we see \(G_1=SU(l_1+1)\) or \(G_1=\text{Spin}(2l_1+1)\). Moreover, \(G_{1x}\) is conjugated to a subgroup of \(G_1\) which contains \(S(U(l_1)\times U(1))\) or \(\text{Spin}(2l_1)\), respectively.

If \(l_1>1\), then  \(S(U(l_1)\times U(1))\) is a maximal subgroup  of \(SU(l_1+1)\) by Lemma \ref{sec:lie-algebra-calc}.
Therefore, if \(G_1=SU(l_1+1)\) and \(l_1>1\), then \(G_{1x}\) is conjugated to  \(S(U(l_1)\times U(1))\).
Because \(\codim S(U(l_1)\times U(1)) = 2l_1 \geq \dim N \geq \codim G_{1x}\), we have \(G_1x=N\) in this case.

If \(G_1=\text{Spin}(2l_1+1)\), \(l_1\geq 1\), then by Lemma~\ref{sec:lie-groups-1} there are two proper subgroups of \(G_1\), which contain \(\text{Spin}(2l_1)\); \(\text{Spin}(2l_1)\) and its normalizer  \(H_0\).
Because of dimension reasons we have \(N=G_1x\).
Because \(\text{Spin}(2l_1+1)/H_0\) is not orientable and \(M^{T^{l_0}}\) is orientable, \(G_{1x}=\text{Spin}(2l_1)\) follows.
The case \(G_1=SU(2)\) is included in the discussion in this paragraph because \(SU(2)=\text{Spin}(3)\).

Now we prove the last statement of the lemma.
If \(G_1=SU(4)\), then \(G_1x\) is \(G_1\)-equivariantly diffeomorphic to \(\mathbb{C}P^3\) by the above discussion.
Because \(\C P^3\) has four characteristic submanifolds with pairwise non-trivial intersections, by Lemmas~\ref{lem:torus-m1} and~\ref{sec:gener-torus-manif}, there are four characteristic submanifolds \(M_1,\dots,M_4\), which intersect transversely with \(G_1x=N\).
Because \(G_1x\) is a component of \(M^{T^{l_0}}\) we have by Lemma~\ref{lem:torus-m2} that \(\lambda(M_i)\not\subset T^{l_0}\).
Therefore \(\lambda(M_i)\) is not fixed pointwise by the action of \(W(G_1)\) on \(T\).
Here \(W(G_1)\) acts on \(T\) by conjugation.
Now it follows with Lemma~\ref{sec:action-weyl-group} that \(M_1,\dots,M_4\) belong to \(\mathfrak{F}_1\).

Now we turn to the case \(G_1=\text{Spin}(4)=SU(2)\times SU(2)\).

Then there are the following proper closed connected maximal rank subgroups \(H\) of \(G_1\) of codimension at most \(4\):
\begin{equation*}
  SU(2)\times S(U(1)\times U(1)),\; S(U(1)\times U(1))\times SU(2),\;  S(U(1)\times U(1)) \times S(U(1)\times U(1)).
\end{equation*}

The last has codimension four in \(G_1\). 
The others have codimension two in \(G_1\).

At first assume that \(G_1x\) has dimension four.
Then we have \(G^0_{1x}= S(U(1)\times U(1)) \times S(U(1)\times U(1))\).
There are five proper subgroups of \(\text{Spin}(4)\) which contain \(S(U(1)\times U(1)) \times S(U(1)\times U(1))\) as a maximal connected subgroup, namely:
\begin{align*}
  H_1'&=S(U(1)\times U(1)) \times S(U(1)\times U(1))\\
  H_2'&=N_{SU(2)}S(U(1)\times U(1)) \times S(U(1)\times U(1))\\
  H_3'&= S(U(1)\times U(1)) \times N_{SU(2)}S(U(1)\times U(1))\\
  H_4'&=N_{SU(2)}S(U(1)\times U(1)) \times N_{SU(2)}S(U(1)\times U(1))\\
  H_5'&=\{(g_1,g_2)\in N_{SU(2)}S(U(1)\times U(1)) \times N_{SU(2)}S(U(1)\times U(1));\\
  &\quad\quad g_1\in S(U(1)\times U(1))\Leftrightarrow g_2\in S(U(1)\times U(1))\} 
\end{align*}
Therefore \(G_1x\) is \(G_1\)-equivariantly diffeomorphic to one of the following spaces:
\begin{align*}
  \text{Spin}(4)/H_1'&=S^2\times S^2,\\ 
  \text{Spin}(4)/H_5'&=S^2\times_{\mathbb{Z}_2}S^2=\text{orientable double cover of } \mathbb{R}P^2\times \mathbb{R}P^2,\\
  \text{Spin}(4)/H_2'&=\mathbb{R}P^2\times S^2,\\
  \text{Spin}(4)/H_3'&=S^2 \times \mathbb{R}P^2,\\
  \text{Spin}(4)/H_4'&=\mathbb{R}P^2\times \mathbb{R}P^2.
\end{align*}

Since \(G_1x=M^{T^{l_0}}\) is orientable, the latter three do not occur.

For \(N= G_1x=S^2\times S^2,S^2\times_{\mathbb{Z}_2}S^2\), let \(N^{(1)}\) be the union of the \(T\)-orbits in \(N\) of dimension less than or equal to one.
Then \(W(G_1)=\mathbb{Z}_2\times \mathbb{Z}_2\) acts on the orbit space \(N^{(1)}/T\).
This space is given by one of the following graphs:
\begin{equation*}
  \begin{array}{cc}
  \xymatrix{
    \circ \ar@{-}[rr]\ar@{-}[dd]& \ar@{<-}`u[r]`/0pt[r]`r/10pt[dd]`/0pt[dd]^{w_2}`d[dd]
 & \circ \ar@{-}[dd] &\\
    \ar@{<->}[rr]^{w_1} & &\\
    \circ \ar@{-}[rr]& \ar@{<-}`d/10pt[r]`/0pt[r]`r/10pt[ur]`/0pt[ur] & \circ\\
} 
   &\xyoption{curve}
  \xymatrix{
    & \circ\ar@{-}`r/20pt[dr]`/0pt[dr]\ar@{-}`l/20pt[dl]`/0pt[dl] &\\
    \ar@(ul,dl)@{<->}[]_{w_1}&       &\ar@(ur,dr)@{<->}[]^{w_2}\\
    & \circ\ar@{-}`r/20pt[ur]`/0pt[ur]\ar@{-}`l/20pt[ul]`/0pt[ul] &\\
}\\
\\
(S^2\times S^2)^{(1)}/T \quad\quad\quad&(S^2\times_{\mathbb{Z}_2}S^2)^{(1)}/T
  \end{array}
\end{equation*}
Where the edges correspond to orbits of dimension one and the vertices to the fixed points.
The arrows indicate the action of the generators \(w_1,w_2\in W(G_1)\) on this space.
Let \(M_1,M_2\) be the two characteristic submanifolds of \(M\) which intersect transversely with \(N\) in \(x\).
Because \(N\) is a component of \(M^{T^{l_0}}\), \(\lambda(M_i)\), \(i=1,2\), is not a subgroup of \(T^{l_0}\) by Lemma~\ref{lem:torus-m2}.
Therefore \(\lambda(M_i)\) is not fixed pointwise by \(W(G_1)\).
By Lemma~\ref{sec:action-weyl-group}, this fact implies \(M_1,M_2\in \mathfrak{F}_1\).
Therefore there is a \(w \in W(G_1)\) with \(w(M_1)=M_2\).
But from the pictures above we see that \(M_1\) and \(M_2\) are not in the same \(W(G_1)\)-orbits.
Therefore the case \(\dim G_1x=4\) does not occur.

Now assume that \(G_1x\) has dimension two.
Then we may assume without loss of generality that \(G_{1x}^0= SU(2) \times S(U(1)\times U(1))\).
Therefore  \(G_1x \subset M^{SU(2)\times 1}\).
Because \(G_1x\subset M^{T^{l_0}}\), \(G_1x\) is a component of \(M^{S(U(1)\times U(1))\times{1}\times T^{l_0}}\) in this case.
Therefore, by Lemmas \ref{lem:torus-m2} and \ref{sec:gener-torus-manif}, there are characteristic submanifolds \(M_2,\dots,M_{l_0+2}\) of \(M\) such that \(G_1x\) is a component of \(\bigcap_{i=2}^{l_0+2}M_i\).
Furthermore, we may assume that \(\lambda(M_2)\not\subset T^{l_0}\). 
Therefore,  by Lemma \ref{sec:action-weyl-group}, we have \(M_2\in \mathfrak{F}_1\).

But there is also a characteristic submanifold \(M_1\) of \(M\) which intersects \(G_1x\) transversely in \(x\).
With the Lemmas \ref{lem:torus-m2} and \ref{sec:action-weyl-group}, we see \(M_1\in \mathfrak{F}_1\).

Therefore there is a \(w\in W(G_1)\) with \(w(M_2)=M_1\).
But this is impossible because \(M_2\supset G_1x\not\subset M_1\).

Therefore \(G_1\neq \text{Spin}(4)\) and the lemma is proved. 
\end{proof}

\begin{remark}
\label{sec:g-action-m-1}
If, in the situation of Lemma~\ref{lem:iso}, \(T\cap G_1\) is the standard maximal torus of \(G_1\), then it follows
  by Proposition 2 of \cite[p. 325]{0712.57010} that \(G_{1x}\) is conjugated to the groups given in Lemma~\ref{lem:iso} (\ref{item:15})  by an element of the normalizer of the maximal torus.
\end{remark}

\begin{lemma}
\label{sec:g-action-m}
  In the situation of the previous lemma \(x\) is contained in the intersection of exactly \(l_1\) characteristic submanifolds belonging to \(\mathfrak{F}_1\).
\end{lemma}
\begin{proof}
  Because \(N=G_1x\) has dimension \(2l_1\), \(x\) is contained in exactly \(l_1\) characteristic submanifolds of \(N\).
  By Lemmas \ref{lem:torus-m1} and \ref{sec:gener-torus-manif}, we know that they are components of intersections of characteristic submanifolds \(M_1,\dots,M_{l_1}\) of \(M\) with \(N\).

  Because \(G_1x\) is a component of \(M^{T^{l_0}}\), \(\lambda(M_i)\) is not a subgroup of \(T^{l_0}\) for \(i=1,\dots,l_1\) by Lemmas \ref{lem:torus-m2} and~\ref{sec:gener-torus-manif}.
  Therefore \(\lambda(M_i)\) is not fixed pointwise by \(W(G_1)\).
  By Lemma~\ref{sec:action-weyl-group}, this implies that \(M_i\) belongs to \(\mathfrak{F}_1\).

  By Lemmas~\ref{sec:gener-torus-manif} and~\ref{lem:torus-m2}, \(G_1x\) is the intersection of \(l_0\) characteristic submanifolds \(M_{l_1+1},\dots,M_n\) of \(M\).
  We show that these manifolds do not belong to \(\mathfrak{F}_1\).
  Assume that there is an \(i\geq l_1+1\) such that \(M_i\) belongs to \(\mathfrak{F}_1\).
  Because \(W(G_1)\) acts transitively on \(\mathfrak{F}_1\), there is a \(w\in W(G_1)\) with \(w(M_i)=M_j\), \(j\leq l_1\).
  But this is impossible because \(M_i\supset G_1x\not\subset M_j\).
\end{proof}

Now we turn to the case, where there is  a \(T\)-fixed point which is fixed by \(G_1\).

\begin{lemma}
  \label{lem:iso2}
  Let \(\tilde{G}= G_1\times T^{l_0}\) with \(G_1\) elementary, \(\rank G_1=l_1\) and \(M\) a torus manifold with \(G\)-action of dimension \(2n=2(l_0+l_1)\).
If there is a \(T\)-fixed point \(x \in M^T\), which is fixed by \(G_1\), then \(G_1=SU(l_1+1)\) or \(G_1=\text{Spin}(2l_1)\).

Moreover, if \(G_1\neq \text{Spin}(8)\) one has
\begin{gather}
 \label{eq:9} T_xM=V_1\oplus V_2 \otimes_\C W_1 \text{ if } G_1=SU(l_1+1) \text{ and } \#\mathfrak{F}_1=4 \text{ in the case } l_1=3,\\
 \label{eq:10} T_xM=V_3\oplus W_2             \text{ if } G_1=\text{Spin}(2l_1) \text{ and } \#\mathfrak{F}_1=3 \text{ in the case } l_1=3,
\end{gather}
where \(W_1\) is the standard complex representation of \(SU(l_1+1)\) or its dual, \(W_2\) is the standard real representation of \(SO(2l_1)\) and the \(V_i\) are complex  \(T^{l_0}\)-representations.

In the case \(G_1=\text{Spin}(8)\), one may change the action of \(G_1\) on \(M\) by an automorphism of \(G_1\), which is independent of \(x\), to reach the situation described in (\ref{eq:10}).

Furthermore we have \(x \in\bigcap_{M_i\in \mathfrak{F}_1}M_i\).
If \(l_1=1\), then we have \(\#\mathfrak{F}_1=2\).
\end{lemma}
\begin{proof}

Let \(M_1,\dots,M_n\) be the characteristic submanifolds of \(M\), which intersect in \(x\). Then the weight spaces of the \(\tilde{G}\)-representation \(T_xM\) are given by
\begin{equation*}
  N_x(M_1,M),\dots,N_x(M_n,M).
\end{equation*}
For \(g \in N_GT\) we have \(M_i=gM_j\) if and only if \(N_x(M_i,M)=gN_x(M_j,M)\).
Because \(G_1\) acts non-trivially on \(T_xM\), there is at least one \(M_i\), \(i\in\{1,\dots,n\}\), such that \(M_i\in\mathfrak{F}_1\).

In the following a weight space of \(T_xM\) together with a choice of an orientation for this weight space is called an oriented weight space of \(T_xM\).
The action of \(G_1\) on \(T_xM\) induces an action of \(W(G_1)\) on the set of oriented weight spaces of \(T_xM\).

Because \(W(G_1)\) acts transitively on \(\mathfrak{F}_1\) and \(x\) is a \(G\)-fixed point, we have 
\begin{equation}
  \label{eq:ff}
  \frac{1}{2}\#\{\text{oriented weight spaces of } T_xM \text{ which are not fixed by } W(G_1)\}=\#\mathfrak{F}_1
\end{equation}
and 
\(x \in\bigcap_{M_i\in \mathfrak{F}_1}M_i\).

For the  \(\tilde{G}\)-representation \(T_xM\) we have
  \begin{equation}
    \label{eq:T_x}
    T_xM = N_x(M^{T^{l_0}},M) \oplus T_xM^{T^{l_0}}.
  \end{equation}
If \(l_0=0\), then we have \(N_x(M^{T^{l_0}},M)=\{0\}\).
Otherwise the action of \(T^{l_0}\) induces a complex structure on \(N_x(M^{T^{l_0}},M)\).
By \cite[p. 68]{0581.22009} and \cite[p. 82]{0581.22009}, we have
\begin{equation}
\label{eq:1}
   N_x(M^{T^{l_0}},M) = \bigoplus_i V_i\otimes_\C W_i,
\end{equation}
where the \(V_{i}\)  are one-dimensional complex \(T^{l_0}\)-representations and the \(W_{i}\) are irreducible complex \(G_1\)-representations.
Since \(T^{l_0}\) acts almost effectively on \(M\), there are at least \(n-l_1\) summands in this decomposition.
Therefore we get
\begin{equation}
\label{eq:11}  
 \dim_{\mathbb{C}} W_{i} = \dim_{\mathbb{C}} N_x(M^{T^{l_0}},M) - \sum_{j\neq i}\dim_{\mathbb{C}} V_j\otimes_\C W_j
   \leq n - (n-l_1-1) = l_1+1.
\end{equation}
Furthermore
\begin{equation}
  \label{eq:12}
   \dim_{\mathbb{R}} T_xM^{T^{l_0}} \leq 2(n-l_0)=2l_1.
\end{equation}
 
If there is a \(W_{i_0}\) with \(\dim_{\mathbb{C}} W_{i_0}=l_1+1\), then from equation (\ref{eq:1}) we get, for all other \(W_i\),
\begin{equation}
  \label{eq:13}
   \dim_{\mathbb{C}} W_{i} = \dim_{\mathbb{C}}  N_x(M^{T^{l_0}},M) - \dim_{\mathbb{C}} V_{i_0}\otimes_\C W_{i_0} -\sum_{j\neq i,i_0}\dim_{\mathbb{C}} V_j\otimes_\C W_j\leq 1.
\end{equation}
So they are one-dimensional.
Therefore they are trivial.
Furthermore we have
\begin{equation*}
  \dim_{\mathbb{C}} N_x(M^{T^{l_0}},M) = \sum_i \dim_{\mathbb{C}} V_i \otimes_\C W_i \geq n
\end{equation*}
because there are at least \(n-l_1\) summands in the decomposition~(\ref{eq:1}).
Therefore \(T_xM^{T^{l_0}}\) is zero-dimensional in this case.

If \(\dim_{\mathbb{R}} T_xM^{T^{l_0}}=2l_1\), then we have
\begin{equation*}
   \dim_{\mathbb{C}} W_{i} = \dim_{\mathbb{C}} N_x(M^{T^{l_0}},M) - \sum_{j\neq i}\dim_{\mathbb{C}} V_j\otimes_\C W_j\leq 1.
\end{equation*}
Therefore all \(W_i\) are one dimensional.
So they are trivial in this case.

There are the following lower bounds \(d_\mathbb{R},d_\mathbb{C}\) for the dimension of real and complex non-trivial irreducible  representations of \(G_1\) \cite[p. 53-54]{0796.57001}:
\begin{center}
    \begin{tabular}{|c|c|c|}
      \(G_1\)                                  &  \(d_\mathbb{R}\) & \(d_\mathbb{C}\)   \\\hline\hline
      \(SU(2)=\text{Spin}(3)=Sp(1)\)           &  \(3\)           & \(2\)              \\\hline
      \(\text{Spin}(4)\)                       &  \(3\)           & \(2\)              \\\hline
      \(\text{Spin}(5)=Sp(2)\)                 &  \(5\)           & \(4\)              \\\hline
      \(SU(4)=\text{Spin}(6)\)                 & \(6\)            & \(4\)              \\\hline
      \(SU(l_1+1)\), \(l_1 \neq 1,3\)          & \(2l_1+2\)       & \(l_1+1\)         \\\hline
      \(\text{Spin}(2l_1+1)\), \(l_1 >2\)      & \(2l_1+1\)       & \(2l_1+1\)         \\\hline
      \(\text{Spin}(2l_1)\), \(l_1 > 3\)       & \(2l_1\)         & \(2l_1\)           \\\hline
      \(Sp(l_1)\), \(l_1 >2\)                  &\(2l_1+1\)        & \(2l_1\)  
  \end{tabular}  
\end{center}

In \cite[p. 53-54]{0796.57001} the dominant weights of the \(G_1\)-representations realising these bounds are also given.
They are important in the discussion below.

Because \(G_1\) acts non-trivially on \(T_xM\), one of the \(W_i\)'s or \(T_xM^{T^{l_0}}\) is a non-trivial \(G_1\)-representation.
Therefore we have \(d_{\mathbb{R}}\leq 2l_1\) or \(d_{\mathbb{C}}\leq l_1+1\) by (\ref{eq:11}) and (\ref{eq:12}).
Therefore \(G_1 \neq Sp(l_1)\), \(l_1>1\), and \(G_1 \neq \text{Spin}(2l_1+1)\), \(l_1>1\).

If \(G_1=\text{Spin}(2l_1),l_1>3\), then all \(W_i\) are trivial because
\begin{equation*}
  \dim_{\mathbb{C}}W_i\leq l_1+1<2l_1=d_\mathbb{C}.
\end{equation*}
Moreover, \(T_xM^{T^{l_0}}\) has dimension \(2l_1\).
Therefore it is the standard real \(SO(2l_1)\)-representation if \(l_1>4\).
If \(l_1=4\), then there are three eight-dimensional real representations of \(\text{Spin}(8)\), namely the standard real \(SO(8)\)-representation and the two half spinor representations.
They have three different kernels.
Notice that the kernel of the \(G_1\)-representation \(T_xM^{T^{l_0}}\) is equal to the kernel of the \(G_1\)-action on \(M\).
Therefore, if one of them is isomorphic to \(T_xM^{T^{l_0}}\), then it is isomorphic to  \(T_yM^{T^{l_0}}\) for all \(y\in M^{T}\).
So we may -- after changing the action of \(\text{Spin}(8)\) on \(M\) by an automorphism -- assume that \(T_xM^{T^{l_0}}\) is the standard real \(SO(8)\)-representation.

If \(G_1=SU(l_1+1), l_1\neq 1,3\), then only one \(W_i\) is non-trivial and \(T_xM^{T^{l_0}}\) has dimension zero.
The non-trivial \(W_i\) is the standard representation of \(SU(l_1+1)\) or its dual depending on the complex structure of \(N_x(M^{T^{l_0}},M)\).

If  \(G_1=SU(4)\), then there are one real representation of dimension \(6\) and two complex representations of dimension \(4\).
If the first representation occurs in the decomposition of  \(T_xM\), then, by (\ref{eq:ff}), we have \(\#\mathfrak{F}_1=3\). 
If one of the others occurs, then  \(\#\mathfrak{F}_1=4\).

If \(G_1=SU(2)\), then there is one non-trivial \(W_i\) of dimension 2. Therefore, by (\ref{eq:ff}), one has \(\#\mathfrak{F}_1=2\).

If \(G_1=\text{Spin}(4)\), then \(T_xM\) is an almost faithful representation.
Because all almost faithful complex representations of \(\text{Spin}(4)\) have at least dimension four there is no \(W_i\) of dimension three.

If there is one \(W_{i_0}\) of dimension two, then we see as in (\ref{eq:13}) that all other \(W_i\) and \(T_xM^{T^{l_0}}\) have dimension less than or equal to two.
Because there is no non-trivial two-dimensional real \(\text{Spin}(4)\)-representation there is another \(W_i\) of dimension two.
Therefore there are eight oriented weight spaces of \(T_xM\) which are not fixed by the action on \(W(G_1)\).
But this contradicts (\ref{eq:ff}) because \(\#\mathfrak{F}_1=2\).

Therefore all \(W_i\) are one-dimensional.
Hence, they are trivial.
\(T_xM^{T^{l_0}}\) has to be the standard four-dimensional real representation of \(\text{Spin}(4)\).
\end{proof}

With the Lemmas \ref{lem:iso} and \ref{lem:iso2}, we see that there is no elementary factor of \(\tilde{G}\), which is isomorphic to \(Sp(l_1)\) for \(l_1>2\).

Now let \(G_1=\text{Spin}(2l)\).
 If \(l=3\), we assume \(\# \mathfrak{F}_1=3\).
 Then, by looking at the \(G_1\)-representation \(T_xM\), one sees with Lemma~\ref{lem:iso2} that the \(G_1\)-action factors through \(SO(2l)\).

Now let \(G_1=\text{Spin}(2l+1)\), \(l>1\).
Then, by Lemma~\ref{lem:iso}, we have \(G_{1x}\cong\text{Spin}(2l)\).
Because the \(G_{1x}\)-action on \(N_x(G_1x,M)\) is trivial by Lemma~\ref{lem:iso2},
 the \(G_1\)-action factors through \(SO(2l+1)\).

In the case \(G_1=\text{Spin}(3)\) and \(\# \mathfrak{F}_1=1\) we have \(G_1x=S^2\).
The characteristic submanifold \(M_1\in \mathfrak{F}_1\) intersects \(G_1x\) transversely in \(x\).
Because  \(\# \mathfrak{F}_1=1\), \(\lambda(M_1)\)
is invariant under the action of \(W(G_1)\) on the maximal torus of \(G\).
Because, by Lemma~\ref{sec:action-weyl-group}, the non-trivial element of \(W(G_1)\) reverses the orientation of  \(\lambda(M_1)\),
it is a maximal torus of \(G_1\).
Therefore the center of \(G_1\) acts trivially on \(M\).
Hence, the \(G_1\)-action on \(M\) factors through \(SO(3)\).

If, in the case \(G_1=\text{Spin}(3)\) and \(\# \mathfrak{F}_1=2\), the principal orbit type of the \(G_1\)-action is given by \(\text{Spin}(3)/\text{Spin}(2)\), then the \(G_1\)-action factors through \(SO(3)\).

Therefore in the following we may replace an elementary factor \(G_i\) of \(\tilde{G}\) isomorphic to \(\text{Spin}(l)\), which satisfies the above conditions, by \(SO(l)\).

\begin{convention}
\label{sec:g-action-m-3}
  If we say that an elementary factor \(G_i\) is isomorphic to \(SU(2)\) or \(SU(4)\), then we mean that \(\#\mathfrak{F}_i=2\) or \(\#\mathfrak{F}_i=4\), respectively.
Conversely, if we say that \(G_i\) is isomorphic to \(SO(3)\) we mean that \(\#\mathfrak{F}_i=1\) or \(\#\mathfrak{F}_i=2\) and the \(SO(3)\)-action has principal orbit type \(SO(3)/SO(2)\).
If we say \(G_i=SO(6)\), then we mean \(\#\mathfrak{F}_i=3\).
\end{convention}

\begin{cor}
\label{sec:g-action-m-2}
 Assume that \(G\) is elementary.
  Then \(M\) is equivariantly diffeomorphic to \(\C P^{l_1}\) or \(M=S^{2l_1}\) if \(\tilde{G}=SU(l_1+1)\) or \(\tilde{G}=SO(2l_1+1),SO(2l_1)\), respectively.
\end{cor}
\begin{proof}
  If \(G\) is elementary, then we may assume that \(G=\tilde{G}= SO(2l_1),SO(2l_1 +1),SU(l_1+1)\) and \(\dim M= 2l_1\).
  
  If \(G=SO(2l_1)\), then, by Lemmas  \ref{lem:iso} and \ref{lem:iso2}, the principal orbit type  of the \(SO(2l_1)\)-action is given by \(SO(2l_1)/SO(2l_1-1)\), which has codimension one in \(M\).

  The group \(S(O(2l_1-1)\times O(1))\) is the only proper subgroup of \(SO(2l_1)\), which contains \(SO(2l_1-1)\) properly.
  Because \(SO(2l_1)/S(O(2l_1-1)\times O(1))=\mathbb{R}P^{2l_1-1}\) is orientable all orbits of the \(SO(2l_1)\)-action are of types \(SO(2l_1)/SO(2l_1 -1)\) or \(SO(2l_1)/SO(2l_1)\)  by \cite[p. 185]{0246.57017}.

By \cite[p. 206-207]{0246.57017}, we have
\begin{equation*}
  M= D^{2l_1}_1\cup_\phi D^{2l_1}_2,
\end{equation*}
where \(SO(2l_1)\) acts on the disks \(D_i^{2l_1}\) in the usual way and
\begin{equation*}
  \phi: S^{2l_1-1}=SO(2l_1)/SO(2l_1-1) \rightarrow  S^{2l_1-1}=SO(2l_1)/SO(2l_1-1)
\end{equation*}
  is given by
\(gSO(2l_1-1)\mapsto gnSO(2l_1-1)\), where \(n \in N_{SO(2l_1)}SO(2l_1-1)=S(O(2l_1-1)\times O(1))\).

Therefore \(\phi=\pm \id_{S^{2l_1-1}}\) and \(M=S^{2l_1}\).

  If \(G=SO(2l_1+1)\), then 
  \begin{equation*}
    M=SO(2l_1+1)/SO(2l_1)=S^{2l_1}
  \end{equation*}
  follows directly from Lemmas \ref{lem:iso} and \ref{lem:iso2}.

  Now assume \(G=SU(l_1+1)\).
  Because \(\dim M=2l_1\), the intersection of \(l_1+1\) pairwise distinct characteristic submanifolds of \(M\) is empty.
  By Lemma \ref{lem:iso2}, no \(T\)-fixed point is fixed by \(G\).
  Therefore from Lemma \ref{lem:iso} we get
  \begin{equation*}
    M=SU(l_1+1)/S(U(l_1)\times U(1))=\C P^{l_1}.
  \end{equation*}
\end{proof}
\begin{remark}
  Another proof of this statement follows from the classification given in section \ref{sec:classif}.
\end{remark}

\section{Blowing up}
\label{sec:blow}

In this section we describe blow ups of torus manifolds with \(G\)-action.
They are used in the following sections to construct from a torus manifold \(M\) with \(G\)-action another torus manifold \(\tilde{M}\) with \(G\)-action, such that an elementary factor of the covering group \(\tilde{G}\) of \(G\) has no fixed point in \(\tilde{M}\).

References for this construction are \cite[p. 602-611]{0408.14001} and \cite[p. 269-270]{0567.53031}.

As before we write \(\tilde{G}=\prod_{i=1}^k G_i \times T^{l_0}\) with \(G_i\) elementary and \(T^{l_0}\) a torus.

We will see in sections \ref{sec:su} and \ref{sec:so} that there are the following two cases:
\begin{enumerate}
\item A component \(N\) of \(M^{G_1}\) has odd codimension in \(M\).
\item A component \(N\) of \(M^{G_1}\) has even codimension in \(M\) and there is a \(g \in Z(\tilde{G})\) such that \(g\) acts trivially on \(N\) and \(g^2\) acts as \(-\id\) on \(N(N,M)\). 
\end{enumerate}

In the second case the action of \(g\) on \(N(N,M)\) induces a \(G\)-invariant complex structure.
We equip \(N(N,M)\) with this structure.
Let \(E=N(N,M) \oplus\mathbb{K}\), where \(\mathbb{K}= \mathbb{R}\) in the first case and \(\mathbb{K}=\mathbb{C}\) in the second case.

In the following we call case (1) the real case and case (2) the complex case.

\begin{Lemma}
  The projectivication \(P_{\mathbb{K}}(E)\) is orientable.
\end{Lemma}
\begin{proof}
Because \(M\) is orientable the total space of the normal bundle of \(N\) in \(M\) is orientable.
Therefore
\begin{equation*}
  E=N(N,M)\oplus \mathbb{K}=N(N,M)\times \mathbb{K}
\end{equation*}
and the associated sphere bundle \(S(E)\) are orientable.

Let \(Z_{\mathbb{K}}=\mathbb{Z}/2\mathbb{Z}\) if \(\mathbb{K}=\mathbb{R}\) and \(Z_{\mathbb{K}}=S^1\) if \(\mathbb{K}=\mathbb{C}\).
Then \(Z_{\mathbb{K}}\) acts on \(E\) and \(S(E)\) by multiplication on the fibers.
Now \(P_{\mathbb{K}}(E)\) is given by \(S(E)/Z_{\mathbb{K}}\).
If  \(\mathbb{K}=\mathbb{C}\), then \(Z_\mathbb{K}\) is connected.
Therefore it acts orientation preserving on \(S(E)\).

If \(\mathbb{K}=\mathbb{R}\), then  \(\dim E\) is even.
Therefore the restriction of the \(Z_{\mathbb{K}}\)-action to a fiber of \(E\) is orientation preserving.
Hence, it preserves the orientation of \(S(E)\).

Because the action of \(Z_{\mathbb{K}}\)  is orientation preserving on \(S(E)\), \(P_{\mathbb{K}}(E)\) is orientable.
\end{proof}

Choose a \(G\)-invariant Riemannian metric on \(N(N,M)\) and
 a \(G\)-equivariant closed tubular neighborhood \(B\) around \(N\). Then one may identify
\begin{equation*}
  B= \{z_0\in N(N,M); |z_0| \leq 1\} = \{(z_0:1)\in P_\mathbb{K}(E); |z_0| \leq 1\}.
\end{equation*}

By gluing the complements of the interior of \(B\) in M and \(P_{\mathbb{K}}(E)\) along the boundary of \(B\), we get a new torus manifold with \(G\)-action \(\tilde{M}\),  the \emph{blow up} of \(M\) along \(N\).
It is easy to see, using isotopies of tubular neighborhoods, that the \(G\)-equivariant diffeomorphism-type of \(\tilde{M}\) does not depend on the choices of the Riemannian metric and the tubular neighborhood. 

\(\tilde{M}\) is oriented in such a way that the induced orientation on \(M-\mathring{B}\) coincides with the orientation induced from \(M\).
This forces the inclusion of \(P_\mathbb{K}(E)-\mathring{B}\) to be orientation reversing.
Because \(G_1\) is elementary there is no one-dimensional \(G_1\)-invariant subbundle of \(N(N,M)\).
Therefore we have
\(\#\pi_0(\tilde{M}^{G_1})= \#\pi_0(M^{G_1})-1\).

So by iterating this process over all components of \(M^{G_1}\) one ends up at a torus manifold \(\tilde{M}'\) with \(G\)-action without \(G_1\)-fixed points.
In the following we will call \(\tilde{M}'\) the blow up of \(M\) along \(M^{G_1}\).

\begin{Lemma}
  \label{lem:blow1}
  There is a \(G\)-equivariant map \(F:\tilde{M}\rightarrow M\) which maps the exceptional submanifold \(M_0=P_{\mathbb{K}}(N(N,M)\oplus \{0\})\) to \(N\) and is the identity on \(M-B\).
  Moreover, \(F\) restricts to a diffeomorphism \(\tilde{M}-M_0\rightarrow M-N\).
  Its restriction to \(M_0\) is the bundle projection \(P_{\mathbb{K}}(N(N,M)\oplus\{0\})\rightarrow N\).
\end{Lemma}
\begin{proof}
  The \(G\)-equivariant map
  \begin{equation*}
    f: P_\mathbb{K}(E)-\mathring{B} \rightarrow B \;\; (z_0:z_1) \mapsto (z_0\bar{z}_1:|z_0|^2) \;\; (z_0 \in N(N,M),z_1\in \mathbb{K})
  \end{equation*}
is the identity on \(\partial B\).
Therefore it may be extended to a continuous map \(h: \tilde{M}\rightarrow M\), which is the identity outside of \(P_{\mathbb{K}}(E)-\mathring{B}\).

Because \(f|_{P_\mathbb{K}(E)-\mathring{B}-M_0}: P_\mathbb{K}(E)-\mathring{B}-M_0 \rightarrow B-N\) is a diffeomorphism there is a \(G\)-equivariant diffeomorphism \(F':\tilde{M}-M_0\rightarrow M-N\), which is the identity outside \(P_\mathbb{K}(E)-\mathring{B}-M_0\) and coincides with \(f\) near \(M_0\) by \cite[p. 24-25]{pre05136053}.
Therefore \(F'\) extends to a differentiable map \(F:\tilde{M}\rightarrow M\) such that \(F|_{M_0}=f|_{M_0}\) is the bundle projection.
\end{proof}

\begin{Lemma}
\label{lem:proper}
  Let \(H\) be a closed subgroup of \(G\). Then there is  a bijection
  \begin{equation*}
    \{\text{components of } M^H \not\subset N\} \rightarrow \{\text{components of } \tilde{M}^H \not\subset M_0\}
\end{equation*}
such that
\begin{equation*}
      N' \mapsto \tilde{N}'=\left(P_{\mathbb{K}}(N(N\cap N',N')\oplus \mathbb{K})-\mathring{B}\right)
      \cup_{\partial B \cap N'} \left(N'-\mathring{B}\right)
\end{equation*}
and its inverse is given by
\begin{equation*}
   F(N'') \mapsfrom N'',
\end{equation*}
   where \(N'\) is a component of \(M^H\) and \(N''\) is one of \(\tilde{M}^H\).
Here \(F(N'')\) is the image of \(N''\) under the map \(F\) defined in Lemma~\ref{lem:blow1}.
For a component \(N'\) of \(M^H\), we call \(\tilde{N}'\) the proper transform of \(N'\).
\end{Lemma}
\begin{proof}
  At first we calculate the fixed point set of the \(H\)-action on \(\tilde{M}\).
  \begin{align*}
    \tilde{M}^H&=\left(\left(P_{\mathbb{K}}(E) - \mathring{B}\right)\cup_{\partial B} \left(M-\mathring{B}\right)\right)^H\\
    &=\left(P_{\mathbb{K}}(E) - \mathring{B}\right)^H\cup_{\partial B^H} \left(M-\mathring{B}\right)^H.
  \end{align*}

Because \(H\) is compact, there are pairwise distinct \(i\)-dimensional non-trivial irreducible \(H\)-re\-pre\-sen\-ta\-tions \(V_{ij}\) and \(H\)-vector bundles \(E_{ij}\) over \(N^H\) such that
\begin{equation*}
  N(N,M)|_{N^H}= N(N,M)|_{N^H}^H \oplus \bigoplus_i \bigoplus_{j} E_{ij},
\end{equation*}
and the \(H\)-representation on each fiber of \(E_{ij}\) is isomorphic to \(\mathbb{K}^{d_{ij}}\otimes_{\mathbb{K}} V_{ij}\), where \(\mathbb{K}^{d_{ij}}\) denotes the trivial \(H\)-representation of dimension \(d_{ij}\).

  Now the \(H\)-fixed points in \(P_{\mathbb{K}}(E)\) are given by
  \begin{align*}
    P_{\mathbb{K}}(E)^H &= P_{\mathbb{K}}(N(N,M)\oplus\mathbb{K})|_{N^H}^H\\
    &= P_{\mathbb{K}}(N(N,M)|_{N^H}^H \oplus \mathbb{K}) \amalg \coprod_{j} P_{\mathbb{K}}(E_{1j}\oplus \{0\}).
  \end{align*}
 Because \(N(N,M)|_{N^H}^H=N(N^H,M^H)\) we get
\begin{align*}
  \tilde{M} ^H&=\left(\left(P_{\mathbb{K}}(N(N^H,M^H)\oplus \mathbb{K}) - \mathring{B}^H\right) \cup_{\partial B^H} \left(M - \mathring{B}\right)^H\right)\\
  &\qquad\qquad\qquad\qquad\qquad\qquad\qquad\qquad  \amalg \coprod_{j} P_{\mathbb{K}}(E_{1j}\oplus \{0\})\\
  &=\coprod_{N'\subset M^H}\tilde{N}' \amalg \coprod_{j} P_{\mathbb{K}}(E_{1j}\oplus \{0\}),
\end{align*}
where \(N'\) runs through the connected components of \(M^H\) which are not contained in \(N\).
Thus the statement follows.
\end{proof}

By replacing \(H\) in Lemma~\ref{lem:proper} by an one-dimensional subtorus of \(T\), we get:
\begin{cor}
\label{sec:blowing-up-1}
  There is a bijection between the characteristic submanifolds of \(M\) and the characteristic submanifolds of \(\tilde{M}\), which are not contained in \(M_0\).
\end{cor}
\begin{proof}
  The only thing, that is to prove here, is that for a characteristic submanifold \(M_i\) of \(M\), \(\tilde{M}_i^T\) is non-empty.
  If \((M_i-N)^T\neq \emptyset\), then this is clear.
  
  If \(p\in (M_i\cap N)^T\), then \(P_\mathbb{K}(N(M_i \cap N,M_i)\oplus\{0\})|_p\) is a \(T\)-invariant submanifold of \(\tilde{M}_i\), which is diffeomorphic to \(\mathbb{C}P^k\) or \(\mathbb{R}P^{2k}\).
 Therefore it contains a \(T\)-fixed point.
\end{proof}

This bijection is compatible with the action of the Weyl-group of \(G\) on the sets of characteristic submanifolds of \(\tilde{M}\) and \(M\).

In the real case the exceptional submanifold \(M_0\) has codimension one in \(\tilde{M}\) and is \(G\)-invariant. Because there is no \(S^1\)-representation of real dimension one, \(M_0\) does not contain a characteristic submanifold of \(\tilde{M}\) in this case.

In the complex case \(M_0\) is \(G\)-invariant and may be a characteristic submanifold of \(\tilde{M}\).

Therefore there is a bijection between the non-trivial orbits of the \(W(G)\)-actions on the sets of characteristic submanifolds of \(M\) and \(\tilde{M}\).
Hence
 we get the same elementary factors for the \(G\)-actions on \(\tilde{M}\) and \(M\).

\begin{cor}
\label{cor:blowing-up}
  Let \(H\) be a closed subgroup of \(G\) and \(N'\) a component of \(M^H\) such that \(N\cap N'\) has codimension one  --in the real case--  or two --in the complex case-- in \(N'\).
  Then \(F\) induces a  \((N_GH)^0\)-equivariant diffeomorphism of \(\tilde{N}'\) and \(N'\).
\end{cor}
\begin{proof}
  Because of the dimension assumption the \((N_GH)^0\)-equivariant map
  \begin{equation*}
     f|_{P_\mathbb{K}(N(N\cap N',N')\oplus \mathbb{K})-\mathring{B}\cap N'} :P_\mathbb{K}(N(N\cap N',N')\oplus \mathbb{K})-\mathring{B}\cap N' \rightarrow B\cap N'
  \end{equation*}  
  from the proof of Lemma \ref{lem:blow1}
  is  a diffeomorphism.
  Because the restriction of \(F\) to \(\tilde{M}-M_0\) is an \(G\)-equivariant diffeomorphism the restriction \(F|_{\tilde{N}'-M_0}:\tilde{N}'-M_0\rightarrow N'-N\) is a  \((N_GH)^0\)-equivariant diffeomorphism.
  Therefore \(F|_{\tilde{N}'}:\tilde{N}'\rightarrow N'\) is a diffeomorphism.
\end{proof}

\begin{Lemma}
  In the complex case let \(\bar{E}=N(N,M)^*\oplus \C\), where \(N(N,M)^*\) is the normal bundle of \(N\) in \(M\) equipped with the dual complex structure.
  Then there is a \(G\)-equivariant diffeomorphism
  \begin{equation*}
    \tilde{M}\rightarrow P_{\mathbb{C}}(\bar{E})-\mathring{B} \cup_{\partial B}M-\mathring{B}.
  \end{equation*}
That means that the diffeomorphism type of \(\tilde{M}\) does not change if we replace the complex structure on \(N(N,M)\) by its dual.
\end{Lemma}
\begin{proof}
  We have \(P_{\mathbb{C}}(E)=E/\sim\) and \(P_{\mathbb{C}}(\bar{E})=E/\sim'\), where
  \begin{align*}
    (z_0,z_1)\sim(z_0',z_1') &\Leftrightarrow \exists t \in \mathbb{C}^* \quad (tz_0,tz_1)=(z_0',z_1'),\\
    (z_0,z_1)\sim'(z_0',z_1') &\Leftrightarrow \exists t \in \mathbb{C}^* \quad (tz_0,\bar{t}z_1)=(z_0',z_1').\\
  \end{align*}
Therefore
\begin{align*}
  E &\rightarrow E&(z_0,z_1)&\mapsto (z_0,\bar{z}_1)
\end{align*}
induces a \(G\)-equivariant diffeomorphism \(P_{\mathbb{C}}(E)-\mathring{B}\rightarrow P_{\mathbb{C}}(\bar{E})-\mathring{B}\) which is the identity on \(\partial B\).
By \cite[p. 24-25]{pre05136053} the result follows.
\end{proof}

\begin{lemma}
  \label{sec:blowing-up}
  If in the complex case \(G_1=SU(l_1+1)\) and \(\codim N=2l_1+2\) or in the real case \(G_1=SO(2l_1+1)\) and \(\codim N=2l_1+1\), then \(F:\tilde{M}\rightarrow M\) induces a homeomorphism \(\bar{F}: \tilde{M}/G_1 \rightarrow M/G_1\).
\end{lemma}
\begin{proof}
  Because \(F|_{\tilde{M}-M_0}:\tilde{M}-M_0\rightarrow M-N\) is a equivariant diffeomorphism and \(\tilde{M}/G_1\),\(M/G_1\) are compact Hausdorff-spaces, 
   the only thing, that has to be checked, is that
  \begin{equation*}
    F|_{P_\mathbb{K}(N(N,M))}: P_\mathbb{K}(N(N,M)) \rightarrow N
  \end{equation*}
  induces a homeomorphism of the orbit spaces.
  But this map is just the bundle map \(P_\mathbb{K}(N(N,M)) \rightarrow N\).
  
  If \(G_1=SU(l_1+1)\), then,  because of dimension reasons \cite[p. 53-54]{0796.57001}, the \(G_1\)-representation on the fibers of \(N(N,M)\) is the standard representation of \(G_1\) or its dual.
  If \(G_1=SO(2l_1+1)\), then, by \cite[p. 53-54]{0796.57001}, the \(G_1\)-representation on the fibers of \(N(N,M)\) is the standard representation of \(G_1\).

  Thus, in both cases the \(G_1\)-action on the fibers of \(P_\mathbb{K}(N(N,M)) \rightarrow N\) is transitive.
  Therefore the statement follows.
\end{proof}

\begin{remark}
  All statements proved above also hold for non-connected groups of the form \(G\times K\) where \(K\) is a finite group and \(G\) is connected if we replace \(N\) by a  \(K\)-invariant union of components of \(M^{G_1}\).
\end{remark}

Now we want to reverse the construction of a blow up.
Let \(A\) be a closed \(G\)-manifold and \(E\rightarrow A\) be a \(G\)-vector bundle such that \(G_1\) acts trivially on \(A\).
If \(E\) is even dimensional, we assume that there is a \(g\in Z(G)\) such that \(g\) acts trivially on \(A\) and \(g^2\) acts on \(E\) as \(-\id\).
In this case we equip \(E\) with the complex structure induced by the action of \(g\).

Assume that \(\tilde{M}\) is a \(G\)-manifold and there is a \(G\)-equivariant embedding of \(P_{\mathbb{K}}(E)\hookrightarrow \tilde{M}\) such that the normal bundle of \(P_\mathbb{K}(E)\) is isomorphic to  the tautological bundle over  \(P_\mathbb{K}(E)\).

Then one may identify a closed \(G\)-equivariant tubular neighborhood \(B^c\) of \(P_{\mathbb{K}}(E)\) in \(\tilde{M}\) with
\begin{equation*}
  B^c=\{(z_0:1)\in P_{\mathbb{K}}(E\oplus \mathbb{K});|z_0|\geq 1\}\cup \{(z_0:0)\in P_{\mathbb{K}}(E\oplus \mathbb{K})\}.
\end{equation*}

By gluing the complements of the interior of \(B^c\) in \(\tilde{M}\) and \(P_{\mathbb{K}}(E\oplus \mathbb{K})\), we get a \(G\)-manifold \(M\) such that \(A\) is \(G\)-equivariantly diffeomorphic to a union  of components of \(M^{G_1}\).

We call \(M\) the \emph{blow down}  of \(\tilde{M}\) along \(P_{\mathbb{K}}(E)\).

It is easy to see that the \(G\)-equivariant diffeomorphism type of \(M\) does  not depend on the choices of a metric on \(E\) and the tubular neighborhood of \(P_{\mathbb{K}}(E)\) in \(\tilde{M}\) if \(G_1\) acts transitively on the fibers of \(P_{\mathbb{K}}(E)\rightarrow A\).

It is also easy to see that the blow up and blow down constructions are inverse to each other.
 
\section{The case $G_1 = SU(l_1+1)$}
\label{sec:su}
In this section we discuss actions of groups, which have a covering group of the form \(G_1\times G_2\), where \(G_1=SU(l_1+1)\) is elementary and \(G_2\) acts effectively on \(M\). 
It turns out that the blow up of \(M\) along \(M^{G_1}\) is a fiber bundle over \(\C P^{l_1}\).
This fact leads to our first classification result. 

The assumption on \(G_2\) is no restriction on \(G\), because one may replace any covering group \(\tilde{G}\) by the quotient \(\tilde{G}/H\) where \(H\) is a finite subgroup of \(G_2\) acting trivially on \(M\).
Following Convention~\ref{sec:g-action-m-3}, we also assume \(\#\mathfrak{F}_1=2\) or \(\#\mathfrak{F}_1=4\) in  the cases \(G_1=SU(2)\) or \(G_1=SU(4)\), respectively.
Furthermore, we assume after conjugating \(T\) with some element of \(G_1\) that \(T_1=T\cap G_1\) is the standard maximal torus of \(G_1\). 

\subsection{The $G_1$-action on $M$}
\label{sec:su:G_1onM}
We have the following lemma:
\begin{lemma}
  \label{sec:case-g_1-=}
  Let \(M\) be a torus manifold with \(G\)-action.
  Suppose \(\tilde{G}=G_1\times G_2\) with \(G_1=SU(l_1+1)\) elementary. 
  Then the \(W(S(U(l_1)\times U(1)))\)-action on \(\mathfrak{F}_1\) has an orbit \(\mathfrak{F}_1'\) with \(l_1\) elements and
  there is a component \(N_1\) of \(\bigcap_{M_i\in\mathfrak{F}_1'}M_i\), which contains a \(T\)-fixed point.
\end{lemma}
\begin{proof}
  We know that \(W(SU(l_1+1))=S_{l_1+1}=S(\mathfrak{F}_1)\) and \(W(S(U(l_1)\times U(1)))=S_{l_1}\subset S_{l_1+1}\).
  Therefore the first statement follows.
  Let \(x\in M^T\). Then, by Lemmas \ref{sec:g-action-m} and \ref{lem:iso2}, \(x\) is contained in the intersection of \(l_1\) characteristic submanifolds of \(M\) belonging to \(\mathfrak{F_1}\).
  Because \(W(G_1)=S(\mathfrak{F}_1)\) there is a \(g\in N_{G_1}T_1\) such that \(gx \in \bigcap_{M_i\in \mathfrak{F}_1'}M_i\).
  Therefore the second statement follows.
\end{proof}

\begin{remark}
  We will see in Lemma~\ref{sec:case-g_1-=-2} that \(\bigcap_{M_i\in\mathfrak{F}_1'}M_i\) is connected.
\end{remark}

\begin{Lemma}
  \label{lem:poly}  
  Let \(M\) be a torus manifold with \(G\)-action.
  Suppose \(\tilde{G}=G_1\times G_2\) with \(G_1=SU(l_1+1)\) elementary. 
  Furthermore, let \(N_1\) as in Lemma \ref{sec:case-g_1-=}.
  Then there is a group homomorphism \(\psi_1:S(U(l_1)\times U(1))\rightarrow Z(G_2)\) such that, with
  \begin{align*}
    H_0&= SU(l_1+1)\times \image \psi_1,\\
    H_1&= S(U(l_1)\times U(1)) \times \image \psi_1,\\
    H_2&= \{(g,\psi_1(g))\in H_1; g \in S(U(l_1)\times U(1))\},
  \end{align*}
  \begin{enumerate}
  \item \( \image \psi_1\) is the projection of \(\lambda(M_i)\) to \(G_2\), for all \(M_i \in \mathfrak{F}_1\),
  \item \(N_1\) is a component of \(M^{H_2}\),
  \item \(N_1\) is invariant under the action of \(G_2\),
  \item \(M=G_1N_1=H_0N_1\).
  \end{enumerate}
\end{Lemma}
\begin{proof}
  Denote by \(T_2\) the maximal torus \(T\cap G_2\) of \(G_2\).
  Let \(x\in N_1^T\).
  If \(x\in M^{SU(l_1+1)}\), then  we have, by Lemma~\ref{lem:iso2}, the \(SU(l_1+1)\times T_2\)-representation
  \begin{equation*}
    T_xM=W\otimes_\C V_1 \oplus \bigoplus_{i=2}^{n-l_1} V_i,
  \end{equation*}
  where \(W\) is the standard complex representation of \(SU(l_1+1)\) or its dual and the \(V_i\) are one-dimensional complex representations of \(T_2\).
  Because \(G_2\) acts effectively on \(M\) the weights of the \(V_i\) form a basis of the integral lattice in \(LT_2^*\).
  From the description of the weight spaces of \(T_xM\) given in the proof of Lemma \ref{lem:iso2}, we get that \(T_xN_1\) is \(S(U(l_1)\times U(1))\)-invariant and that there is a one-dimensional complex representation \(W_1\) of \(S(U(l_1)\times U(1))\) such that
  \begin{equation*}
    T_xN_1=W_1\otimes_\C V_1 \oplus \bigoplus_{i=2}^{n-l_1}V_i.
  \end{equation*}

  Now assume that \(x\) is not fixed by \(SU(l_1+1)\).
  Because, by Lemma~\ref{lem:iso}, \(G_1x\subset M^{T_2}\) is \(G_1\)-equivariantly diffeomorphic to \(\mathbb{C}P^{l_1}\),
  we see by the definition of \(N_1\) that \(G_{1x} = S(U(l_1) \times U(1))\).

   At the point \(x\), we get a representation of \(S(U(l_1)\times U(1)) \times T_2\) of the form
\begin{equation*}
  T_x M= T_xN_1 \oplus T_xG_1x.
\end{equation*}
Since \(T_2\) acts effectively on \(M\) and trivially on \(G_1x\), there is a decomposition
\begin{equation*}
  T_xN_1 =  \bigoplus_{i=1}^{n-l_1} V_i \otimes_\C W_i,
\end{equation*}
where the \(W_i\) are one-dimensional complex \(S(U(l_1)\times U(1))\)-representations and the \(V_i\) are one-dimensional complex \(T_2\)-representations whose weights form a basis of the integral lattice in \(LT^*_2\).

Therefore, in both cases, there is a homomorphism \(\psi_1: S(U(l_1)\times U(1)) \rightarrow S^1 \rightarrow T_2\) such that, for all \(g \in  S(U(l_1)\times U(1))\), \((g,\psi_1(g))\) acts trivially on \(T_xN_1=\bigoplus_{i=1}^{n-l_1} V_i \otimes_\C W_i \).

Hence the component of the identity of the isotropy subgroup of the torus \(T\) for generic points in \(N_1\) is given by 
\begin{equation}
\label{eq:2}
  H_3=\{(t,\psi_1(t))\in T_1\times T_2\}.
\end{equation}
With Lemma~\ref{lem:torus-m2}, we see that
\begin{equation}
  \label{eq:7}
  H_3=\langle\lambda(M_i);M_i\in\mathfrak{F}_1,M_i \supset N_1\rangle.
\end{equation}

Because the Weyl-group of \(G_2\) acts trivially and orientation preserving on \(\mathfrak{F}_1\), \(\lambda(M_i)\), \(M_i\in\mathfrak{F}_1\), is pointwise fixed by the action of \(W(G_2)\) on \(T\) by Lemma~\ref{sec:action-weyl-group}.
It follows with (\ref{eq:7}) that \(H_3\) is pointwise fixed by the action of \(W(G_2)\) on \(T\).
Here \(W(G_2)\) acts on \(T\) by conjugation.
Therefore the image of \(\psi_1\) is contained in the center of \(G_2\).
Furthermore \(\image \psi_1\) is the projection of \(\lambda(M_i)\), \(M_i\in \mathfrak{F}_1\), to \(T_2\).

Because \(H_3\) commutes with \(G_2\) it follows that \(N_1\) is \(G_2\)-invariant.
So we have proved the first and the third statement.

Now we turn to the second and fourth part.

Because \(T_xN_1=(T_xM)^{H_3}=(T_xM)^{H_2}\), \(N_1\) is a component of \(M^{H_2}\).
Because, by Lemma \ref{lem:lie-algebra-calc}, \(H_1\) is the only proper closed connected subgroup of \(H_0\), which contains \(H_2\) properly, for \(y \in N_1\) there are the following possibilities
\begin{itemize}
\item \(H_{0y}^0=H_0\),
\item \(H_{0y}^0=H_1\) and  \(\dim H_0y=2l_1\),
\item \(H_{0y}^0=H_2\) and \(\dim H_0y=2l_1+1\),
\end{itemize}
where \(H_{0y}^0\) is the identity component of the stabilizer of \(y\) in \(H_0\).
If \(g \in H_0\) such that \(gy \in N_1\), then we have \(H_{0gy}^0=gH_{0y}^0g^{-1}\in \{H_0,H_1,H_2\}\). Therefore
\begin{equation*}
  g \in N_{H_0}H_{0y}^0=
  \begin{cases}
    H_0& \text{if }y \in M^{H_0}\\
    H_1& \text{if }y \not\in M^{H_0} \text{ and } l_1 > 1\\
    N_{G_1}T_1\times \image \psi_1 & \text{if } H^0_{0y}=H_1 \text{ and } l_1=1\\
    T_1 \times \image \psi_1 & \text{if } H^0_{0y}=H_2, l_1=1 \text{ and } \image \psi_1 \neq \{1\}.  
  \end{cases}
\end{equation*}

Now let \(y \in N_1\) such that \(H_{0y}^0\neq H_0\).
Because \(N_1\) is a component of \(M^{H_2}\) and \(H_0y\) is \(H_2\) invariant, \(N_1\cap H_0y\) is a union of  some components of \((H_0y)^{H_2}\).
Therefore  \(N_1\cap H_0y\) is a submanifold of \(M\).
Moreover,
\begin{equation*}
  T_yN_1\cap T_yH_0y=(T_yM)^{H_2}\cap T_yH_0y=(T_yH_0y)^{H_2}=T_y(N_1\cap H_0y).
\end{equation*}
Hence,
\begin{multline*}
  \dim T_yN_1 \cap T_yH_0y = \dim N_1\cap H_0y
  \leq \dim H_1y\\= \dim H_1/H_{0y}^0=
 \begin{cases}
  0& \text{if } H_{0y}^0=H_1\\
  1& \text{if } H_{0y}^0=H_2  \text{ and } \image \psi_1 \neq \{1\}
  \end{cases}
\end{multline*}
follows.
Therefore \(N_1\) intersects \(H_0y\) transversely in \(y\).
It follows, by Lemma~\ref{sec:lie-groups-2}, that \(GN_1-N_1^{H_0}=H_0N_1-N_1^{H_0}\) is an open subset of \(M\).

Because \(M\) is connected and \(\codim M^{H_0}\geq 4\), \(M-M^{H_0}\) is connected.
Since \((M-M^{H_0})\cap H_0N_1=H_0N_1-N_1^{H_0}\) is closed in \(M-M^{H_0}\), we have \(M-M^{H_0}=H_0N_1-N_1^{H_0}\).
Hence
\begin{align*}
  M&=\left(M-M^{H_0}\right)\amalg M^{H_0}= \left(H_0N_1-N_1^{H_0}\right) \amalg M^{H_0}\\
  &=\left(H_0N_1-N_1^{H_0}\right) \amalg \left(M^{H_0}\cap N_1\right)\amalg \left(M^{H_0}-N_1^{H_0}\right)\\
  &=H_0N_1\amalg \left(M^{H_0}-N_1^{H_0}\right).
\end{align*}

Because \(N_1\) is a component of \(M^{H_2}\), \(N_1^{H_0}\) is a union of components of \(M^{H_0}\).
Therefore \(M^{H_0}-N_1^{H_0}\) is closed in \(M\).
Because \(H_0N_1\) is closed in \(M\) it follows that \(M=GN_1=H_0N_1=G_1N_1\).
\end{proof}

The following lemma guarantees together with Lemma~\ref{sec:lie-groups} that, if \(l_1>1\), then the homomorphism \(\psi_1\) is independent of all choices made in its construction namely the choice of \(N_1\) and of \(x \in N_1^{T}\). 

\begin{lemma}
  \label{sec:case-g_1-=-6}
  In the situation of Lemma~\ref{lem:poly} let \(T'=T_2\) or \(T'=\image \psi_1\).
  Then the principal orbit type of the \(G_1\times T'\)-action on \(M\) is given by \((G_1\times T')/H_2\).
\end{lemma}
\begin{proof}
  Let \(H\subset G_1\times T'\) be a principal isotropy subgroup.
  Then, by Lemma~\ref{lem:poly}, we may assume \(H\supset H_2\).
  Consider the projection
  \begin{equation*}
    \pi_1: G_1\times T' \rightarrow G_1
  \end{equation*}
  on the first factor.

  At first we show that the restriction of \(\pi_1\) to \(H\) is injective.
  Because \((G_1\times T')_x\cap T'= T'_x\) for all \(x\in M\) and the \(T'\)-action on \(M\) is effective there is an \(x\in M\) such that
  \begin{equation*}
    (G_1\times T')_x \cap T'=\{1\}.
  \end{equation*}
  Furthermore, there is an \(g\in G_1\times T'\) such that \((G_1\times T')_x\supset gHg^{-1}\).

  Because \(T'\) is contained in the center of \(G_1 \times T'\), we get
  \begin{align*}
    gHg^{-1}\cap T' &= \{1\},\\
    H\cap g^{-1}T'g &= \{1\},\\
    H\cap T' &= \{1\}.
  \end{align*}
  Therefore the restriction of \(\pi_1\) to \(H\) is injective.

  Furthermore, \(\pi_1(H)\supset \pi_1(H_2)=S(U(l_1)\times U(1))\).
  Therefore, by Lemma~\ref{sec:lie-algebra-calc}, we have
  \begin{equation*}
    \pi_1(H)=
    \begin{cases}
      SU(l_1+1),S(U(l_1)\times U(1)) &\text{if } l_1>1\\
      SU(l_1+1),S(U(l_1)\times U(1)), N_{G_1}T_1 &\text{if } l_1=1.
    \end{cases}
  \end{equation*}
  There is a left inverse \(\phi:\pi_1(H)\rightarrow H\hookrightarrow G_1\times T'\) to \(\pi_1|_H\).
  Therefore there is a group homomorphism \(\psi':\pi_1(H)\rightarrow T'\) such that
  \begin{equation*}
    H=\phi(\pi_1(H))=\{(g,\psi'(g))\in G_1\times T';\; g\in \pi_1(H)\}.
  \end{equation*}
  Because \(H_2\) is a subgroup of \(H\), we see that \(\psi'|_{S(U(l_1)\times U(1))}=\psi_1\).

  At first we discuss the cases \(\pi_1(H)=SU(l_1+1)\) and \(\pi_1(H)=S(U(l_1)\times U(1))\).
  Because \(T'\) is abelian we have in these cases
  \begin{equation*}
    H=\phi(\pi_1(H))=
    \begin{cases}
      G_1 &\text{if } \pi_1(H)= SU(l_1+1)\\
      H_2 &\text{if } \pi_1(H)= S(U(l_1)\times U(1)).
    \end{cases}
  \end{equation*}
 The first case does not occur because \(G_1\) acts non-trivially on \(M\).

 Now we discuss the case \(l_1=1\) and \(\pi_1(H)=N_{G_1}T_1\).
 Because for \(t\in T_1\) and \(g\in N_{G_1}T_1 -T_1\) we have
 \begin{equation*}
   \psi'(t)^{-1}=\psi'(gtg^{-1})=\psi'(g)\psi'(t)\psi'(g)^{-1}=\psi'(t),
 \end{equation*}
 it follows that \(\psi_1\) is trivial in this case.

 Let \(x\in M^T\). Then it follows by the definition of \(\psi_1\) in the proof of Lemma~\ref{lem:poly} that \(x\) is not a fixed point of \(G_1\).
 By Lemma~\ref{lem:iso}, we know that
 \begin{equation*}
   G_{1x}=S(U(l_1)\times U(1))=T_1.
 \end{equation*}
 Therefore \((G_{1}\times T')_x=T_1 \times T'\) is abelian.
 But \(H\) is non-abelian if \(\pi_1(H)=N_{G_1}T_1\).
 This is a contradiction because \(H\) is conjugated to a subgroup of \((G_1\times T')_x\).
\end{proof}

  If \(l_1=1\), we have \(\#\mathfrak{F}_1=2\) and \(W(S(U(l_1)\times U(1)))=\{1\}\).
  Therefore there are two choices for \(N_1\). Denote them by \(M_1\) and \(M_2\).

\begin{lemma}
  \label{sec:case-g_1-=-7}
  In the situation described above 
  let \(\psi_i\) be the homomorphism constructed for \(M_i\), \(i=1,2\).
  Then we have \(\psi_1=\psi_2^{-1}\).
\end{lemma}
\begin{proof}
  By (\ref{eq:2}) and (\ref{eq:7}), we have
  \begin{equation*}
    \lambda(M_i)=\{(t,\psi_i(t))\in H_1;\; t \in S(U(1)\times U(1))\}.
  \end{equation*}
  Now, with Lemma~\ref{sec:action-weyl-group}, we see
  \begin{align*}
    \lambda(M_1)&=g\lambda(M_2)g^{-1}=\{(t^{-1},\psi_2(t))\in H_1;\; t \in S(U(1)\times U(1))\}\\
    &=\{(t,\psi_2(t)^{-1})\in H_1;\; t \in S(U(1)\times U(1))\},
  \end{align*}
  where \(g\in N_{G_1}T_1-T_1\).
  Therefore the result follows.
\end{proof}

\begin{cor}
  \label{cor:hhh}
If in the situation of Lemma   \ref{lem:poly}  the \(G_1\)-action on \(M\) has no fixed point,
 then M is the total space of a \(G\)-equivariant fiber bundle over \(\mathbb{C}P^{l_1}\) with fiber some torus manifold; more precisely  \(M=H_0 \times_{H_1} N_1\). 
\end{cor}
\begin{proof}
  \(H_0 \times_{H_1}N_1\) is defined to be the space \(H_0 \times N_1/\sim_1\), where
  \begin{align*}
    &&(g_1,y_1)&\sim_1(g_2,y_2)\\
    &\Leftrightarrow &\exists h \in H_1  \quad g_1h^{-1}&=g_2 \text{ and } hy_1=y_2.
  \end{align*}
  By Lemma \ref{lem:poly} we have that \(M= H_0N_1= (H_0\times N_1)/\sim_2\), where
  \begin{align*}
    &&(g_1,y_1)&\sim_2 (g_2,y_2)\\
    &\Leftrightarrow &g_1y_1 &=g_2y_2.
  \end{align*}
We show that the two equivalence relations \(\sim_1,\sim_2\) are equal.

For \((g_1,y_1),(g_2,y_2)\in H_0\times N_1\) we have
    \begin{align*}
      &&g_1y_1 &=g_2y_2\\
      &\Leftrightarrow &\exists h \in N_{H_0}H_{0y_1}^0  \quad g_1h^{-1}&=g_2 \text{ and } hy_1=y_2\\
      &\Leftrightarrow &\exists h \in H_1  \quad g_1h^{-1}&=g_2 \text{ and } hy_1=y_2.
    \end{align*}
For the last equivalence we have to show the implication from the second to the third line.
If \(l_1>1\), \(N_{H_0}H_{0y_1}^0\) is equal to \(H_1\) because \(y_1\) is not a \(H_0\)-fixed point. So we have \(h\in H_1\).

If \(l_1=1\), then \(N_1\) is a characteristic submanifold of \(M\) belonging to \(\mathfrak{F}_1\).
If  \(H_{0y_1}^0=H_2\) we are done because \(N_{H_0}H_{0y_1}^0=H_1\).

Now assume that  \(H_{0y_1}^0=H_1\) and there is an \(h\in N_{G_1}T_1 \times \image \psi_1- T_1\times \image \psi_1\) such that \(y_2=hy_1\in N_1\).
Then \(y_2\in N_1\cap N_2\subset M^{T_1\times \image \psi_1}\), where \(N_2\) is the other characteristic submanifold of \(M\) belonging to \(\mathfrak{F}_1\).

As shown in the proof of Lemma  \ref{lem:poly}, \(N_1\) intersects \(H_0y_2\) transversely in \(y_2\).
Therefore one has
\begin{equation*}
  T_{y_2}N_1 \oplus T_{y_2}H_0y_2= T_{y_2}M = T_{y_2}N_2\oplus T_{y_2}H_0y_2
\end{equation*}
as \(T_1\times \image \psi_1\)-representations. 
This implies
\begin{equation*}
  T_{y_2}N_1 = T_{y_2}N_2
\end{equation*}
as \(T_1\times \image \psi_1\)-representations.
Therefore \(T_1\times \image \psi_1\) acts trivially on both \(N_1\) and \(N_2\).
Therefore we have \(\image\psi_1=\{1\}\) and \(\lambda(N_1)=\lambda(N_2)=T_1\).
Hence, we get a contradiction because the intersection of \(N_1\) and \(N_2\) is non-empty.
\end{proof}

\begin{cor}
\label{cor:10}
  In the situation of Lemma \ref{lem:poly} we have \(M^{G_1} =M^{H_0}= \bigcap_{M_i \in \mathfrak{F}_1} M_i\).
\end{cor}
\begin{proof}
  At first let \(l_1 >1\).
  By Lemma \ref{lem:poly}, we know \(M^{H_0}\subset M^{G_1}\subset N_1\). Therefore \(M^{G_1}\subset \bigcap_{g\in N_{G_1}T_1} gN_1 = \bigcap_{M_i \in \mathfrak{F}_1} M_i\).
  There is a \(g \in N_{G_1}T_1-T_1\)  with   \(gH_2g^{-1}\not\subset H_1\).
  Thus, the subgroup \(\langle H_2,gH_2g^{-1}\rangle\) of \(H_0\), which is generated by \(H_2\) and \(gHg^{-1}\), contains \(H_2\) as a proper subgroup.
  Therefore \( \langle H_2,gH_2g^{-1}\rangle=H_0\) follows by Lemma~\ref{lem:lie-algebra-calc}.
  Because \(H_2\) acts trivially on \(N_1\), this equation implies 
  \begin{equation*}
    M^{H_0}\supset \bigcap_{g\in N_{G_1}T_1} gN_1 = \bigcap_{M_i \in \mathfrak{F}_1} M_i.
  \end{equation*}

  Now let \(l_1=1\).
  Then \(\mathfrak{F}_1\) contains two characteristic submanifolds \(M_1\) and \(M_2\). As in the first case one can show that \(M^{H_0}\subset M^{G_1}\subset M_1\cap M_2\).

  So \(M^{H_0}\supset M_1\cap M_2\) remains to be shown.
  Assume that there is an \(y \in M_1\cap M_2 - M^{H_0}\).
  Then we also have \(y \in M^{H_1}\).
  Now the above assumption leads to a contradiction as in the proof of Corollary   \ref{cor:hhh}.
\end{proof}

\begin{cor} 
\label{cor:case-g_1-=}
  If in the situation of Lemma \ref{lem:poly} \(\psi_1\) is trivial, then \(M^{G_1}\) is empty. 
  Otherwise the normal bundle of \(M^{G_1} =M^{H_0}= \bigcap_{M_i \in \mathfrak{F}_1} M_i\) possesses a \(G\)-invariant complex structure.
  It is induced by the action of some element \(g\in \image \psi_1\).
  Furthermore, it is unique up to conjugation.
\end{cor}
\begin{proof}
  If \(\psi_1\) is trivial, then \(\langle\lambda(M_i);M_i\in \mathfrak{F}_1\rangle\) is contained in the \(l_1\)-dimensional maximal torus of \(G_1\) by Lemma~\ref{lem:poly}.
  By Corollary \ref{cor:10} and Lemma \ref{lem:torus-m2}, it follows that \(M^{H_0}\) is empty.

  If \(\psi_1\) is non-trivial, then for \(y\in M^{H_0}\) we have
\begin{equation*}
  N_y(M^{H_0},M)=V_{\mathbb{C}}\oplus V_{\mathbb{R}},
\end{equation*}
where \(\image \psi_1\) acts non-trivially on the \(H_0\)-representation \(V_\C\) and trivially on the \(H_0\)-representation \(V_\R\).
Clearly \(V_\C\) has at least real dimension two and the action of \(\image\psi_1\) induces a \(H_0\)-invariant complex structure on \(V_\C\).
Because \(M^{H_0}\) has codimension \(2l_1+2\) by Corollary~\ref{cor:10} and Lemma~\ref{lem:torus-m2}, the dimension of  \(V_\mathbb{R}\) is at most \(2l_1\). So it follows from \cite[p. 53-54]{0796.57001} that \(V_\mathbb{R}\) is trivial if \(l_1\neq 3\).

If \(l_1=3\), we have \(SU(4)=\text{Spin}(6)\), and there are two possibilities:
\begin{enumerate}
\item \(V_\mathbb{R}\) is trivial.
\item \(V_\mathbb{R}\) is the standard representation of \(SO(6)\) and \(V_\mathbb{C}\) a one-dimensional complex representation of \(\image \psi_1\).
\end{enumerate}
Because the principal orbits are dense in \(M\), it follows with the slice theorem that the principal orbit types of the \(H_0\)-actions on \(N_y(M^{H_0},M)\) and \(M\) are equal.
Therefore in the second case the principal orbit type of the \(H_0\)-action on \(M\) is given by \(\text{Spin}(6)\times S^1/\text{Spin}(5)\times \{1\}\).
Therefore we see with Lemma~\ref{sec:case-g_1-=-6} that the second case does not occur.

Because of dimension reasons we get
\begin{equation*}
   N_y(M^{H_0},M)=V_{\mathbb{C}}=W\otimes_\C V,
\end{equation*}
where \(W\) is the standard complex representation of \(SU(l_1+1)\) or its dual  and \(V\) is a complex one-dimensional \(\image \psi_1\)-representation.
Because \(\image \psi_1\subset Z(G)\), we see that \(N(M^{H_0},M)\) has a \(G\)-invariant complex structure, which is induced by the action of some \(g\in \image \psi_1\).

Next we prove the uniqueness of this complex structure.
Assume that there is another \(g'\in Z(G)\cap G_y\) whose action induces a complex structure on \(N_y(M^{H_0},M)\). Then \(g'\) induces a -- with respect to the complex structure induced by \(g\) -- complex linear \(H_0\)-equivariant map
\begin{equation*}
  J: N_y(M^{H_0},M) \rightarrow N_y(M^{H_0},M)
\end{equation*}
with \(J^2 + \id = 0\).
Because \(N_y(M^{H_0},M)\) is an irreducible \(H_0\)-representation it follows by Schur's Lemma that \(J\) is multiplication with \(\pm i\).
Therefore \(g'\) induces up to conjugation the same complex structure as \(g\).
\end{proof}

\begin{cor}
\label{sec:case-g_1-=-8}
  If in the situation of Lemma \ref{lem:poly} \(M^{G_1}=M^{H_0}\neq \emptyset\), then \(\ker \psi_1 =SU(l_1)\).
\end{cor}
\begin{proof}
  Let \(y\in M^{H_0}\). Then by the proof of Corollary \ref{cor:case-g_1-=} we have
  \begin{equation*}
    N_y(M^{H_0},M)= W\otimes_\C V,
  \end{equation*}
where \(W\) is the standard complex \(SU(l_1+1)\)-representation or its dual and \(V\) is a one-dimensional complex \(\image \psi_1\)-representation.
Furthermore, \(\image \psi_1\) acts effectively on \(M\).

Because the principal orbits are dense in \(M\), it follows with the slice theorem that the principal orbit types of the \(H_0\)-actions on \(N_y(M^{H_0},M)\) and \(M\) are equal.
Therefore a principal isotropy subgroup of the \(H_0\)-action on \(M\) is given by
\begin{equation*}
    H =\left\{(g,g_{l+1}^{\pm 1})\in H_1; g=
\left( 
 \begin{matrix}
    A&0\\
    0&g_{l_1+1}
  \end{matrix}
\right)\in S(U(l_1)\times U(1))
\text{ with } A \in U(l_1)
\right\}.
\end{equation*}
Now the statement follows by the uniqueness of the principal orbit type and Lemmas~\ref{sec:case-g_1-=-6} and \ref{sec:lie-groups}.
\end{proof}

\begin{lemma}
  \label{sec:case-g_1-=-2}
   In the situation of Lemma~\ref{sec:case-g_1-=}, the intersection \(\bigcap_{M_i\in\mathfrak{F}_1'}M_i=N_1\) is connected.
\end{lemma}
\begin{proof}
  Let \(\tilde{M}\) be the blow up of \(M\) along \(M^{G_1}\) and \(\tilde{N}_1\) the proper transform of \(N_1\) in \(\tilde{M}\).  
  By Corollary~\ref{cor:hhh}, we have \(\tilde{M}= H_0\times_{H_1}\tilde{N}_1\), which is a fiber bundle over \(\C P^{l_1}\).
  The characteristic submanifolds of \(\tilde{M}\), which are permuted by \(W(G_1)\), are given by the preimages of the characteristic submanifolds of \(\C P^{l_1}\) under the bundle map.
  By Corollary~\ref{sec:blowing-up-1} and the discussion following this corollary, they are also given by the proper transforms \(\tilde{M}_i\) of the characteristic submanifolds \(M_i\in\mathfrak{F}_1\) of \(M \).
  Because \(l_1\) characteristic submanifolds of \(\mathbb{C}P^{l_1}\) intersect in a single point we see \(\bigcap_{M_i\in\mathfrak{F}_1'}\tilde{M}_i=\tilde{N}_1\).
  Therefore this intersection is connected.
  Because  \(\bigcap_{M_i\in\mathfrak{F}_1'}\tilde{M}_i\) is mapped by \(F\) to  \(\bigcap_{M_i\in\mathfrak{F}_1'}M_i\), we see that \(\bigcap_{M_i\in\mathfrak{F}_1'}M_i=N_1\) is connected.
\end{proof}

\subsection{Blowing up along $M^{G_1}$}
\label{sec:blowing_M^G_1}
By blowing up a torus manifold \(M\) with \(G\)-action along \(M^{G_1}\) one gets a torus manifold \(\tilde{M}\) without \(G_1\)-fixed points.

Denote by \(\tilde{N}_1\) the proper transform of \(N_1\) as defined in Lemma~\ref{sec:case-g_1-=}.
Then by Corollary \ref{cor:blowing-up} there is a \(\langle H_1,G_2\rangle\)-equivariant diffeomorphism  \(F:\tilde{N}_1\rightarrow N_1\).

As in section \ref{sec:blow}, we denote by \(M_0=P_\mathbb{C}(N(M^{G_1},M)\oplus\{0\})\) the exceptional submanifold of \(\tilde{M}\).
Because \(M_0\cap \tilde{N}_1\) is mapped by this diffeomorphism to \(M^{G_1}=M^{H_0}=N_1^{H_0}\), \(H_1\) acts trivially on \(M_0\cap\tilde{N}_1\).
By Corollary \ref{cor:hhh} we know that \(\tilde{M}\) is diffeomorphic to \(H_0\times_{H_1}\tilde{N}_1= H_0\times_{H_1}N_1\).

A natural question arising here is:
When is a torus manifold of this form a blow up of another torus manifold with \(G\)-action?

We claim that this is the case if and only if \(N_1\) has a codimension two submanifold, which is fixed by the \(H_1\)-action and \(\ker \psi_1 =SU(l_1)\).

\begin{Lemma}
  \label{lem:class1}
Let \(N_1\) be a torus manifold with \(G_2\)-action, \(A\) a closed codimension two submanifold of \(N_1\), \(\psi_1 \in \Hom(S(U(l_1)\times U(1)), Z(G_2))\) such that \(\image \psi_1\) acts trivially on \(A\) and \(\ker \psi_1=SU(l_1)\).
Let also
\begin{align*}
  H_0&=SU(l_1+1)\times \image \psi_1,\\
  H_1&=S(U(l_1)\times U(1)) \times \image \psi_1,\\
  H_2&= \{(g,\psi_1(g)); g \in S(U(l_1)\times U(1))\}.
\end{align*}
\begin{enumerate}
\item Then \(H_1\) acts on \(N_1\) by \((g,t)x=\psi_1(g)^{-1}t x\), where \(x\in N_1\) and \((g,t)\in H_1\).
\item Assume that \(Z(G_2)\) acts effectively on \(N_1\) and let \(y \in A\) and \(V\) the one-dimensional complex \(H_1\)-representation \(N_y(A,N_1)\).
Then \(V\) extends to an \(l_1+1\)-dimensional complex representation of \(H_0\).
Therefore there is an \(l_1+1\)-dimensional complex \(G\)-vector bundle \(E'\) over \(A\) which contains \(N(A,N_1)\) as a subbundle. 
\item Then the normal bundle of \(H_0/H_1\times A\) in \(H_0\times_{H_1}N_1\) is isomorphic to the tautological bundle over \(P_\C(E'\oplus\{0\})\).
\end{enumerate}
\end{Lemma}
The lemma guarantees together with the discussion at the end of section~\ref{sec:blow} that one can remove  \(H_0/H_1\times A\) from \(H_0\times_{H_1}N_1\) and replace it by \(A\) to get a torus manifold with \(G\)-action \(M\) such that \(M^{H_0}=A\). The blow up of \(M\) along \(A\) is \(H_0\times_{H_1}N_1\).

\begin{proof}
  (1) is trivial. 

  (2)
  For \(i=1,\dots,l_1+1\) let
  \begin{equation*}
    \lambda_i: T_1 \rightarrow S^1 \quad \left(
      \begin{matrix}
        g_1& & \\
        &\ddots & \\
        & & g_{l_1+1}
      \end{matrix}
      \right)
      \mapsto g_i
  \end{equation*}
and \(\mu: \image \psi_1 \rightarrow S^1\) the character of the \(\image \psi_1\) representation \(N_y(A,N_1)\).
Then \(\mu\) is an isomorphism.

And by \cite[p. 176]{0581.22009} the character ring of the maximal torus \(T_1\times  \image \psi_1\) of \(H_1=S(U(l_1)\times U(1))\times\image \psi_1\) is given by
  \begin{equation*}
    R(T_1\times \image \psi_1)=\mathbb{Z}[\lambda_1,\dots,\lambda_{l_1+1},\mu,\mu^{-1}]/(\lambda_1\dotsm\lambda_{l_1+1}-1).
  \end{equation*}
With this notation, the character of \(V\) is given by \(\mu \lambda_{l_1+1}^{\pm 1}\).
Therefore the \(H_0\)-representation \(W\) with the character \(\mu \sum_{i=1}^{l_1+1} \lambda_{i}^{\pm 1}\)
is \(l_1+1\)-dimensional and \(V\subset W\).

 Let \(G_2=G_2'\times \image\psi_1\) and \(E''=N(A,N_1)\) equipped with the action of \(G_2'\), but without the action of \(H_1\).
  Then \(E'=E''\otimes_\C W\) is a \(G\)-vector bundle with the required features.

  Now we turn to (3).
  The normal bundle of \(H_0/H_1 \times A\) in \(H_0\times_{H_1}N_1\) is given by \(H_0\times_{H_1}N(A,N_1)\).
  
Consider the following commutative diagram
\begin{equation*}
  \xymatrix{
    H_0\times_{H_1}N(A,N_1) \ar[r]_(.47)f\ar[d]_{\pi_1} &   P_\mathbb{C}(E'\oplus \{0\}) \times  E' \ar[d]_{\pi_2}\\
    H_0/H_1 \times A   \ar[r]_(.47)g  &   P_\mathbb{C}(E'\oplus \{0\})
}
\end{equation*}
where the vertical maps are the natural projections and \(f,g\) are given by 
\begin{equation*}
  f([(h_1,h_2):m])= ([m\otimes h_2h_1e_1],m\otimes h_2h_1e_1)
\end{equation*}
 and
 \begin{equation*}
   g([h_1,h_2],q)=[m_q\otimes h_2h_1e_1],
 \end{equation*}
 where \(e_1\in W-\{0\}\) is fixed such that for all \(g' \in S(U(l_1)\times U(1))\) \(\psi_1(g')g'e_1=e_1\) and \(m_q\neq 0\) some element of the fiber of \(N(A,N_1)\) over \(q\in A\).

The map \(f\) induces an isomorphism of the normal bundle of \(H_0/H_1\times A\) in \(H_0\times_{H_1}N_1\) and the tautological bundle over \(P_\mathbb{C}(E'\oplus \{0\})\).
\end{proof}

\subsection{Admissible triples}
\label{sec:admissible_triples}
Now we are in the position to state our first classification theorem.
To do so, we need the following definition.
\begin{definition}
  Let \(\tilde{G}=G_1\times G_2\) with \(G_1=SU(l_1+1)\). Then a triple \((\psi,N,A)\) with
  \begin{itemize}
  \item \(\psi \in \Hom(S(U(l_1)\times U(1)),Z(G_2))\),
  \item \(N\) a torus manifold with \(G_2\)-action,
  \item \(A\) the empty set or a closed codimension two submanifold of \(N\), such that \(\image \psi\) acts trivially on \(A\) and \(\ker \psi=SU(l_1)\) if \(A\neq \emptyset\), 
  \end{itemize}
  is called \emph{admissible for} \((\tilde{G},G_1)\).
  We say that two admissible triples \((\psi,N,A)\), \((\psi',N',A')\) for \((\tilde{G},G_1)\) are equivalent if there is a \(G_2\)-equivariant diffeomorphism \(\phi:N\rightarrow N'\) such that \(\phi(A)=A'\) and
  \begin{equation*}
    \psi=
    \begin{cases}
      \psi' &\text{ if } l_1>1\\
      \psi'^{\pm 1} &\text{ if } l_1=1.
    \end{cases}
  \end{equation*}
\end{definition}

\begin{theorem}
\label{sec:case-g_1-=-1}
  Let \(\tilde{G}=G_1\times G_2\) with \(G_1=SU(l_1+1)\).
  There is a one-to-one-correspondence between the \(\tilde{G}\)-equivariant diffeomorphism classes of torus manifolds with \(\tilde{G}\)-action such that \(G_1\) is elementary and the equivalence classes of admissible triples for \((\tilde{G},G_1)\).
\end{theorem}
\begin{proof}
  Let \(M\) be a torus manifold with \(\tilde{G}\)-action such that \(G_1\) is elementary.
  Then, by Corollaries~\ref{cor:10} and~\ref{sec:case-g_1-=-8}, \((\psi_1,N_1,M^{H_0})\) is an admissible triple, where \(\psi_1\) is defined as in Lemma~\ref{lem:poly} and \(N_1\) is defined as in Lemma~\ref{sec:case-g_1-=}.
  
  Let \((\psi,N,A)\) be an admissible triple for \((\tilde{G},G_1)\). If \(A\neq \emptyset\), then, by Lemma~\ref{lem:class1}, the blow down of \(H_0\times_{H_1}N\) along \(H_0/H_1\times A\) is a torus manifold with \(\tilde{G}\)-action.
  If \(A=\emptyset\), then we have the torus manifold \(H_0\times_{H_1}N\).

  We show that these two operations are inverse to each other.
  Let \(M\) be a torus manifold with \(\tilde{G}\)-action.
  If \(M^{H_0}=\emptyset\), then, by Corollary~\ref{cor:hhh}, we have \(M=H_0\times_{H_1}N_1\).
  If \(M^{H_0}\neq\emptyset\),
  then by the discussion before Lemma~\ref{lem:class1}, \(M\) is the blow down of \(H_0\times_{H_1}N_1\) along  \(H_0/H_1\times M^{H_0}\).
  
  Now assume \(l_1>1\).
 Let \((\psi,N,A)\) be an admissible triple with \(A\neq\emptyset\) and \(M\) the blow down of  \(H_0\times_{H_1}N\) along  \(H_0/H_1\times A\).
  Then, by the remark after Lemma~\ref{lem:class1}, we have \(A=M^{H_0}\). By  Lemma~\ref{sec:case-g_1-=-2} and Corollary \ref{cor:blowing-up}, we have \(N=N_1\).
  With Lemmas~\ref{sec:case-g_1-=-6} and~\ref{sec:lie-groups},
  one sees that \(\psi=\psi_1\), where \(\psi_1\) is the homomorphism defined in Lemma~\ref{lem:poly} for \(M\).

  Now let \((\psi,N,\emptyset)\) be an admissible triple and \(M=H_0\times_{H_1}N\). Then we have \(M^{H_0}=\emptyset\). By Lemma~\ref{sec:case-g_1-=-2} we have \(N=N_1\).
  As in the first case one sees \(\psi=\psi_1\).

  Now assume \(l_1=1\). Let \((\psi,N,A)\) be an admissible triple with \(A\neq\emptyset\) and \(M\) the blow down of  \(H_0\times_{H_1}N\) along  \(H_0/H_1\times A\).
  Then, by the remark after Lemma~\ref{lem:class1}, \(A=M^{H_0}\).
  By Lemma~\ref{sec:case-g_1-=-7}, we have two choices for \(N_1\) and \(\psi=\psi_1^{\pm 1}\).
  Because the two choices for \(N_1\) lead to equivalent admissible triples we recover the equivalence class of \((\psi,N,A)\).
  In the case \(A=\emptyset\) a similar argument completes the proof of the theorem.
\end{proof}

\begin{cor}
\label{sec:case-g_1-=-9}
  Let \(\tilde{G}=G_1\times G_2\) with \(G_1=SU(l_1+1)\).
  Then the torus manifolds with \(\tilde{G}\)-action such that \(G_1\) is elementary and \(M^{G_1}\neq \emptyset\) are given by blow downs of fiber bundles over \(\C P^{l_1}\) with fiber some torus manifold with \(G_2\)-action along a submanifold of codimension two.
\end{cor}

Now we specialise our classification result to special classes of torus manifolds.

\begin{theorem}
\label{sec:case-g_1-=-3}
  Let \(\tilde{G}=G_1\times G_2\) with \(G_1=SU(l_1+1)\), \(M\) a torus manifold with \(\tilde{G}\)-action and \((\psi,N,A)\) the admissible triple for \((\tilde{G},G_1)\) corresponding to \(M\).
  Then \(H^*(M;\mathbb{Z})\) is generated by its degree two part if and only if \(H^{*}(N;\mathbb{Z})\) is generated by its degree two part and \(A\) is connected.
\end{theorem}
\begin{proof}
  To make the notation simpler we omit the coefficients of the cohomology in the proof. 
  If \(H^*(M)\) is generated by its degree two part, then \(H^*(N)\) is generated by its degree two part by \cite[p. 716]{1111.57019}.
  Moreover, \(A\) is connected by 
  \cite[p. 738]{1111.57019} and Corollary~\ref{cor:10}.

  Now assume that \(H^*(N)\) is generated by its degree two part and \(A=\emptyset\).
  Then by Poincar\'e duality \(H_{\text{odd}}(N)=0\).
  Therefore by an universal coefficient theorem \(H^*(N)= \Hom(H_*(N),\mathbb{Z})\) is torsion free.
  By Corollary~\ref{cor:hhh}, \(M\) is a fiber bundle over \(\C P^{l_1}\) with fiber \(N\).
  Because the Serre-spectral sequence of this fibration degenerates we have
  \begin{equation*}
     H^*(M)\cong H^*(\C P^{l_1})\otimes H^*(N)
  \end{equation*}
  as a \(H^*(\C P^{l_1})\)-modul. Because \(H^*(N)\) is generated by its degree two part, it follows that
  the cohomology of \(M\) is generated by its degree two part.
  
  Now we turn to the general case \(A\neq \emptyset\).
  Then, by  \cite[p. 716]{1111.57019}, \(H^*(A)\) is generated by its degree two part.
  Moreover, \(H^*(N)\rightarrow H^*(A)\) is surjective.
  Let \(\tilde{M}\) be the blow up of \(M\) along \(A\) and \(F:\tilde{M}\rightarrow M\) the map defined in section~\ref{sec:blow}.
  
  Because, by Lemma~\ref{lem:blow1}, \(F\) is the identity outside some open tubular neighborhood of \(A\times \C P^{l_1}\), the induced homomorphism \(F^*:H^*(M,A)\rightarrow H^*(\tilde{M},A\times \C P^{l_1})\) is an isomorphism by excision.
  Furthermore, the push forward \(F_!:H^*(\tilde{M})\rightarrow H^*(M)\) is a section of \(F^*:H^*(M)\rightarrow H^*(\tilde{M})\).
  Therefore \(F^*:H^*(M)\rightarrow H^*(\tilde{M})\) is injective and \(H^{\text{odd}}(M)\) vanishes.

  Because \(A\) is connected, we have the following commutative diagram with exact rows and columns:

  $$
      \xymatrix{
  && 0\ar[d]&0\ar[d]\\
  &0 \ar[d] & H^2(\tilde{M},A\times \C P^{l_1}) \ar[r]\ar[d] & H^2(N,A)\ar[d]\\
  0\ar[r]&H^2(\C P^{l_1})\ar[r]\ar[d]& H^2(\tilde{M})\ar[r]\ar[d]&H^2(N)\ar[r]\ar[d]&0\\
  0\ar[r]&H^2(\C P^{l_1})\ar[r]\ar[d]&H^2(A\times \C P^{l_1})\ar[r]\ar[d]&H^2(A)\ar[r]\ar[d]&0\\
  & 0& H^3(\tilde{M},A\times \C P^{l_1})\ar[d] \ar[r]&0\\
  &  &0
}
 $$

Now from the snake lemma it follows that
\begin{align*}
  H^2(M,A)\cong_{F^*} H^2(\tilde{M},A\times \C P^{l_1})\cong H^2(N,A)
\end{align*}
and
\begin{equation*}
   H^3(M,A)\cong_{F^*}  H^3(\tilde{M},A\times \C P^{l_1}) \cong 0.
\end{equation*}
Because \(\iota_{NM}=F\circ \iota_{N\tilde{M}}\), where \(\iota_{NM},\iota_{N\tilde{M}}\) are the inclusions of \(N\) in \(M\) and \(\tilde{M}\), 
the left arrow in the following diagram is an isomorphism.
\begin{equation*}
  \xymatrix{
    0\ar[r]&H^2(M,A)\ar[r]\ar[d]_{\iota_{NM}^*}&H^2(M)\ar[r]\ar[d]_{\iota_{NM}^*}&H^2(A)\ar[r]\ar[d]_{\id}&0\\
    0\ar[r]&H^2(N,A)\ar[r]&H^2(N)\ar[r]&H^2(A)\ar[r]&0\\
}
\end{equation*}
Therefore it follows from the five lemma that
\begin{equation*}
  H^2(M)\cong H^2(N)
\end{equation*}
and
\begin{equation*}
  H^2(\tilde{M})\cong H^2(\C P^{l_1})\oplus H^2(N)\cong H^2(\C P^{l_1})\oplus H^2(M).
\end{equation*}

Let \(t\in H^2(\C P^{l_1})\) be a generator of \(H^*(\C P^{l_1})\) and \(x \in H^*(M)\).
Then, because \(H^*(\tilde{M})\) is generated by its degree two part, there are sums of products \(x_i\in H^*(M)\) of elements of \(H^2(M)\) such that
\begin{equation*}
  x = F_!F^*(x)= F_!\left(\sum F^*(x_i) t^i\right)= \sum x_i F_!(t^i).
\end{equation*}
Therefore it remains to show that \(F_!(t^i)\) is a product of elements of \(H^2(M)\).

The \(l_1+1\) characteristic submanifolds \(\tilde{M}_1,\dots,\tilde{M}_{l_1+1}\) of \(\tilde{M}\) which are permuted by \(W(G_1)\) are the preimages of the characteristic submanifolds of \(\C P^{l_1}\) under the projection \(\tilde{M}\rightarrow \C P^{l_1}\).
Therefore they can be oriented in such a way that \(t\) is the Poincar\'e-dual of each of them.

Because \(F\) restricts to a diffeomorphism \(\tilde{M}-A\times \C P^{l_1} \rightarrow M-A\) and \(F(\tilde{M}_i)=M_i\), \(F_!(t^i)\), \(i\leq l_1\), is the Poincar\'e-dual \(PD\left(\bigcap_{1\leq k\leq i}M_k\right)\) of the intersection \(\bigcap_{1\leq k\leq i}M_k\) of characteristic submanifolds of \(M\), which belong to \(\mathfrak{F}_1\).
Therefore for \(i\leq l_1\) we have
\begin{equation*}
  F_!(t)^i=PD\left(\bigcap_{1\leq k\leq i}M_k\right)=F_!(t^i).
\end{equation*}
Because \(t^i=0\) for \(i>l_1\), the statement follows.
\end{proof}

\begin{theorem}
\label{sec:case-g_1-=-4}
  Let \(\tilde{G}=G_1\times G_2\) with \(G_1=SU(l_1+1)\), \(M\) a torus manifold with \(\tilde{G}\)-action and \((\psi,N,A)\) the admissible triple for \((\tilde{G},G_1)\) corresponding to \(M\).
  Then \(M\) is quasitoric if and only if \(N\) is quasitoric and \(A\) is connected.
\end{theorem}
\begin{proof}
  At first assume that \(M\) is quasitoric.
  Then \(N\) is quasitoric and \(A\) connected because all intersections of characteristic submanifolds of \(M\) are quasitoric and connected.
  
  Now assume that \(N\) is quasitoric and \(A\subset N\) connected.
  Then, by Theorem~\ref{sec:case-g_1-=-3} and \cite[p. 738]{1111.57019}, the \(T\)-action on \(M\) is locally standard and \(M/T\) is a homology polytope.
  We have to show that \(M/T\) is face preserving homeomorphic to a simple polytope.

\begin{figure}
  \begin{center}
    \includegraphics{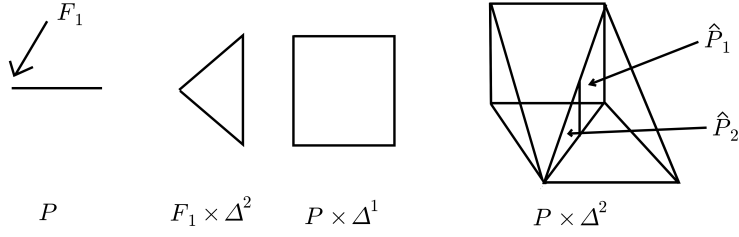}
    \caption{The orbit space of a blow down}
    \label{fig:2}
  \end{center}
\end{figure}

  Let \(T_2=T\cap G_2\).
  Then the orbit space \(N/T_2\) is face preserving homeomorphic to a simple polytope \(P\).
  Because \(A\) is connected, \(A/T_2\) is a facet \(F_1\) of \(P\).
 
  With the notation from Lemma~\ref{lem:class1} let
  \begin{equation*}
    B= \{(z_0:1)\in P_\C(E'\oplus\C); z_0\in E', |z_0|\leq 1\}.
  \end{equation*}
  Then the orbit space of the \(T\)-action on \(B\) is given by \(F_1\times \Delta^{l_1+1}\).

  Let \(B'\) be a closed \(\tilde{G}\)-invariant tubular neighborhood of \(H_0/H_1\times A\) in \(H_0\times_{H_1}N\).
  Then the bundle projection \(\partial B' \rightarrow H_0/H_1\times A\) extends to  an equivariant map
  \begin{equation*}
    H_0\times_{H_1}N -\mathring{B}' \rightarrow H_0\times_{H_1}N,
  \end{equation*}
  which induces a face preserving homeomorphism
  \begin{equation*}
    \left(H_0\times_{H_1}N-\mathring{B}'\right)/T \cong P\times \Delta^{l_1}.
  \end{equation*}

  Now \(M\) is given by gluing \(B\) and \(H_0\times_{H_1}N-\mathring{B}'\) along the boundaries \(\partial B, \partial B'\).
  The corresponding gluing of the orbit spaces is illustrated in Figure~\ref{fig:2} for the case \(\dim N=2\) and \(l_1=1\).
  Because the gluing map \(f: \partial B \rightarrow \partial B'\) is \(\tilde{G}\)-equivariant and \(G_1\) acts transitively on the fibers of \(\partial B \rightarrow A\) and  \(\partial B' \rightarrow A\), it induces a map
  \begin{equation*}
    \hat{f}: F_1\times \Delta^{l_1} = \partial B/T \rightarrow \partial B'/T = F_1\times \Delta^{l_1}, \quad\quad (x,y) \mapsto (\hat{f}_1(x),\hat{f}_2(x,y)),
  \end{equation*}
  where \(\hat{f}_1: F_1 \rightarrow F_1\) is a face preserving homeomorphism and \(\hat{f}_2: F_1\times\Delta^{l_1}\rightarrow \Delta^{l_1}\) such that, for all \(x \in F_1\), 
  \(\hat{f}_2(x,\cdot)\) is a face preserving homeomorphism of \(\Delta^{l_1}\).

  Now fix embeddings
  \begin{equation*}
    \Delta^{l_1+1}\hookrightarrow \R^{l_1+1} \text{ and } P \hookrightarrow \R^{n-l_1-1}\times [0,1[
  \end{equation*}
  such that \(\Delta^{l_1}\subset \R^{l_1}\times \{1\}\) and \(\Delta^{l_1+1}= \conv(0,\Delta^{l_1})\)  and \(P\cap \R^{n-l_1-1}\times \{0\}=F_1\).
  
  Denote by \(p_1:\R^{l_1+1}\rightarrow \R\) and \(p_2:\R^{n-l_1}\rightarrow \R\) the projections on the last coordinate.
  For \(\epsilon>0\) small enough, \(P\) and \(P\cap \{p_2\geq \epsilon\}\) are combinatorially equivalent.
  Therefore there is a face preserving homeomorphism
  \begin{equation*}
       g_1:P\rightarrow P\cap \{p_2\geq \epsilon\}
  \end{equation*}
  such that \(g_1(F_1)=P\cap \{p_2= \epsilon\}\) and \(g_1(F_i)=F_i\cap \{p_2\geq \epsilon\}\) for the other facets of \(P\).
  The map
  \begin{align*}
    g_2:F_1\times [0,1]&\rightarrow P \cap\{p_2\leq \epsilon\}\\
    (x,y)&\mapsto x(1-y)+y g_1(x) 
  \end{align*}
  is a face preserving homeomorphism with \(p_2\circ g_2(x,y)=\epsilon y\) for all \((x,y)\in F_1\times [0,1]\).
  Now let 
  \begin{align*}
    \hat{P}&=P\times \Delta^{l_1+1}\cap \{p_1=p_2\} \subset \R^{n-l_1}\times \R^{l_1+1},\\
    \hat{P}_1&=P\times \Delta^{l_1+1}\cap \{p_1=p_2\geq\epsilon\} \subset \R^{n-l_1}\times \R^{l_1+1},\\
    \hat{P}_2&=P\times \Delta^{l_1+1}\cap \{p_1=p_2\leq\epsilon\} \subset \R^{n-l_1}\times \R^{l_1+1}.
  \end{align*}
  
  Then there are face preserving homeomorphisms
  \begin{equation*}
    h_1: P\times \Delta^{l_1} \rightarrow \hat{P}_1\quad\quad (x,y)\mapsto(g_1(x),p_2(g_1(x))y)
  \end{equation*}
  and
  \begin{equation*}
    h_2: F_1\times \Delta^{l_1+1}\rightarrow \hat{P}_2 \quad\quad (x,y)\mapsto (g_2(x,p_1(y)),\epsilon y).
  \end{equation*}
  We claim that \(\hat{P}\) and \(M/T\) are face preserving homeomorphic.
  This is the case if
  \begin{equation*}
    \hat{f}^{-1}\circ h_1^{-1} \circ h_2: F_1\times \Delta^{l_1} \rightarrow F_1\times \Delta^{l_1}
  \end{equation*}
  extends to a face preserving homeomorphism of \(F_1\times \Delta^{l_1+1}\).
  Now for \((x,y)\in F_1\times \Delta^{l_1}\)
  we have
  \begin{align*}
    \hat{f}^{-1}\circ h_1^{-1}\circ h_2(x,y)&=\hat{f}^{-1}\circ h_1^{-1}(g_2(x,p_1(y)),\epsilon y)\\
    &= \hat{f}^{-1}\circ h_1^{-1}(g_2(x,1),\epsilon y)\\
    &= \hat{f}^{-1}(g_1^{-1}\circ g_2(x,1),y)\\
    &= (\hat{f}_1^{-1}(x), (\hat{f}_2(x, \cdot))^{-1}(y)).
  \end{align*}
  Because \(\Delta^{l_1+1}\) is the cone over \(\Delta^{l_1}\) this map extends to a face preserving homeomorphism of \(F_1\times \Delta^{l_1+1}\).
\end{proof}

\begin{lemma}
\label{sec:case-g_1-=-5}
    Let \(\tilde{G}=G_1\times G_2\) with \(G_1=SU(l_1+1)\), \(M\) a torus manifold with \(\tilde{G}\)-action and \((\psi,N,A)\) the admissible triple for \((\tilde{G},G_1)\) corresponding to \(M\).
  Then there is an isomorphism \(\pi_1(N)\rightarrow \pi_1(M)\).
\end{lemma}
\begin{proof}
  Let \(\tilde{M}\) be the blow up of \(M\) along \(A\). Then, by \cite[p. 270]{0567.53031}, there is a isomorphism  \(\pi_1(\tilde{M})\rightarrow \pi_1(M)\).
  
  Now, by Corollary~\ref{cor:hhh}, \(\tilde{M}\) is the total space of a fiber bundle over \(\C P^{l_1}\) with fiber \(N\).
  Therefore there is an exact sequence
  \begin{equation*}
    \pi_2(\tilde{M})\rightarrow \pi_2(\C P^{l_1}) \rightarrow \pi_1(N) \rightarrow \pi_1(\tilde{M})\rightarrow 0.
  \end{equation*}
  Because the torus action on \(N\) has fixed points, there is a section in this bundle.
  Hence, \( \pi_2(\tilde{M})\rightarrow \pi_2(\C P^{l_1})\) is surjective.
\end{proof}

\section{The case $G_1=SO(2l_1)$}
\label{sec:so2l}

In this section we study torus manifolds with \(G\)-action, where \(\tilde{G}=G_1\times G_2\) and \(G_1=SO(2l_1)\) is elementary.
It turns out that the restriction of the action of \(G_1\) to \(U(l_1)\) on such a manifold has the same orbits as the action of \(SO(2l_1)\).
Therefore the results of the previous section may be applied to construct invariants for such manifolds.
For simply connected torus manifolds with \(G\)-action these invariants determine their \(\tilde{G}\)-equivariant diffeomorphism type.

Let \(\tilde{G}=G_1\times G_2\), where \(G_1= SO(2l_1)\) is elementary,  and \(M\) a torus manifold with \(G\)-action.
Then, by Lemmas \ref{lem:iso} and \ref{lem:iso2}, one sees that the principal orbit type of the \(G_1\)-action is given by \(SO(2l_1)/SO(2l_1-1)\).
Therefore the \(G_1\)-action has only three orbit types \(SO(2l_1)/SO(2l_1-1)\), \(SO(2l_1)/S(O(2l_1-1)\times O(1))\) and \(SO(2l_1)/SO(2l_1)\).
The induced action of \(U(l_1)\) has the same orbits, which are of type \(U(l_1)/U(l_1-1)\), \(U(l_1)/\langle U(l_1-1),\mathbb{Z}_2\rangle\) and  \(U(l_1)/U(l_1)\), respectively.
Here \(\langle U(l_1-1),\mathbb{Z}_2\rangle\) denotes the subgroup of \(U(l_1)\), which is generated by \(U(l_1-1)\) and the diagonal matrix with all entries equal to \(-1\).

Let \(S=S^1\). Then there is a finite covering
\begin{align*}
  SU(l_1)\times S &\rightarrow U(l_1) & (A,s)&\mapsto sA.
\end{align*}
So we may replace the factor \(G_1\) of \(\tilde{G}\) by \(SU(l_1)\) and \(G_2\) by \(S\times G_2\) to reach the situation of the previous section.

Let \(x \in M^T\) and \(T_2=T\cap G_2\). Then we may assume by Lemma~\ref{lem:iso2} that the \(G_1\times T_2\)-representation \(T_xM\) is given by
\begin{equation*}
  T_xM=V\oplus W,
\end{equation*}
where \(V\) is a complex representation of \(T_2\) and \(W\) is the standard real representation of \(G_1\).
Therefore
\begin{equation*}
  T_xM=V\oplus V_0\otimes_\C W_0
\end{equation*}
as a \(SU(l_1)\times S\times T_2\)-representation, where \(V_0\) is the standard complex one-dimensional representation of \(S\) and \(W_0\) is the standard complex representation of \(SU(l_1)\).

Therefore the group homomorphism \(\psi_1\) and the groups \(H_0,H_1,H_2\) introduced in Lemma~\ref{lem:poly} have the following form:
\begin{equation*}
  \image \psi_1=S,
\end{equation*}
and
\begin{align*}
  H_0 &=SU(l_1)\times S,\\
  H_1 &=S(U(l_1-1)\times U(1))\times S,\\
  H_2 &=\left\{(g,g_{l_1+1}^{-1})\in H_1; g=
\left( 
 \begin{matrix}
    A&0\\
    0&g_{l_1+1}
  \end{matrix}
\right)
\text{ with } A \in U(l_1-1)
\right\}.
\end{align*}

Let \(N_1\) be the intersection of \(l_1-1\) characteristic submanifolds of \(M\) belonging to \(\mathfrak{F}_1\) as defined in Lemmas~\ref{sec:case-g_1-=} and \ref{sec:case-g_1-=-2}.
Then, by Lemma~\ref{lem:poly}, we know that \(N_1\) is a component of \(M^{H_2}\) and \(M=H_0N_1\).
Therefore we have \(N_1=M^{H_2}\) if, for all \(H_0\)-orbits \(O\), \(O^{H_2}\) is connected.
Because all orbits are of type \(H_0/H_0\), \(H_0/H_2\), \(H_0/\langle H_2,\mathbb{Z}_2\rangle\) and
\begin{gather*}
  \left(H_0/H_2\right)^{H_2}=N_{H_0}H_2/H_2=H_1/H_2,\\
   \left(H_0/\langle H_2, \mathbb{Z}_2\rangle \right)^{H_2}=N_{H_0}H_2/\langle H_2, \mathbb{Z}_2\rangle=H_1/\langle H_2, \mathbb{Z}_2\rangle,    
\end{gather*}
it follows that \(N_1=M^{H_2}\).

The projection \(H_1\rightarrow H_1/H_2\) induces an isomorphism \(S\rightarrow H_1/H_2\).
Therefore \(S\) acts freely on \(\left(H_0/H_2\right)^{H_2}\).
Hence, \(S\) acts effectively on \(N_1\).

By Corollary~\ref{cor:10}, \(N_1^S=M^{H_0}\) has codimension two in \(N_1\).

After these general remarks we first discuss the case, where there are no exceptional \(SO(2l_1)\)-orbits. That means the case, where there are no orbits of type \(SO(2l_1)/S(O(2l_1-1)\times O(1))\).
Then the induced \(U(l_1)\)-action has also no exceptional orbits.
Moreover, by Corollary~\ref{cor:10}, \(M\) is a special \(SO(2l_1)\)-, \(U(l_1)\)-manifold in the sense of J\"anich \cite{0153.53703}.

At first we discuss the question under which conditions the action of \(U(l_1)\times G_2\) on a torus manifold satisfying the above conditions on the \(U(l_1)\)-orbits and having no exceptional \(U(l_1)\)-orbits extends to an action of \(SO(2l_1)\times G_2\).

Let \(X\) be the orbit space of the \(U(l_1)\)-action on \(M\).
Then, by \cite[p. 303]{0153.53703}, \(X\) is a manifold with boundary such that the interior \(\mathring{X}\) of \(X\) corresponds to orbits of type \(U(l_1)/U(l_1-1)\) and the boundary \(\partial X\) to the fixed points. The action of \(G_2\) on \(M\) induces a natural action of \(G_2\) on \(X\).

Following J\"anich \cite{0153.53703}  we may construct from \(M\) a manifold \(M\odot M^{U(l_1)}\) with boundary, on which \(U(l_1)\times G_2\) acts such that all orbits of the \(U(l_1)\)-action on 
 \(M\odot M^{U(l_1)}\) are of type \(U(l_1)/U(l_1-1)\) and \(\left(M\odot M^{U(l_1)}\right)/U(l_1)=X\).
Denote by \(P_M\) the \(G_2\)-equivariant principal \(S^1\)-bundle
\begin{equation*}
  \left(M\odot M^{U(l_1)}\right)^{U(l_1-1)}\rightarrow X.
\end{equation*}

\begin{lemma}
  Let \(M\) be a torus manifold with \(U(l_1)\times G_2\)-action such that all \(U(l_1)\)-orbits are of type \(U(l_1)/U(l_1-1)\) or  \(U(l_1)/U(l_1)\).
  Then the action of \(U(l_1)\times G_2\) on \(M\) extends to an action of \(SO(2l_1)\times G_2\) if and only if there is a \(G_2\)-equivariant \(\mathbb{Z}_2\)-principal bundle \(P_M'\) such that
  \begin{equation*}
    P_M=S^1\times_{\mathbb{Z}_2}P_M',
  \end{equation*}
where the action of \(G_2\) on \(S^1\) is trivial.
\end{lemma}
\begin{proof}
  If the action extends to a \(SO(2l_1) \times G_2\)-action, then  \(SO(2l_1) \times G_2\) acts on  \(M\odot M^{U(l_1)}\).
  Therefore \(P_M'=\left(M\odot M^{U(l_1)}\right)^{SO(2l_1-1)}\rightarrow X\) is such a \(G_2\)-equivariant  \(\mathbb{Z}_2\)-principal bundle.
  
  If there is such a \(G_2\)-equivariant \(\mathbb{Z}_2\)-bundle \(P_M'\), then by a \(G_2\)-equivariant version of J\"anich's Klassifikationssatz \cite{0153.53703} there is a torus manifold  \(M'\) with \(SO(2l_1)\times G_2\)-action with \(M'/U(l_1)=X\) and \(P_M=S^1\times_{\mathbb{Z}_2}P_M'=P_{M'}\).
  Therefore \(M'\) and \(M\) are \(U(l_1)\times G_2\)-equivariantly diffeomorphic.
\end{proof}

\begin{lemma}
\label{sec:case-g_1=so2l_1-1}
  Let \(M,M'\) be torus manifolds with \(SO(2l_1)\times G_2\)-action such that there are no exceptional \(SO(2l_1)\)-orbits and \(H_1(M;\mathbb{Z})\) and \(H_1(M';\mathbb{Z})\) are torsion.
  If there is a \(U(l_1)\times G_2\)-equivariant diffeomorphism \(f:M\rightarrow M'\), then there is a \(SO(2l_1)\times G_2\)-equivariant diffeomorphism \(g:M\rightarrow M'\).
Moreover, \(g\) and \(f\) induce the same map on \(M/U(l_1)-B\), where \(B\) is a collar of \(\partial (M/U(l_1))\).
\end{lemma}
\begin{proof}
 The map \(f\) induces a \(G_2\)-equivariant diffeomorphism \(\hat{f}:X=M/SO(2l_1)\rightarrow M'/SO(2l_1)\). We use this map to identify these spaces.
  It follows from \cite[p. 91]{0246.57017} and the equality \(H_1(X;\mathbb{Z})=\pi_1(X)/[\pi_1(X),\pi_1(X)]\) that \(H_1(X;\mathbb{Z})\) is torsion.
  Hence, \(H^1(X;\mathbb{Z})=0\).

  Recall that for the universal principal \(\mathbb{Z}_2\)-bundle \(P\rightarrow \R P^\infty\), the first Chern-class of the principal \(S^1\)-bundle \(S^1\times_{\mathbb{Z}_2}P\rightarrow \R P^\infty\) is given by \(\delta w_1(P)\), where \(\delta:H^1(\R P^\infty;\mathbb{Z}_2)\rightarrow H^2(\R P^\infty;\mathbb{Z})\) is the Bockstein-homomorphism and \(w_1(P)\) is the first Stiefel-Whitney-class of \(P\). 
  By naturallity, this relation also holds for any principal \(\mathbb{Z}_2\)-bundle over \(X\).
  Because \(H^1(X;\mathbb{Z})=0\), the Bockstein-homomorphism \(\delta: H^1(X;\mathbb{Z}_2)\rightarrow H^2(X;\mathbb{Z})\) is injective.

  Hence, the principal \(S^1\)-bundle \(P_M\rightarrow X\) has up to isomorphism at most one restriction of structure group to \(\mathbb{Z}_2\).
  Therefore the two restrictions of the structure group induced by the \(SO(2l_1)\)-actions on \(M,M'\) are the same up to \(G_2\)-equivariant isomorphism.

  Therefore, by the proof of J\"anich's Klassifikationssatz, there is a \(SO(2l_1)\times G_2\)-equivariant diffeomorphism \(g:M\rightarrow M'\), which induces the same map as \(f\) outside a neighbourhood of \(\partial X\).
\end{proof}

Now we turn to the case where there are exceptional \(SO(2l_1)\)-orbits.
Then we have:

\begin{theorem}
\label{sec:case-g_1=so2l_1-2}
Let \(M,M'\) be two simply connected torus manifolds with \(SO(2l_1)\times G_2\)-action. Then \(M\) and \(M'\) are \(SO(2l_1)\times G_2\)-equivariantly diffeomorphic if and only if they are \(U(l_1)\times G_2\)-equivariantly diffeomorphic.
\end{theorem}
\begin{proof}
  In this proof we take all cohomology groups with coefficients in \(\mathbb{Z}\).
  Let \(f:M\rightarrow M'\) be a \(U(l_1)\times G_2\)-equivariant diffeomorphism.
  Moreover, let \(A,A'\) be the union of the exceptional \(U(l_1)\)-orbits in \(M,M'\), respectively.
  Because the \(U(l_1)\)-representation \(N_x(M^{U(l_1)},M)\) is the standard representation for all \(x\in M^{U(l_1)}\), there  are invariant neighbourhoods of \(M^{U(l_1)}\) and \(M'^{U(l_1)}\) which do not contain any exceptional orbit. Hence, \(A,A'\) are closed submanifolds of \(M,M'\).

Denote by \(D,D'\) the unit disc bundle in \(N(A,M)\) and \(N(A',M')\), respectively. Let \(h:D\rightarrow B\subset M\) and \(h':D'\rightarrow B'\subset M'\) be \(SO(2l_1)\times G_2\)-equivariant tubular neighbourhoods of \(A\) and \(A'\).

Then, by Theorems 4.6 and 8.3 of \cite[p. 10,19]{pre05136053} , we may assume that \(f(B)=B'\) and that \(h'^{-1}\circ f\circ h\) is a linear map.

It is sufficient to show the following two things:
\begin{enumerate}
\item\label{item:14} There is a \(SO(2l_1)\times G_2\)-equivariant diffeomorphism \(g:M-\mathring{B}\rightarrow M'-\mathring{B}'\) such that \(g\) and \(f\) induce the same maps on \((\partial B)/U(l_1)\).
\item\label{item:16} The map \(g\) extends to an \(SO(2l_1)\times G_2\)-equivariant diffeomorphism \(M\rightarrow M'\).
\end{enumerate}

If \(H_1(M-\mathring{B})\) is torsion, we may apply the arguments from the proof of Lemma \ref{sec:case-g_1=so2l_1-1} to show (\ref{item:14}).
Therefore we show that \(H_1(M-\mathring{B})\) is torsion.

Let \(A_1,\dots,A_k\) be the orientable components of \(A\) of codimension two in \(M\).
We fix orientations for each of these components and for \(M\).
Let \(\tau_1,\dots,\tau_k\in H^2(M)\) be the Poincar\'e duals for \(A_1,\dots,A_k\).
Because \(H_1(M)=0\), it follows from an universal coefficient theorem and Poincar\'e-duality that
\begin{equation*}
H^2(M)\cong \Hom(H_2(M),\mathbb{Z}) \cong \Hom(H^{2n-2}(M),\mathbb{Z}),
\end{equation*}
where an isomorphism is given by
\begin{equation*}
  \alpha \mapsto (\beta \mapsto \langle \beta \alpha,[M]\rangle). 
\end{equation*}
Here we have \(\dim M=2n\). In particular, \(H^2(M)\) is torsion free.

We claim that the \(\tau_1,\dots,\tau_k\) are linear independent.
Let \(a_1,\dots,a_k\in \mathbb{Z}\) such that
\begin{equation}
  \label{eq:16}
  0=\sum_{i=1}^k a_i \tau_i.
\end{equation}
Then we have \(0=a_i\iota_{A_i}^*\tau_i\) where \(\iota_{A_i}:A_i\rightarrow M\) is the inclusion.
By restricting to an orbit \(O\) contained in \(A_i\), we get
\begin{equation*}
  0=a_i\iota_O^*\iota_{A_i}^*\tau_i \in H^2(SO(2l_1)/S(O(2l_1-1)\times O(1)))=\mathbb{Z}_2.
\end{equation*}
 Because \(N(A_i,M)|_O= SO(2l_1)/SO(2l_1-1)\times_{\mathbb{Z}_2}\R^2\) with \(\mathbb{Z}_2\) acting on \(\R^2\) by multiplication with \(-1\), it follows that \(\iota_O^*\iota_{A_i}^*\tau_i\neq 0\).
Therefore \(a_i\) is divisible by two.

Hence, we may replace \(a_i\mapsto \frac{1}{2}a_i\) in (\ref{eq:16}).
Since the above arguments then hold for the new \(a_i\), we see that the original \(a_i\) are divisible by arbitrary high powers of two.
Therefore they must vanish.

There is an exact sequence
\begin{equation*}
  H^{2n-2}(M)\rightarrow H^{2n-2}(A)\rightarrow H^{2n-1}(M,A)\rightarrow 0.
\end{equation*}
Because, by \cite[p. 185]{0246.57017}, there are no components of \(A\), which have codimension one in \(M\), there is an isomorphism
\begin{equation*}
  H^{2n-2}(A)\cong\mathbb{Z}^k\oplus (\mathbb{Z}_2)^{k_1},
\end{equation*}
where \(k_1\) is the number of non-orientable components of codimension two of \(A\).
Let
\begin{align*}
  \phi:H^{2n-2}(A)&\rightarrow \mathbb{Z}^k\\
  \alpha&\mapsto (\langle\alpha,[A_1]\rangle,\dots,\langle\alpha,[A_k]\rangle).
\end{align*}
Because the \(\tau_1,\dots,\tau_k\) are linear independent, it follows that
\(\phi\circ \iota^*: H^{2n-2}(M)\rightarrow \mathbb{Z}^k\) has rank \(k\).

Therefore, from the exactness of the above sequence, it follows that \(H^{2n-1}(M,A)\) is torsion.
By Poincar\'e-duality and excision, it follows that \(H_1(M-\mathring{B})\) is torsion.
Hence we have proven (\ref{item:14}).

Now we prove (\ref{item:16}). 
By Theorem 9.4 of \cite[p. 24]{pre05136053}, it is sufficient to show that
\begin{equation*}
  k=h'^{-1}\circ g \circ h: \partial D\rightarrow \partial D'
\end{equation*}
extends to an \(SO(2l_1)\times G_2\)-equivariant diffeomorphism \(D\rightarrow D'\).

Let \(O\) be an \(SO(2l_1)\)-orbit in \(A\) and \(S\rightarrow O\) be the restriction of the sphere bundle \(\partial D\rightarrow A\) to \(O\).
Because \(f\) and \(g\) induce the same maps on the orbit space \((\partial B)/U(l_1)\) and \(S\) is \(SO(2l_1)\)-invariant, we have \(k(S)=h'^{-1}\circ f \circ h(S)=S'\).
Because \(h'^{-1}\circ f\circ h:D\rightarrow D'\) is a linear map, we see that \(S'\) is the restriction of the sphere bundle \(\partial D'\rightarrow A'\) to an \(SO(2l_1)\)-orbit \(O'\).

We may choose \(SO(2l_1)\)-equivariant bundle isomorphisms \(k_1:SO(2l_1)/SO(2l_1-1)\times_{\mathbb{Z}_2}S^m \rightarrow S\) and \(k_1':SO(2l_1)/SO(2l_1-1)\times_{\mathbb{Z}_2}S^m \rightarrow S'\).
Because  \(f\) and \(g\) induce the same maps on the orbit space \(S/SO(2l_1)=S^m/\mathbb{Z}_2=\R P^m\) and \(h'^{-1}\circ f\circ h\) is a linear map, it follows that \(k_1'^{-1}\circ g \circ k_1\) is of the form
\begin{equation*}
  [gSO(2l_1-1),x]\mapsto [gzSO(2l_1-1),\pm Ax]= [gSO(2l_1-1), \pm Ax],
\end{equation*}
where \(z\in S(O(2l_1-1)\times O(1))/SO(2l_1-1)=\mathbb{Z}_2\) and \(A\in O(m+1)\).
Therefore \(k\) is linear on each fiber. Hence, it extends to an \(SO(2l_1)\times G_2\)-equivariant diffeomorphism \(D\rightarrow D'\).
\end{proof}

Let \(M\) be a simply connected torus manifold with \(SO(2l_1)\times G_2\)-action.
By Theorem~\ref{sec:case-g_1-=-1}, there is an admissible triple \((\psi,N,A)\) corresponding to \(M\) equipped with the action of \(SU(l_1)\times S\times G_2\) as above.
The admissible triple \((\psi,N,A)\) determines the \(SU(l_1)\times S\times G_2\)-equivariant diffeomorphism type of \(M\).
With Theorem~\ref{sec:case-g_1=so2l_1-2} we see that the \(SO(2l_1)\times G_2\)-equivariant diffeomorphism type of \(M\) is determined by \((\psi,N,A)\).

\begin{lemma}
\label{sec:case-g_1=so2l_1}
  Let \(M\) be a torus manifold with \(G_1\times G_2\)-action, where \(G_1=SO(2l_1)\) is elementary and \(G_2\) is a not necessary connected Lie-group.
  If \(M^{SO(2l_1)}\) is connected then \(G_2\) acts orientation preserving on \(N(M^{SO(2l_1)},M)\).
  Therefore \(G_2\) acts orientation preserving on \(M\) if and only if it acts orientation preserving on \(M^{SO(2l_1)}\).
\end{lemma}
\begin{proof}
  Let \(g\in G_2\), \(x\in M^{SO(2l_1)}\) and \(y=gx\in M^{SO(2l_1)}\).
  Because \(M^{SO(2l_1)}\) is connected there is a orientation preserving \(SO(2l_1)\)-invariant isomorphism
  \begin{equation*}
    N_x(M^{SO(2l_1)},M)\cong N_y(M^{SO(2l_1)},M).
  \end{equation*}
  Therefore \(g: N_x(M^{SO(2l_1)},M)\rightarrow N_y(M^{SO(2l_1)},M)\) induces an automorphism \(\phi\) of the \(SO(2l_1)\)-representation \(N_x(M^{SO(2l_1)},M)\) which is orientation preserving if and only if \(g\) is orientation preserving.

  Because, by Lemma~\ref{lem:iso2}, \(N_x(M^{SO(2l_1)},M)\) is just the standard real representation of \(SO(2l_1)\), its complexification \(N_x(M^{SO(2l_1)},M)\otimes_\R \C\) is an irreducible complex representation.
  Therefore, by Schur's Lemma, there is a \(\lambda \in \C -\{0\}\) such that for all \(a\in N_x(M^{SO(2l_1)},M)\)
  \begin{equation*}
    \phi(a)\otimes 1= \phi_\C(a\otimes 1)= a\otimes \lambda.
  \end{equation*}
This equation implies that \(\lambda \in \R-\{0\}\) and \(\phi(a)=\lambda a\).
Therefore \(\phi\) is orientation preserving.
\end{proof}

\section{The case $G_1=SO(2l_1 +1)$}
\label{sec:so}

In this section we discuss actions of groups, which have a covering group, whose action on \(M\) factors through \(\tilde{G}=G_1\times G_2\) with \(G_1=SO(2l_1+1)\) elementary.
In the case \(G_1=SO(3)\) we also assume \(\#\mathfrak{F}_1=1\) or that the principal orbit type of the \(SO(3)\)-action on \(M\) is given by \(SO(3)/SO(2)\).

It is shown that a torus manifold with  \(\tilde{G}\)-action is a product of a sphere and a torus manifold with \(G_2\)-action or the blow up along the fixed points of \(G_1\) is a fiber bundle over a real projective space.

We assume that \(T_1=T\cap G_1\) is the standard maximal torus of \(G_1\).

\subsection{The $G_1$-action on $M$}
\label{sec:so:g_1onM}

\begin{Lemma}
  \label{lem:so1}
  Let \(\tilde{G}=G_1\times G_2\) with \(G_1=SO(2l_1+1)\), \(M\) a torus manifold with \(G\)-action such that \(G_1\) is elementary.
  If \(l_1>1\) there is, by Lemma~\ref{sec:g-action-m}, a component \(N_1\) of \(\bigcap_{M_i\in \mathfrak{F}_1}M_i\) with \(N_1^T\neq \emptyset\).
  If \(l_1=1\) let \(N_1\) be a characteristic submanifold belonging to \(\mathfrak{F}_1\).
Then 
\begin{enumerate}  
\item \(N_1\) is a component of \(M^{SO(2l_1)}\).
\item  \(M=G_1N_1\).
\end{enumerate}
\end{Lemma}
\begin{proof}
  Let \(x \in N_1^T\). Then, by Lemmas \ref{lem:iso},~\ref{lem:iso2} and Remark \ref{sec:g-action-m-1},  \(G_{1x}=SO(2l_1)\).
  Let \(T_2\) be the maximal torus \(T\cap G_2\) of \(G_2\).
  On the tangent space of \(M\) in \(x\) we have the \(SO(2l_1)\times T_2\)-representation
  \begin{equation*}
    T_xM=N_x(G_1x,M)\oplus T_xG_1x.
  \end{equation*}
  By Lemma~\ref{lem:iso}, \(T_2\) acts trivially on \(G_1x\).
  Moreover, \(T_2\) acts almost effectively on \(N_x(G_1x,M)\).
  Therefore it follows by dimension reasons that
 \(N_x(G_1x,M)\) splits as a sum of complex one dimensional  \(SO(2l_1)\times T_2\)-representations.
  If \(l_1>1\), \(SO(2l_1)\) has no non-trivial one-dimensional complex representation.
  Therefore we have
  \begin{equation}
    \label{eq:3}
    T_xM=\bigoplus_i V_i \oplus W,
  \end{equation}
  where the \(V_i\) are one-dimensional complex representations of \(T_2\) and \(W\) is the standard real representation of \(SO(2l_1)\).

  If \(l_1=1\) and \(\#\mathfrak{F}_1=2\), then \(SO(2l_1)\) acts trivially on \(N_x(G_1x,M)\) because \(SO(3)/SO(2)\) is the principal orbit type of the \(SO(3)\)-action on \(M\) \cite[p. 181]{0246.57017}.

  If \(l_1=1\) and \(\#\mathfrak{F}_1=1\), then, by the discussion leading to Convention~\ref{sec:g-action-m-3}, \(SO(2)\) acts trivially on \(N_x(G_1x,M)\).
  Therefore in these cases \(T_xM\) splits as in~(\ref{eq:3}).

  Because \(N_x(G_1x,M)\) is the tangent space of \(N_1\) in \(x\) the maximal torus \(T_1\) of \(G_1\) acts trivially on \(N_1\).
  Therefore \(N_1\) is the component of \(M^{T_1}\), which contains \(x\).
  Because \(T_xN_1=(T_xM)^{T_1}=(T_xM)^{SO(2l_1)}\), \(N_1\) is a component of \(M^{SO(2l_1)}\).

  Now we prove (2).
  Let \(y \in N_1\). Then there are the following possibilities:
  \begin{itemize}
  \item \(G_{1y}=G_1\).
  \item \(G_{1y}=S(O(2l_1)\times O(1))\) and  \(\dim G_1y=2l_1\).
  \item \(G_{1y}=SO(2l_1)\) and \(\dim G_1y=2l_1\).
  \end{itemize}
  If \(g \in G_1\) such that \(gy\in N_1\), then 
  \begin{equation*}
    gG_{1y}g^{-1}=G_{1gy}\in \{S(O(2l_1)\times O(1)),SO(2l_1),G_1\}
  \end{equation*}
  and
  \begin{equation*}
    g \in N_{G_1}G_{1y}=
    \begin{cases}
      G_1& \text{if }y \in M^{G_1}\\
      S(O(2l_1)\times O(1))& \text{if }y \not\in M^{G_1}.\\
    \end{cases}
  \end{equation*}
  Therefore \(G_1y\cap N_1\subset S(O(2l_1)\times O(1))y\) contains at most two elements. 
  If \(y\) is not fixed by \(G_1\), then one sees as in the proof of Lemma~\ref{lem:poly} that \(G_1y\) and \(N_1\) intersect transversely in \(y\).

  Therefore \(G_1(N_1-N_1^{G_1})\) is open in \(M-M^{G_1}\) by Lemma~\ref{sec:lie-groups-2}.
  Because \(M^{G_1}\) has codimension at least three, \(M-M^{G_1}\) is connected.
  But 
  \begin{equation*}
    G_1\left(N_1-N_1^{G_1}\right)=G_1N_1\cap \left(M-M^{G_1}\right)
  \end{equation*}
  is also closed in \(M-M^{G_1}\).
  Hence,
  \begin{equation*}
  M-M^{G_1}= G_1\left(N_1-N_1^{G_1}\right) = G_1N_1 -N_1^{G_1}.
\end{equation*}
Therefore one sees as in the proof of Lemma~\ref{lem:poly} that
\begin{equation*}
  M = G_1N_1\amalg \left(M^{G_1}-N_1^{G_1}\right).
\end{equation*}
Because \(G_1N_1\) and \(M^{G_1}-N_1^{G_1}\) are closed in \(M\) the statement follows.
\end{proof}

\begin{cor}
  \label{cor:so3}
  If in the situation of Lemma   \ref{lem:so1} the \(G_1\)-action on \(M\) has no fixed point in \(M\), then \(M= SO(2l_1+1)/SO(2l_1) \times N_1\) or \(M=SO(2l_1+1)/SO(2l_1) \times_{\mathbb{Z}_2} N_1\),
where \(\mathbb{Z}_2= S(O(2l_1)\times O(1)) / SO(2l_1)\).

In the second case the \(\mathbb{Z}_2\)-action on \(N_1\) is orientation reversing.

If \(l_1=1\) and \(\#\mathfrak{F}_1=1\), then we have \(M=SO(2l_1+1)/SO(2l_1) \times_{\mathbb{Z}_2} N_1\).
If \(l_1=1\) and \(\#\mathfrak{F}_1=2\), then we have \(M=SO(2l_1+1)/SO(2l_1) \times N_1\).
\end{cor}
\begin{proof}
  Let \(g\in  S(O(2l_1)\times O(1))=N_{G_1}SO(2l_1)\).
  Then \(gN_1\)
  is a component of \(M^{SO(2l_1)}\).
  Because \(N_1\subset M^{SO(2l_1)}\), \(gN_1\) only depends on the class 
  \begin{equation*}
    gSO(2l_1)\in S(O(2l_1)\times O(1))/SO(2l_1)=\mathbb{Z}_2.
  \end{equation*}
  Therefore there are two cases
  \begin{enumerate}
  \item\label{item:10} There is a \(g\in S(O(2l_1)\times O(1))\) such that \(gN_1\neq N_1\).
  \item\label{item:11} The submanifold \(N_1\) is \(S(O(2l_1)\times O(1))\)-invariant, i.e.  \(gN_1=N_1\) for all  \(g\in S(O(2l_1)\times O(1))\).
  \end{enumerate}

  If \(l_1=1\) and \(\#\mathfrak{F}_1=1 \), then \(N_1\) is the only characteristic submanifold of \(M\) belonging to \(\mathfrak{F}_1\).
  Therefore only the second case occurs.

  If \(l_1=1\) and \(\#\mathfrak{F}_1=2\), then there is a \(g_1\in N_{G_1}T_1\) such that \(N_1\neq g_1N_1\).
Therefore we are in the first case.

  In general we have \(M=G_1\times N_1/\sim\) with
  \begin{align*}
    &&(g_1,y_1)&\sim(g_2,y_2)\\
    &\Leftrightarrow& g_1y_1&=g_2y_2\\
    &\Leftrightarrow& g_2^{-1}g_1y_1&=y_2\\
    &\Leftrightarrow& g_2^{-1}g_1\in S(O(2l_1)\times O(1)) &\text{ and } g_2^{-1}g_1y_1=y_2
  \end{align*}
In  case~(\ref{item:10}) the last statement is equivalent to 
\begin{equation*}
  g_2^{-1}g_1\in SO(2l_1) \text{ and } g_2^{-1}g_1y_1=y_2.
\end{equation*}
Therefore we get \(M=SO(2l_1+1)/SO(2l_1)\times N_1\).

In case~(\ref{item:11}) we have as in the proof of Corollary \ref{cor:hhh} 
\begin{equation*}
  M=SO(2l_1+1)\times_{S(O(2l_1)\times O(1))}N_1 = SO(2l_1+1)/SO(2l_1)\times_{\mathbb{Z}_2}N_1.
\end{equation*}
This equation implies that \(M\) is the orbit space of a diagonal \(\mathbb{Z}_2\)-action on \(SO(2l_1+1)/SO(2l_1)\times N_1\).
Because \(M\) is orientable this action has to be orientation preserving.
But the \(\mathbb{Z}_2\)-action on  \(SO(2l_1+1)/SO(2l_1)\) is orientation reversing.
Therefore the \(\mathbb{Z}_2\)-action on \(N_1\) is also orientation reversing.
\end{proof}

\begin{cor}
  \label{cor:so2}
  In the situation of Lemma \ref{lem:so1}, \(M^{G_1}\subset N_1\) is empty or has codimension one in \(N_1\).
\end{cor}
\begin{proof}
  By Lemma \ref{lem:so1}, it is clear that \(M^{G_1}\subset N_1\).
  For \(y \in M^{G_1}\) consider the \(G_1\) representation \(T_yM\).
  Because \(N_1\) is a component of \(M^{SO(2l_1)}\), the restriction of \(T_yM\) to \(SO(2l_1)\) equals the \(SO(2l_1)\)-representation \(T_xM\), where \(x\in N_1^T\).

  Because, by Lemma~\ref{lem:iso2}, \(T_xM\) is a direct sum of a trivial representation and the standard real representation of \(SO(2l_1)\) and \(T_1\subset SO(2l_1)\),
  \(T_yM\) is a sum of a trivial and the standard real representation of \(SO(2l_1+1)\) by \cite[p. 167]{0581.22009}.
  Therefore
  \(M^{G_1}\subset N_1\) has codimension one.
\end{proof}

\subsection{Blowing up along $M^{G_1}$}
\label{sec:so:blow_along}
As in section \ref{sec:su} we discuss the question when a manifold of the form given in Corollary \ref{cor:so3} is a blow up.

If \(\tilde{M}\) is the blow up of \(M\) along \(M^{G_1}\), then there is an equivariant embedding of \(P_{\mathbb{R}}(N(M^{G_1},M))\) into \(\tilde{M}\).
Therefore the \(G_1\)-action on \(\tilde{M}\) has an orbit of type \(SO(2l_1+1)/S(O(2l_1)\times O(1))\).
This fact shows that \(\tilde{M}\) is of the form \(SO(2l_1+1)/SO(2l_1) \times_{\mathbb{Z}_2} \tilde{N}_1\) where \(\tilde{N}_1\) is the proper transform of \(N_1\).
By Lemma \ref{cor:blowing-up}, \(\tilde{N}_1\) and \(N_1\) are \(G_2\)-equivariantly diffeomorphic.
Because \(M^{G_1}\) has codimension one in \(N_1\),  the \(\mathbb{Z}_2\)-action on \(N_1\) has a fixed point component of codimension one.

The following Lemma shows that these two conditions are sufficient.

\begin{Lemma}
\label{sec:case-g_1=so2l_1-+1-1}
Let \(N_1\) be a torus manifold with \(G_2\)-action.
Assume that there are a non-trivial orientation reversing action of \(\mathbb{Z}_2=S((O(2l_1)\times O(1))/SO(2l_1))\) on \(N_1\), which commutes with the action of \(G_2\), and a closed codimension one submanifold \(A\) of \(N_1\), on which \(\mathbb{Z}_2\) acts trivially.

Let \(E'=N(A,N_1)\) equipped with the action of \(G_2\) induced from the action on \(N_1\) and the trivial action of \(\mathbb{Z}_2\).
Denote by \(W\) the standard real representation of \(SO(2l_1+1)\).
Then:
\begin{enumerate}
\item  \(SO(2l_1+1)/SO(2l_1)\times_{\mathbb{Z}_2}N_1\) is orientable.
\item  The normal bundle of \(SO(2l_1+1)/S(O(2l_1)\times O(1))\times A\) in \(SO(2l_1+1)/SO(2l_1)\times_{\mathbb{Z}_2}N_1\) is isomorphic to the tautological bundle over \(P_\R(E'\otimes W\oplus \{0\})\).
\end{enumerate}
\end{Lemma}

The lemma guarantees together with the discussion at the end of section~\ref{sec:blow} that one may remove \(SO(2l_1+1)/S(O(2l_1)\times O(1))\times A\) from  \(SO(2l_1+1)/SO(2l_1)\times_{\mathbb{Z}_2}N_1\) and replace it by \(A\) to get a torus manifold with \(G\)-action \(M\) such that \(M^{SO(2l_1+1)}=A\).
The blow up of \(M\) along \(A\) is \(SO(2l_1+1)/S(O(2l_1))\times_{\mathbb{Z}_2}N_1\).

\begin{proof}
  The diagonal \(\mathbb{Z}_2\)-action on \(SO(2l_1+1)/SO(2l_1)\times N_1\) is orientation preserving.
  Therefore \(SO(2l_1+1)/SO(2l_1)\times_{\mathbb{Z}_2}N_1\) is orientable.

  The normal bundle of  \(SO(2l_1+1)/S(O(2l_1)\times O(1)) \times A\) in \(SO(2l_1+1)/SO(2l_1)\times_{\mathbb{Z}_2}N_1\) is given by \(SO(2l_1+1)/SO(2l_1)\times_{\mathbb{Z}_2}N(A,N_1)\).
 
Consider the following commutative diagram
\begin{equation*}
  \xymatrix{
    SO(2l_1+1)/SO(2l_1)\times_{\mathbb{Z}_2}N(A,N_1) \ar[r]_(.58)f\ar[d]_{\pi_1} &   P_{\mathbb{R}}(E'\otimes W)\times E'\otimes W \ar[d]_{\pi_1}\\
    SO(2l_1+1)/S(O(2l_1)\times O(1)) \times A   \ar[r]_(.66)g  &   P_{\mathbb{R}}(E'\otimes W)
}
\end{equation*}
where the vertical maps are the natural projections and  \(f,g\) are given by
\begin{equation*}
  f([hSO(2l_1):m])=([m\otimes he_1],m\otimes he_1)
\end{equation*}
 and
 \begin{equation*}
   g(hS(O(2l_1)\times O(1)),q)=[m_q\otimes he_1],
 \end{equation*}
 where \(e_1\in W-\{0\}\) is fixed such that for all \(g' \in SO(2l_1)\), \(g'e_1=e_1\) and \(m_q\neq 0\) some element of the fiber of \(E'\) over \(q\).

The map \(f\) induces an isomorphism of the normal bundle of \(SO(2l_1)/S(O(2l_1)\times O(1))\times A\) in \(SO(2l_1)/SO(2l_1)\times_{\mathbb{Z}_2} N_1\) and the tautological bundle over  \(P_\R(E'\otimes W\oplus \{0\})\).
\end{proof}

\begin{lemma}
\label{sec:case-g_1=so2l_1-+1}
  If \(l_1>1\), in the situation of Lemma~\ref{lem:so1}, then \(\bigcap_{M_i\in \mathfrak{F}_1}M_i=M^{SO(2l_1)}\) has at most two components.
  It has two components if and only if \(M=S^{2l_1}\times N_1\).
\end{lemma}
\begin{proof}
  If \(M=S^{2l_1}\times N_1\), then \(\bigcap_{M_i\in \mathfrak{F}_1}M_i=\{N,S\}\times N_1\), where \(N,S\) are the north and the south pole of the sphere, respectively.
  Otherwise the blow up of \(M\) along \(M^{SO(2l_1+1)}\) is given by \(S^{2l_1}\times_{\mathbb{Z}_2}N_1\), which is a fiber bundle over \(\R P^{2l_1}\).
  The characteristic submanifolds of \(S^{2l_1}\times_{\mathbb{Z}_2}N_1\), which are permuted by \(W(G_1)\), are given by the preimages of the following submanifolds of \(\R P^{2l_1}\):
  \begin{equation*}
    \R P^{2l_1-2}_i=\{(x_1:x_2:\dots:x_{2i-2}:0:0:x_{2i+1}:\dots:x_{2l_1+1})\in \R P^{2l_1}\},\;\;i=1,\dots,l_1.
  \end{equation*}
  These characteristic submanifolds are also given by the proper transforms \(\tilde{M}_i\) of the characteristic submanifolds \(M_i\in \mathfrak{F}_1\) of \(M\). 
  Because
  \begin{equation*}
    \bigcap_{i=1}^{l_1} \R P^{2l_1-2}_i=\{(0:0:\dots:0:1)\},
  \end{equation*}
  it follows that
  \begin{equation*}
    \bigcap_{M_i\in \mathfrak{F}_1}\tilde{M}_i=\tilde{N}_1=\tilde{M}^{SO(2l_1)}.
  \end{equation*}
  Therefore, with Lemma~\ref{lem:proper} and Corollary~\ref{cor:so2},
  \begin{equation*}
    \bigcap_{M_i\in \mathfrak{F}_1}M_i=N_1=M^{SO(2l_1)}
  \end{equation*}
follows.
In particular \(\bigcap_{M_i\in \mathfrak{F}_1}M_i\) is connected.
\end{proof}

\begin{lemma}
  \label{sec:case-g_1=so2l_1-+1-3}
  If \(l_1=1\), in the situation of Lemma~\ref{lem:so1}, then the following statements are equivalent:
  \begin{itemize}
  \item \(M^{SO(2)}\) has two components.
  \item \(\#\mathfrak{F}_1=2\).
  \item \(M=S^2\times N_1\).
  \end{itemize}
If \(l_1=1\) and \(\#\mathfrak{F}_1=1\), then \(M^{SO(2)}\) is connected.
\end{lemma}
\begin{proof}
  At first we prove that all components of \(M^{SO(2)}\) are characteristic submanifolds of \(M\) belonging to \(\mathfrak{F}_1\).
  By Lemma~\ref{lem:so1}, \(N_1\) is a characteristic submanifold of \(M\) and a component of \(M^{SO(2)}\) such that \(G_1N_1=M\).
  Therefore, if \(x\in M^{SO(2)}\), then there is a \(g\in N_{G_1}SO(2)\) such that \(g^{-1}x\in N_1\).
  This implies \(x\in gN_1\).
  Because \(gN_1\) is a characteristic submanifold belonging to \(\mathfrak{F}_1\) and a component of \(M^{SO(2)}\) it follows that \(M^{SO(2)}\) is a union of characteristic submanifolds of \(M\) belonging to \(\mathfrak{F}_1\).

  Now assume that \(\#\mathfrak{F}_1=1\).
  Then we have \(M^{SO(2)}=N_1\). 
  Therefore \(M^{SO(2)}\) is connected.

  Now assume that \(M=SO(3)/SO(2)\times N_1\).
  Then it is clear that \(M^{SO(2)}\) has two components.

  Now assume that \(M^{SO(2)}\) has two components.
  Because these components are characteristic submanifolds belonging to \(\mathfrak{F}_1\) it follows that \(\#\mathfrak{F}_1=2\).

  Now assume that \(\#\mathfrak{F}_1=2\).
  If there is no \(G_1\)-fixed point then it follows from Corollary~\ref{cor:so3} that \(M=SO(3)/SO(2)\times N_1\).
  Assume that there is a \(G_1\)-fixed point in \(M\). Then the blow up of \(M\) along \(M^{G_1}\) contains an orbit of type \(SO(3)/S(O(2)\times O(1))\).
  Now Corollary~\ref{cor:so3} implies \(\#\mathfrak{F}_1=1\).
  Therefore there is no \(G_1\)-fixed point if \(\#\mathfrak{F}_1=2\).
\end{proof}

\subsection{Admissible pairs}
\label{sec:admissible_pairs}
We are now in the position to state another classification theorem.
To do so, we use the following definition.
\begin{definition}
  Let \(\tilde{G}=G_1\times G_2\) with \(G_1=SO(2l_1+1)\). Then a pair \((N,A)\) with
  \begin{itemize}
  \item \(N\) a torus manifold with \(G_2\times \mathbb{Z}_2\)-action such that the \(\mathbb{Z}_2\)-action is orientation-reversing or trivial,
  \item \(A\subset N\) the empty set or a closed \(G_2\times \mathbb{Z}_2\)-invariant submanifold of codimension one, on which \(\mathbb{Z}_2\) acts trivially, such that if \(A\neq \emptyset\), then \(\mathbb{Z}_2\) acts non-trivially on \(N\),
  \end{itemize}
is called \emph{admissible for} \((\tilde{G},G_1)\).

We say that two admissible pairs \((N,A)\), \((N',A')\) are equivalent if there is a \(G_2\times \mathbb{Z}_2\)-equivariant diffeomorphism \(\phi:N\rightarrow N'\) such that \(\phi(A)=A'\).
\end{definition}

\begin{theorem}
\label{sec:case-g_1=so2l_1-+1-2}
  Let \(\tilde{G}=G_1\times G_2\) with \(G_1=SO(2l_1+1)\).
  There is a one-to-one correspondence between the \(\tilde{G}\)-equivariant diffeomorphism classes of torus manifolds with \(\tilde{G}\)-actions such that \(G_1\) is elementary and equivalence classes of admissible pairs for \((\tilde{G},G_1)\).
\end{theorem}
\begin{proof}
  Let \(M\) be a torus manifold with \(\tilde{G}\)-action.
  If \(\bigcap_{M_i\in\mathfrak{F}_1}M_i\) has two components and \(l_1>1\) or \(\#\mathfrak{F}_1=2\) and \(l_1=1\), then we assign to \(M\) the admissible pair \(\Phi(M)=(N_1,\emptyset)\), where \(N_1\) is a component of \(\bigcap_{M_i\in\mathfrak{F}_1}M_i\) or a characteristic submanifold belonging to \(\mathfrak{F}_1\) in the case \(l_1=1\).
  The action of \(\mathbb{Z}_2\) is trivial in this case.

  If \(\bigcap_{M_i\in\mathfrak{F}_1}M_i\) is connected and \(l_1>1\) or \(\#\mathfrak{F}_1=1\) and \(l_1=1\), then we assign to \(M\) the pair 
  \begin{equation*}
    \Phi(M)=\left(\bigcap_{M_i\in\mathfrak{F}_1}M_i,M^{SO(2l_1+1)}\right).
  \end{equation*}
  Because \(\bigcap_{M_i\in\mathfrak{F}_1}M_i=M^{SO(2l_1)}\) there is a non-trivial action of 
  \begin{equation*}
    \mathbb{Z}_2=S(O(2l_1)\times O(1))/SO(2l_1)
  \end{equation*}
 on \(\bigcap_{M_i\in\mathfrak{F}_1}M_i\).

  Now let \((N,A)\) be a admissible pair for \((\tilde{G},G_1)\).
  If the \(\mathbb{Z}_2\)-action on \(N\) is trivial, we have \(A= \emptyset\) and we assign to \((N,\emptyset)\) the torus manifold with \(\tilde{G}\)-action \(\Psi((N,\emptyset))=S^{2l_1}\times N\).
  
  If the \(\mathbb{Z}_2\)-action on \(N\) is non-trivial, we assign to \((N,A)\) the blow down \(\Psi((N,A))\) of \(SO(2l_1+1)/SO(2l_1)\times_{\mathbb{Z}_2}N\) along \(SO(2l_1+1)/S(O(2l_1)\times O(1))\times A\).

  By Lemma~\ref{sec:case-g_1=so2l_1-+1}, it is clear that this construction  gives a one-to-one correspondence between torus manifolds with \(\tilde{G}\)-action such that \(\bigcap_{M_i\in\mathfrak{F}_1}M_i\) has two components and \(l_1>1\) and admissible pairs with trivial \(\mathbb{Z}_2\)-action.
  With Lemma~\ref{sec:case-g_1=so2l_1-+1-3}, we see that an analogous statement holds for \(l_1=1\) and \(\#\mathfrak{F}_1=2\).

 Now let \((N,A)\) be an admissible pair such that \(\mathbb{Z}_2\) acts non-trivially on \(N\).
 Then the discussion after Lemma~\ref{sec:case-g_1=so2l_1-+1-1} shows that \(\Phi(\Psi((N,A)))\) is equivalent to \((N,A)\).

 If \(M\) is a torus manifold with \(G_1\times G_2\)-action such that \(G_1\) is elementary and \(N_1=\bigcap_{M_i\in\mathfrak{F}_1}M_i\) is connected the blow up of \(M\) along \(M^{SO(2l_1+1)}\) is given by
 \begin{equation*}
   SO(2l_1+1)/SO(2l_1)\times_{\mathbb{Z}_2}N_1.
 \end{equation*}
Therefore we find that \(\Psi(\Phi(M))\) is equivariantly diffeomorphic to \(M\).
\end{proof}

\section{Classification}
\label{sec:classif}

Here we use the results of the previous sections to state a classification of torus manifolds with \(G\)-action.
We do not consider actions of groups, which have \(SO(2l_1)\) as an elementary factor, because as explained in section~\ref{sec:so2l} these factors may be replaced by \(SU(l_1)\times S^1\).
We get the classification by iterating the constructions given in Theorem~\ref{sec:case-g_1-=-1} and Theorem~\ref{sec:case-g_1=so2l_1-+1-2}.

We illustrate this iteration in the case that all elementary factors of \(G\) are isomorphic to \(SU(l_i+1)\).
Let \(\tilde{G}=\prod_{i=1}^k G_i\times T^{l_0}\) and \(M\) a torus manifold with \(\tilde{G}\)-action such that all \(G_i\) are elementary and isomorphic to \(SU(l_i+1)\).

In Theorem~\ref{sec:case-g_1-=-1} we constructed a triple \((\psi_1,N_1,A_1)\), which determines the \(\tilde{G}\)-equivariant diffeomorphism type of \(M\).
Here \(N_1\) is a  torus manifold with \(\prod_{i=2}^k G_i \times T^{l_0}\)-action.
Therefore there is a triple \((\psi_2,N_2,A_2)\) which determines the \(\prod_{i=2}^kG_i \times T^{l_0}\)-equivariant diffeomorphism type of \(N_1\).
Because \(N_2\subset N_1\) such that \(G_2N_2=N_1\) and \(A_1\) is \(G_2\)-invariant we have \(G_2(A_1\cap N_2)=A_1\).
Therefore the  \(G\)-equivariant diffeomorphism type of \(M\) is determined by
\begin{equation*}
  (\psi_1\times \psi_2,N_2,A_1\cap N_2,A_2).
\end{equation*}
Continuing in this manner leads to a triple
\begin{equation*}
  (\psi,N,(A_1,\dots,A_k)),
\end{equation*}
where \(\psi\in \Hom\left(\prod_{i=1}^kS(U(l_i)\times U(1)),T^{l_0}\right)\), \(N\) is a \(2l_0\)-dimensional torus manifold and the \(A_i\) are codimension two submanifolds of \(N\) or empty.

The iteration becomes more complicated if there are more than one elementary factors of \(\tilde{G}\) isomorphic to \(SO(2l_i+1)\).
To illustrate what happens here, we discuss the case \(\tilde{G}=G_1\times G_2\times T^{l_0}\), where the \(G_i\) are elementary and isomorphic to \(SO(2l_i+1)\).

Then, by Theorem~\ref{sec:case-g_1=so2l_1-+1-2}, there is an admissible pair \((N_1,B_1)\) for \((\tilde{G},G_1)\) corresponding to \(M\), where \(N_1\) is a torus manifold with \(G_2\times T^{l_0}\times (\mathbb{Z}_2)_1\)-action.
By Lemmas~\ref{sec:case-g_1=so2l_1-+1} and~\ref{sec:case-g_1=so2l_1-+1-3}, we have two cases
\begin{enumerate}
\item \(N_1^{SO(2l_2)}\) has two components.
\item \(N_1^{SO(2l_2)}\) is connected.
\end{enumerate}

In the first case we have
\begin{equation*}
  N_1=SO(2l_2+1)/SO(2l_2)\times N_2,
\end{equation*}
where \(N_2\) is a \(2l_0\)-dimensional torus manifold.
The action of \((\mathbb{Z}_2)_1\) on \(N_1\) commutes with the action of \(G_2\times T^{l_0}\).
Therefore the action of \((\mathbb{Z}_2)_1\) on  \(N_1\) splits as an product of an action on \(SO(2l_2+1)/SO(2l_2)\) and an action on \(N_2\).
Because there is only one non-trivial action of \(\mathbb{Z}_2\) on \(SO(2l_2+1)/SO(2l_2)\) which commutes with the action of \(SO(2l_2+1)\), the \(G_2\times T^{l_0}\times (\mathbb{Z}_2)_1\)-equivariant diffeomorphism type of \(N_1\) is completely determined by a pair \((N_2,a_{12})\), where \(N_2\) is equipped with the action of \(T^{l_0}\times(\mathbb{Z}_2)_1\) and \(a_{12}\in \{0,1\}\) is non-zero if and only if the \((\mathbb{Z}_2)_1\)-action on  \(SO(2l_2+1)/SO(2l_2)\) is non-trivial.

In the second case the \(G_2\times T^{l_0}\)-equivariant diffeomorphism type of \(N_1\) is determined by a pair \((N_2,B_2)\), where \(N_2=N_1^{SO(2l_2)}\).
Because  \(N_2\) is \((\mathbb{Z}_2)_1\)-invariant in this case, \(N_2\) is a torus manifold with \(T^{l_0}\times (\mathbb{Z}_2)_1\times (\mathbb{Z}_2)_2\)-action, where \((\mathbb{Z}_2)_2=S(O(2l_2)\times O(1))/SO(2l_2)\).
We put \(a_{12}=0\) in this case.

As in the case where there are only elementary factors isomorphic to \(SU(l_i+1)\), one sees that the \(G_1\times G_2\times T^{l_0}\)-equivariant diffeomorphism type of \(M\) is determined by
\begin{equation*}
  (N_2,(N_2\cap B_1,B_2),a_{12}).
\end{equation*}

There are some relations between \(a_{12}\) and \(B_1\).
For example, if \(a_{12}=1\), then there are no \((\mathbb{Z}_2)_1\)-fixed points in \(N_1\).
Therefore \(B_1\) has to be empty.

If there are more than two elementary factors of \(\tilde{G}\) isomorphic to \(SO(2l_i+1)\), we have to introduce more numbers \(a_{ij}\).
There are some relations between the \(a_{ij}\) coming from the fact that \(M\) is required to be orientable.
This will be explained in the proof of Lemma~\ref{sec:classification}.

\subsection{Admissible 5-tuples}
\label{sec:class:ad5tup}
We use the following definition to make the above constructions more formal.

\begin{definition}
\label{sec:classification-2}
  Let \(\tilde{G}=\prod_{i=1}^k G_i \times G'\) with
  \begin{equation*}
    G_i=
    \begin{cases}
      SU(l_i+1)&\text{if } i \leq k_0\\
      SO(2l_i+1)& \text{if } i > k_0
    \end{cases}
  \end{equation*}
  and \(k_0\in \{0,\dots,k\}\). Then a \(5\)-tuple 
  \begin{equation*}
    (\psi,N,(A_i)_{i=1,\dots,k_0},(B_i)_{i=k_0+1,\dots,k},(a_{ij})_{k_0+1\leq i < j\leq k})
  \end{equation*}
with
  \begin{enumerate}
  \item \(\psi\in \Hom(\prod_{i=1}^{k_0}S(U(l_i)\times U(1)),Z(G'))\) and \(\psi_i=\psi|_{S(U(l_i)\times U(1))}\),
  \item \(N\) a torus manifold with \(G'\times\prod_{i=k_0+1}^k(\mathbb{Z}_2)_i\)-action,
  \item \label{item:9}\(A_i \subset N \) the empty set or a \(G'\times\prod_{i=k_0+1}^k(\mathbb{Z}_2)_i\)-invariant closed submanifold of codimension two, on which \(\image \psi_i\) acts trivially, such that if \(A_i \neq \emptyset\), then \(\ker \psi_i=SU(l_i)\),
  \item \label{item:8} \(B_i \subset N\) the empty set or a \(G'\times\prod_{i=k_0+1}^k(\mathbb{Z}_2)_i\)-invariant closed submanifold of codimension one, on which \((\mathbb{Z}_2)_i\) acts trivially, such that if \(B_i\neq \emptyset\), then the action of \((\mathbb{Z}_2)_i\) on \(N\) is non-trivial,
  \item \label{item:1} \(a_{ij}\in \{0,1\}\) such that 
    \begin{enumerate}
    \item \label{item:2}if \(a_{ij}=1\), then
      \begin{enumerate}
      \item \label{item:3}the action of \((\mathbb{Z}_2)_j\) on \(N\) is trivial,
      \item \label{item:4}\(a_{jk}=0\) for \(k>j\),
      \item \label{item:5}\(B_i = \emptyset\),
      \end{enumerate}
    \item \label{item:6}if the action of \((\mathbb{Z}_2)_i\) on \(N\) is non-trivial, then it is orientation preserving if and only if \(\sum_{j>i}a_{ij}\) is odd,
    \item \label{item:7}if the action of \((\mathbb{Z}_2)_i\) on \(N\) is trivial, then  \(\sum_{j>i}a_{ij}\) is odd or zero,
    \end{enumerate}   
  \end{enumerate}
is called \emph{admissible for} \((\tilde{G},\prod_{i=1}^k G_i)\) if the \(A_i\) and \(B_i\) intersect pairwise transversely.

If \(G'\) is a torus we also say that a \(5\)-tuple is admissible for \(\tilde{G}\) instead of \((\tilde{G},\prod_{i=1}^k G_i)\).

We say that two admissible \(5\)-tuples
\begin{gather*}
  (\psi,N,(A_i)_{i=1,\dots,k_0},(B_i)_{i=k_0+1,\dots,k},(a_{ij})_{k_0+1\leq i < j\leq k})\\\intertext{and}(\psi',N',(A_i')_{i=1,\dots,k_0},(B_i')_{i=k_0+1,\dots,k},(a_{ij}')_{k_0+1\leq i < j\leq k})
\end{gather*}
 are equivalent if 
\begin{itemize}
\item \(\psi_i=\psi_i'\) if \(l_i >1\) and  \(\psi_i=\psi_i'^{\pm 1}\) if \(l_i=1\),
\item \(a_{ij}=a_{ij}'\),
\item there is a  \(G'\times\prod_{i=k_0+1}^k(\mathbb{Z}_2)_i\)-equivariant diffeomorphism \(\phi:N\rightarrow N'\) such that
  \(\phi(A_i)=A_i'\) and \(\phi(B_i)=B_i'\).
\end{itemize}
\end{definition}

\begin{remark}
  By Lemma~\ref{lem:torus-m2}, two submanifolds \(A_1,A_2\) of \(N\) satisfying the condition (3) intersect transversely if and only if no component of \(A_1\) is a component of \(A_2\).
  
  By Lemma~\ref{sec:gener-torus-manif-1}, two submanifolds \(A_1,B_1\) of \(N\) satisfying the conditions (3) and (4), respectively, intersect always transversely.

  By Lemma~\ref{sec:gener-torus-manif-2}, two submanifolds \(B_1,B_2\) of \(N\) satisfying the condition (4) intersect transversely if and only if no component of \(B_1\) is a component of \(B_2\).
\end{remark}

\begin{lemma}
\label{sec:classification}
  Let \(\tilde{G}\) as above. Then there is a one-to-one correspondence between the equivalence classes of admissible \(5\)-tuples
  \begin{equation*}
    (\psi,N,(A_i)_{i=1,\dots,k_0},(B_i)_{i=k_0+1,\dots,k},(a_{ij})_{k_0+1\leq i < j\leq k})
  \end{equation*}
 for \((\tilde{G},\prod_{i=1}^k G_i)\) and the equivalence classes of admissible \(5\)-tuples 
  \begin{equation*}
    (\psi',N',(A'_i)_{i=1,\dots,k_0},(B'_i)_{i=k_0+1,\dots,k-1},(a'_{ij})_{k_0+1\leq i < j\leq k-1})
  \end{equation*}
for \((\tilde{G},\prod_{i=1}^{k-1} G_i)\) such that \(G_k\) is elementary for the \(G_k\times G'\)-action on \(N'\).
\end{lemma}
\begin{proof}
  At first assume that \(G_k=SU(l_k+1)\).
  Let \((\psi,N,(A_i)_{i=1,\dots,k-1},\emptyset,\emptyset)\) be an admissible \(5\)-tuple for \((\tilde{G},\prod_{i=1}^{k-1} G_i)\) such that \(G_k\) is elementary for the \(G_k\times G'\)-action on \(N\).

  Let \((\psi_k,N_k,A_k)\) be the admissible triple for \((G_k\times G',G_k)\), which corresponds to \(N\) under the correspondence given in  
  Theorem~\ref{sec:case-g_1-=-1}. Then \(N_k\) is a submanifold of \(N\).
  By Lemma~\ref{lem:torus-m2}, \(A_i\), \(i=1,\dots,k-1\), intersects \(N_k\) transversely.
  Therefore \(N_k\cap A_i\) has codimension \(2\) in \(N_k\).
  Because \(A_i=G_k(N_k \cap A_i)\),  \(N_k\cap A_i\) has no component, which is contained in \(A_k\) or \(N_k\cap A_j\), \(j\neq i\).
  Therefore by 
  \begin{equation*}
    (\psi\times\psi_k,N_k,(A_1\cap N_k,\dots,A_{k-1}\cap N_k,A_k),\emptyset,\emptyset)
  \end{equation*}
 an admissible \(5\)-tuple for \((\tilde{G},\prod_{i=1}^k G_i)\) is given.

Now let
  \begin{equation*}
    (\psi\times\psi_k,N_k,(A_1,\dots,A_k),\emptyset,\emptyset)
  \end{equation*}
  be an admissible \(5\)-tuple for \((\tilde{G},\prod_{i=1}^kG_i)\).
  Let \(H_0=G_k\times \image \psi_k\) and \(H_1=S(U(l_k)\times U(1))\times \image\psi_k\).
  Then, by Lemma~\ref{lem:class1}, the blow down \(N\)  of \(\tilde{N}=H_0\times_{H_1}N_k\) along \(H_0/H_1\times A_k\) is a torus manifold with \(G_k\times G'\)-action.
  By Lemma~\ref{lem:proper}, \(F(H_0\times_{H_1}A_i)=G_kF(A_i)\), \(i<k\), are submanifolds of \(N\) satisfying the condition (3) of Definition~\ref{sec:classification-2}.
  Because \(F(A_i)\) and \(F(A_j)\), \(i<j<k\),  have no components in common, \(G_kF(A_i)\) and \(G_kF(A_j)\) intersect transversely.
  Therefore by
  \begin{equation*}
       (\psi,N,(G_k F(A_1),\dots,G_k F(A_{k-1})),\emptyset,\emptyset)
  \end{equation*}
  an admissible triple for \((\tilde{G},\prod_{i=1}^{k-1}G_i)\) is given.

  As in the proof of Theorem~\ref{sec:case-g_1-=-1} one sees that this construction leads to a one-to-one-correspondence.
  
  Now assume that \(G_k=SO(2l_k+1)\).
  Let 
  \begin{equation}
    \label{eq:4}
    (\psi,N,(A_i)_{i=1,\dots,k_0},(B_i)_{i=k_0+1,\dots,k-1},(a_{ij})_{k_0+1\leq i < j\leq k-1})
  \end{equation}
  be an admissible \(5\)-tuple for \((\tilde{G},\prod_{i=1}^{k-1} G_i)\) such that \(G_k\) is elementary for the \(G_k\times G'\)-action on \(N\).

  At first assume that, for the \(G_k\)-action on \(N\), \(N^{SO(2l_k)}\) is connected.
  Let \((N_k,B_k)\) be the admissible pair for \((G_k\times G',G_k)\)  which corresponds to \(N\) under the correspondence given in  
  Theorem~\ref{sec:case-g_1=so2l_1-+1-2}.
  Then \(N_k\) is a submanifold of \(N\) which is invariant under the action of \(G'\times\prod_{i=k_0+1}^{k}(\mathbb{Z}_2)_i\), where \((\mathbb{Z}_2)_k=S(O(2l_k)\times O(1))/SO(2l_k)\).
  For \(i<k\), let \(a_{ik}=0\).
  
  We claim that by
  \begin{equation}
    \label{eq:5}
    (\psi,N_k,(A_1\cap N_k,\dots,A_{k_0}\cap N_k),(B_{k_0+1}\cap N_k,\dots,B_{k-1}\cap N_k,B_k),(a_{ij}))
  \end{equation}
  an admissible \(5\)-tuple for \((\tilde{G},\prod_{i=1}^k G_i)\) is given.

  At first note that, for \(i=1,\dots,k-1\), the \(A_i\) and \(B_i\) intersect \(N_k\) transversely by Lemmas~\ref{lem:torus-m2} and~\ref{sec:gener-torus-manif-1}.
  Therefore \(A_i\cap N_k\) and \(B_i\cap N_k\) has codimension two or one, respectively, in \(N_k\).

  One sees as in the case \(G_k=SU(l_k+1)\) that the \(N_k\cap A_i\) and \(N_k\cap B_i\) intersect pairwise transversely.

  Now we verify the condition (\ref{item:1}) of Definition~\ref{sec:classification-2} for the \(5\)-tuple~(\ref{eq:5}).
  By Lemma~\ref{sec:case-g_1=so2l_1}, \((\mathbb{Z}_2)_i\), \(i<k\), acts orientation preserving on \(N\) if and only if it acts orientation preserving on \(N_k\). This proves (\ref{item:6}) because (\ref{eq:4}) is an admissible \(5\)-tuple and \(a_{ik}=0\).

  Because, by Lemma~\ref{lem:so1}, \(G_kN_k=N\), \((\mathbb{Z}_2)_i\), \(i<k\), acts trivially on \(N_k\) if and only if it acts trivially on \(N\).
  This proves (\ref{item:7}) and (\ref{item:3}) because (\ref{item:7}) and (\ref{item:3}) hold for the admissible \(5\)-tuple (\ref{eq:4}) and \(a_{ik}=0\).

  Because \(a_{ik}=0\), (\ref{item:4}) and (\ref{item:5}) are clear.

  Now assume that  \(N^{SO(2l_k)}\) is non-connected.
  Then, by Lemmas~\ref{sec:case-g_1=so2l_1-+1} and~\ref{sec:case-g_1=so2l_1-+1-3}, we have
  \begin{equation*}
    N=SO(2l_k+1)/SO(2l_k)\times N_k.
  \end{equation*}
  In this case the \((\mathbb{Z}_2)_i\)-action, \(i<k\), on \(N\) commutes with the action of \(SO(2l_k+1)\).
  Therefore it splits in a product of an action on \(SO(2l_k+1)/SO(2l_k)\) and an action on \(N_k\).
  We put \(a_{ik}=1\) if the \((\mathbb{Z}_2)_i\)-action on  \(SO(2l_k+1)/SO(2l_k)\) is non-trivial and \(a_{ik}=0\) otherwise.
  Because there is only one non-trivial action of \(\mathbb{Z}_2\) on  \(SO(2l_k+1)/SO(2l_k)\), which commutes with the action of \(SO(2l_k+1)\), we may recover the action of \((\mathbb{Z}_2)_i\) on \(N\) from the action on \(N_k\) and \(a_{ik}\).

 We identify \(SO(2l_k)/SO(2l_k)\times N_k\) with \(N_k\) and equip it with the trivial action of \((\mathbb{Z}_2)_k=S(O(2l_k)\times O(1))/SO(2l_k)\).
  We claim that by
  \begin{equation}
    \label{eq:6}
      (\psi,N_k,(A_1\cap N_k,\dots,A_{k_0}\cap N_k),
      (B_{k_0+1}\cap N_k,\dots,B_{k-1}\cap N_k,\emptyset),(a_{ij}))
  \end{equation}
  an admissible \(5\)-tuple for \((\tilde{G},\prod_{i=1}^k G_i)\) is given.
  
  The conditions~(\ref{item:9}) and~(\ref{item:8}) of Definition~\ref{sec:classification-2} and the transversality condition are verified as in the previous cases.

  Therefore we only have to verify condition~(\ref{item:1}).
  Because the non-trivial \(\mathbb{Z}_2\)-action on \(SO(2l_k+1)/SO(2l_k)\) is orientation reversing, the \((\mathbb{Z}_2)_i\)-action, \(i<k\) on \(N_k\) has the same orientation behavior as the action on \(N\) if and only if the \((\mathbb{Z}_2)_i\)-action on \(SO(2l_k+1)/SO(2l_k)\) is trivial.
  By the definition of \(a_{ik}\), this is the case if and only if \(a_{ik}=0\).
  Therefore (\ref{item:6}) follows because (\ref{eq:4}) is an admissible \(5\)-tuple and \((\mathbb{Z}_2)_k\) acts trivially on \(N_k\).

  If the \((\mathbb{Z}_2)_i\)-action on \(N_k\) is trivial and non-trivial on \(SO(2l_k+1)/SO(2l_k)\), then the  \((\mathbb{Z}_2)_i\)-action on \(N\) is orientation reversing.
  Therefore \(\sum_{j>i}a_{ij}\) is odd.

  The \((\mathbb{Z}_2)_i\)-actions on \(N_k\) and  \(SO(2l_k+1)/SO(2l_k)\) are trivial if and only if the  \((\mathbb{Z}_2)_i\)-action on \(N\) is trivial.
  Therefore \(\sum_{j>i}a_{ij}\) is odd or trivial. This verifies (\ref{item:7}).

  If there is a \(j<i\) such that \(a_{ji}=1\), then \((\mathbb{Z}_2)_i\) acts trivially on \(N\) because the admissible \(5\)-tuple (\ref{eq:4}) satisfies (\ref{item:3}).
  Therefore \(a_{ik}=0\). This proves~(\ref{item:4}).

  If the \((\mathbb{Z}_2)_i\)-action on  \(SO(2l_k+1)/SO(2l_k)\) is non-trivial the action on \(N\) has no fixed points.
  Therefore \(B_i=\emptyset\). This proves~(\ref{item:5}). The property (\ref{item:3}) is clear.

  Now let 
 \begin{equation*}
    (\psi,N_k,(A_1,\dots,A_{k_0}),(B_{k_0+1},\dots,B_{k}),(a_{ij}))
  \end{equation*}
  be an admissible \(5\)-tuple for \((\tilde{G},\prod_{i=1}^k G_i)\).
  At first assume that \((\mathbb{Z}_2)_k\) acts non-trivially on \(N_k\).
  Then the blow down \(N\) of \(\tilde{N}=SO(2l_k+1)/SO(2l_k)\times_{(\mathbb{Z}_2)_k}N_k\) along \(SO(2l_k+1)/SO(2l_k)\times_{(\mathbb{Z}_2)_k}B_k\) is a torus manifold with \(G_k\times G' \times \prod_{i=k_0+1}^{k-1}(\mathbb{Z}_2)_i\)-action.
  As in the case \(G_k=SU(l_k+1)\) one sees that
  \begin{equation*}
    (\psi,N,(G_kF(A_1),\dots,G_kF(A_{k_0})),(G_kF(B_{k_0+1}),\dots,G_{k-1}F(B_{k-1})),(a_{ij}))
  \end{equation*}
  is an admissible \(5\)-tuple for  \((\tilde{G},\prod_{i=1}^{k-1} G_i)\).

  If  \((\mathbb{Z}_2)_k\) acts trivially on \(N_k\), then put
  \begin{equation*}
    N=SO(2l_k+1)/SO(2l_k)\times N_k.
  \end{equation*}
  Here \((\mathbb{Z}_2)_i\), \(i<k\), acts by the product action of the non-trivial \(\mathbb{Z}_2\)-action on \(SO(2l_k+1)/SO(2l_k)\) and the action on \(N_k\) if \(a_{ik}=1\).
  Otherwise  \((\mathbb{Z}_2)_i\) acts by the product action of the trivial action on \(SO(2l_k+1)/SO(2l_k)\) and the action on \(N_k\).
  Now by
  \begin{multline*}
    (\psi,N,(SO(2l_k+1)/SO(2l_k)\times A_1),\dots,SO(2l_k+1)/SO(2l_k)\times A_{k_0}),\\
    (SO(2l_k+1)/SO(2l_k)\times B_{k_0+1},\dots, SO(2l_k+1)/SO(2l_k)\times B_{k-1}),(a_{ij}))
  \end{multline*}
  an admissible  \(5\)-tuple for  \((\tilde{G},\prod_{i=1}^{k-1} G_i)\) is given.

  As in the proof of Theorem~\ref{sec:case-g_1=so2l_1-+1-2} one sees that this construction leads to a one-to-one-correspondence.
\end{proof}

Let \(\tilde{G}=\prod_i G_i \times T^{l_0}\) and 
\begin{equation*}
  (\psi,M,(A_i),(B_i),(a_{ij}))
\end{equation*}
be an admissible \(5\)-tuple for \((\tilde{G},\prod_{i=1}^{k-1}G_i)\) such that \(G_k\) is an elementary factor of \(\prod_{i\geq k}G_i\times T^{l_0}\) for the action on \(M\).
Furthermore, let
\begin{equation*}
   (\psi',N,(A_i'),(B_i'),(a_{ij}'))
\end{equation*}
be the admissible \(5\)-tuple for \((\tilde{G},\prod_{i=1}^{k}G_i)\) corresponding to  \((\psi,M,(A_i),(B_i),(a_{ij}))\).
Then the following lemma shows that \(G_i\), \(i>k\), is an elementary factor of \(\prod_{i\geq k} G_i\times T^{l_0}\) for the action on \(M\) if and only if it is an elementary factor of \(\prod_{i\geq k+1} G_i\times T^{l_0}\) for the action on \(N\).

\begin{lemma}
  \label{sec:action-weyl-group-2}
  Let \(\tilde{G}=G_1\times G' \times G''\), \(M\) a torus manifold with \(\tilde{G}\)-action and \(N\) a component of an intersection of characteristic submanifolds of \(M\), which is \(G_1\times G'\)-invariant and contains a \(T\)-fixed point \(x\) such that \(G_1\) acts non-trivially on \(N\).
  Furthermore, assume that \(G''\) is a product of elementary factors for the action on \(M\).

  Then \(N\) is a torus manifold with \(G_1\times G'\times T^{l_0}\)-action for some \(l_0\geq 0\) and \(G_1\) is an elementary factor of \(\tilde{G}\), with respect to the action on \(M\), if and only if it is an elementary factor of \(G_1\times G'\times T^{l_0}\), with respect to the action on \(N\).
\end{lemma}
\begin{proof}
  Assume that \(G_1\) is an elementary factor for one of the two actions on \(M\) and \(N\).
  Then \(G_1\) is isomorphic to a simple group or \(\text{Spin}(4)\).
  If \(G_1\) is simple and not isomorphic to \(SU(2)\) then the statement is clear.

  Therefore there are two cases \(G_1=SU(2),\text{Spin}(4)\).
  
  If \(x\) is not fixed by \(G_1\), then \(G_1=SU(2)\) is elementary for both actions on \(N\) and \(M\) by Lemma~\ref{lem:iso}.
  Therefore we may assume that \(x\in N^{G_1}\subset M^{G_1}\).
  Then there is a bijection
  \begin{equation*}
    \mathfrak{F}_{xM} \rightarrow \mathfrak{F}_{xN}\amalg \mathfrak{F}_N^\perp,
  \end{equation*}
  where
  \begin{align*}
    \mathfrak{F}_{xM}&=\{\text{characteristic submanifolds of } M \text{ containing }x\},\\
    \mathfrak{F}_{xN}&=\{\text{characteristic submanifolds of } N \text{ containing }x\},\\
    \mathfrak{F}_N^\perp&=\{\text{characteristic submanifolds of } M \text{ containing }N\}.
  \end{align*}
  This bijection is compatible with the actions of the Weyl-group of \(G_x\).

  At first assume that \(G_1=SU(2)\) is elementary for the action on \(M\) but not for the action on \(N\).
  Then there is another simple factor \(G_2=SU(2)\) of \(G_1\times G' \times T^{l_0}\) such that \(G_1\times G_2\) is elementary for the action on \(N\).
  At first assume that \(G_2\) is elementary for the action on \(M\).

  Let \(w_i\in W(G_i)\), \(i=1,2\), be generators.
  Then there are two non-trivial \(W(G_1\times G_2)\)-orbits \(\mathfrak{F}_{1},\mathfrak{F}_2\) in \(\mathfrak{F}_{xM}\).
  We have:
  \begin{itemize}
  \item \(\#\mathfrak{F}_i = 2\), \(i=1,2\),
  \item \(w_i\), \(i=1,2\), acts non-trivially on \(\mathfrak{F}_i\) and trivially on the other orbit.
  \end{itemize}

  But because, \(G_1\times G_2\) is elementary for the action on \(N\), there is exactly one non-trivial \(W(G_1\times G_2)\)-orbit \(\mathfrak{F}_1'\) in \(\mathfrak{F}_{xN}\).
  We have:
  \begin{itemize}
  \item \(\#\mathfrak{F}_1' = 2\),
  \item \(w_i\), \(i=1,2\), acts non-trivially on \(\mathfrak{F}_1'\).
  \end{itemize}
  This is a contradiction.

  If \(G_2\) is not elementary, then \(G_2\) is a simple factor of an elementary factor.
  In this case the action of \(W(G_1\times G_2)\) on \(\mathfrak{F}_{xM}\) behaves as in the  first case.
  Therefore we also get a contradiction in this case.

  Under the assumption that \(G_1=\text{Spin}(4)\) is elementary for the action on \(M\) a similar argument shows that \(G_1\) is elementary for the action on \(N\).
  
  Therefore \(G_1\) is elementary for the action on \(N\) if it is elementary for the action on \(M\).
  
  If \(G_1\) is elementary for the action on \(N\) but not elementary for the action on \(M\), then it is a simple factor of an elementary factor \(G_1'\neq G_1\) of \(\tilde{G}\) or a product  \(G_2'\times G_3'\) of elementary factors \(G_2'\) and \(G_3'\) of \(\tilde{G}\).
  But because \(G''\) is a product of elementary factors, it contains all elementary factors of \(\tilde{G}\) which have non-trivial intersection with \(G''\).
 Because \(G_1\) is not contained in \(G''\), it follows that \(G_1',G_2'\) and \(G_3'\) are subgroups of \(G_1\times G'\).
  Therefore, by the above argument, \(G_1'\) or \(G_2'\) and \(G_3'\) are elementary for the action on \(N\).
  Because elementary factors can not contain each other we get a contradiction to the assumption that \(G_1\) is elementary for the action on \(N\).
\end{proof}

Recall from section~\ref{sec:Gact} that if \(M\) is a torus manifold with \(G\)-action, then we may assume that all elementary factors of \(G\) are isomorphic to \(SU(l_i+1)\), \(SO(2l_i+1)\) or \(SO(2l_i)\).
That means \(\tilde{G}=\prod SU(l_i+1)\times \prod SO(2l_i+1)\times \prod SO(2l_i) \times T^{l_0}\).
Because, as described in section~\ref{sec:so2l}, we may replace elementary factors isomorphic to \(SO(2l_i)\) by \(SU(l_i)\times S^1\), the following theorem may be used to construct invariants of torus manifolds with \(\tilde{G}\)-action.
By Theorem~\ref{sec:case-g_1=so2l_1-2} these invariants determine the \(\tilde{G}\)-equivariant diffeomorphism type of simply connected torus manifolds with \(\tilde{G}\)-action.

\begin{theorem}
\label{thm:class1}
   Let \(\tilde{G}=\prod_{i=1}^k G_i \times T^{l_0}\) with
  \begin{equation*}
    G_i=
    \begin{cases}
      SU(l_i+1)&\text{if } i \leq k_0\\
      SO(2l_i+1)& \text{if } i > k_0
    \end{cases}
  \end{equation*}
  and \(k_0\in \{0,\dots,k\}\). Then there is a one-to-one correspondence between the equivalence classes of admissible \(5\)-tuples for \(\tilde{G}\) and the \(\tilde{G}\)-equivariant diffeomorphism classes of torus manifolds with \(\tilde{G}\)-action such that all \(G_i\) are elementary.
\end{theorem}
\begin{proof}
  This follows from Lemma~\ref{sec:classification} and Lemma~\ref{sec:action-weyl-group-2} by induction.
\end{proof}

Using Lemma~\ref{sec:action-weyl-group-1} and Theorem \ref{sec:case-g_1-=-4} we get the following result for quasitoric manifolds.

\begin{theorem}
     Let \(\tilde{G}=\prod_{i=1}^k G_i \times T^{l_0}\) with
    \(G_i= SU(l_i+1)\).
   Then there is a one-to-one correspondence between the equivalence classes of admissible \(5\)-tuples for \(\tilde{G}\) of the form
   \begin{equation*}
     (\psi,N,(A_i)_{1\leq i \leq k},\emptyset, \emptyset)
   \end{equation*}
   with \(N\) quasitoric and \(A_i\), \(1\leq i \leq k\), connected
 and the \(\tilde{G}\)-equivariant diffeomorphism classes of quasitoric manifolds with \(\tilde{G}\)-action.
\end{theorem}

\begin{remark}
  Remark~\ref{sec:action-weyl-group-3} and Theorem~\ref{sec:case-g_1-=-3} lead to a similar result for torus manifolds with \(G\)-actions whose cohomologies are generated by their degree two parts.
\end{remark}

\begin{cor}
\label{sec:classification-3}
  Let \(\tilde{G}=\prod_{i=1}^{k_1} G_i \times T^{l_0}\) with \(G_i\) elementary and \(M\) a torus manifold with \(G\)-action. Then \(M/G\) has dimension \(l_0+\#\{G_i ;\; G_i=SO(2l_i)\}\).
\end{cor}
\begin{proof}
 At first we discuss the case, where  all elementary factors of \(\tilde{G}\) are isomorphic to \(SO(2l_i+1)\) or \(SU(l_i+1)\), i.e. \(\#\{G_i ;\; G_i=SO(2l_i)\}=0\).
  By Lemma \ref{sec:blowing-up},
  replacing \(M\) by the blow up \(\tilde{M}\) of \(M\) along the fixed points of \(G_1\) does not change the orbit space.
  Therefore, by Corollaries~\ref{cor:hhh} and \ref{cor:so3}, we have up to finite coverings
  \begin{align*}
    M/G &=(M/G_1)/(\prod_{i\geq2}G_i \times T^{l_0})= (\tilde{M}/G_1)/(\prod_{i\geq2}G_i \times T^{l_0})\\
    & =((H_0\times_{H_1}N_1)/G_1)/(\prod_{i\geq2}G_i \times T^{l_0}) = N_1/(\prod_{i\geq2}G_i \times T^{l_0}),
  \end{align*}
where \(N_1\) is the \(\prod_{i\geq 2}G_i\times T^{l_0}\)-manifold from the admissible \(5\)-tuple for \((\tilde{G},G_1)\) corresponding to \(M\).
Here \(H_0,H_1\) are defined as in Lemma~\ref{lem:poly} if \(G_1=SU(l_1+1)\).
If \(G_1=SO(2l_1+1)\), we have \(H_0=SO(2l_1+1)\) and \(H_1=S(O(2l_1)\times O(1))\).

 By iterating this argument we find that \(M/G=N/T^{l_0}\) up to finite coverings, where \(N\) is the \(T^{l_0}\)-manifold from the admissible \(5\)-tuple for \(\tilde{G}\) corresponding to \(M\).

Now we study the case \(l_0'=\#\{G_i ;\; G_i=SO(2l_i)\}\neq 0\).
As discussed in section~\ref{sec:so2l}, the orbits of the \(G\)-action on \(M\) do not change if we replace an elementary factor isomorphic to \(SO(2l_i)\) by \(SU(l_i)\times S^1\).
Therefore this replacement does not change the dimension of the orbit space.
But it increases \(l_0\) by one, and decreases \(l_0'\) by one.
Therefore the statement follows by induction on \(l_0'\).
\end{proof}

\subsection{Applications}
\label{sec:class:appli}

Now we apply our classification results to special cases.
We first discuss the case, where \(M\) is a torus manifold with \(G\)-action such that \(G\) is semi-simple and \(H^*(M;\mathbb{Z})\) is generated by its degree two part.

\begin{cor}
\label{sec:classification-5}
  If \(G\) is semi-simple and \(M\) is a torus manifold with \(G\)-action such that \(H^*(M;\mathbb{Z})\) is generated by its degree two part, then
  \begin{equation*}
    \tilde{G}=\prod_{i=1}^k SU(l_i+1)
  \end{equation*}
  and
  \begin{equation*}
    M=\prod_{i=1}^k \C P^{l_i},
  \end{equation*}
  where each \(SU(l_i+1)\) acts in the usual way on \(\C P^{l_i}\) and trivially on \(\C P^{l_j}\), \(j \neq i \).
\end{cor}
\begin{proof}
 By Lemma~\ref{sec:action-weyl-group-1} and Remark~\ref{sec:action-weyl-group-3}, all elementary factors of \(\tilde{G} \) are isomorphic to \(SU(l_i+1)\).
 Because \(G\) is semi-simple, there is only one admissible \(5\)-tuple for \(\tilde{G}\), namely \((\text{const},\text{pt},\emptyset,\emptyset,\emptyset)\).
 It corresponds to a product of complex projective spaces.
\end{proof}

Next we discuss torus manifolds \(M\) with \(G\)-action such that \(\dim M/G\leq 1\).
With Theorem \ref{thm:class1}, we recover the following two results of S. Kuroki~\cite{kuroki_pre_1_2009,kuroki_pre_2_2009}:

\begin{cor}
\label{sec:classification-4}
  Let \(M\) be a simply connected torus manifold with \(G\)-action such that \(M\) is a homogeneous \(G\)-manifold.
  Then \(M\) is a product of even-dimensional spheres and complex projective spaces.
\end{cor}
\begin{proof}
  By Corollary~\ref{sec:classification-3}, the center of \(G\) is zero-dimensional.
  Moreover, all elementary factors of \(G\) are isomorphic to \(SU(l_i+1)\) or \(SO(2l_i+1)\).
  Therefore the admissible \(5\)-tuple corresponding to \(M\) is given by
  \begin{equation*}
    \left(\text{const},\text{pt},\emptyset,\emptyset,\left(a_{ij}\right)\right),
  \end{equation*}
  where the \(a_{ij}\in\{0,1\}\) are unknown.
  In particular, no elementary factor of \(G\) has a fixed point in \(M\).
  Therefore, by Corollaries~\ref{cor:hhh} and~\ref{cor:so3}, \(M\) splits into a direct product of complex projective spaces and even dimensional spheres.
\end{proof}

\begin{cor}
\label{sec:classification-8}
  If the \(G\)-action on the simply connected torus manifold \(M\) has an orbit of codimension one, then \(M\) is the projectivication of a complex vector bundle or a sphere bundle over a product of complex projective spaces and even-dimensional spheres.
\end{cor}
\begin{proof}
  By Corollary~\ref{sec:classification-3}, we may assume that there is a covering group  \(\tilde{G}=S^1\times \prod_i G_i\) of \(G\) with \(G_i\) elementary and \(G_i= SU(l_i+1)\) or \(G_i= SO(2l_i+1)\).
  We assume that the \(G_i\) are sorted in such a way that
  \begin{itemize}
  \item \(G_i=SO(2l_i+1)\) and \(G_i\) has no fixed point in \(M\) if \(i\leq k_0\),
  \item \(G_i=SU(l_i+1)\) and \(G_i\) has no fixed point in \(M\) if \(k_0+1\leq i\leq k_1\),
  \item \(G_i=SU(l_i+1),SO(2l_i+1)\) and \(G_i\) has fixed points in \(M\) if \(i\geq k_1+1\),
  \end{itemize}
  where \(k_0\leq k_1\) are some constants.

  By Corollaries \ref{cor:hhh} and \ref{cor:so3}, we know that \(M\) is of the form
  \begin{equation*}
    M=\prod_{i=1}^{k_0}S^{2l_i}\times H_{0k_0+1}\times_{H_1{k_0+1}}\left(H_{0k_0+2}\times_{H_{1k_0+2}}\left(\dots \left(H_{0k_1}\times_{H_{1k_1}}M'\right)\dots\right)\right),
  \end{equation*}
  where
  \begin{align*}
    H_{0i}&=SU(l_i+1)\times \image \psi_i,\\
    H_{1i}&=S(U(l_i+1)\times U(1))\times \image \psi_i,
  \end{align*}
  for \(i=k_0+1,\dots,k_1\), and \(M'\) is a torus manifold with \(\tilde{G}'\)-action, where \(\tilde{G}'=\prod_{i\geq k_1+1}G_i\times S^1\).

  Because the action of \(H_{1i}\) on \(H_{0j}\), \(j>i\), is trivial and the actions of the \(H_{1i}\) on \(M'\) commute, \(M\) may be written as
  \begin{equation*}
    M=\prod_{i=1}^{k_0}S^{2l_i}\times \left(\prod_{i=k_0+1}^{k_1}H_{0i}\times_{\prod H_{1i}}M'\right).
  \end{equation*}

  Therefore \(M\) is a fiber bundle over a product of even dimensional spheres and complex projective spaces with fiber \(M'\).

  Let \((\psi,N',(A_i),(B_i),(a_{ij}))\) be the admissible \(5\)-tuple for \(\tilde{G}'\) corresponding to \(M'\). 
  Because \(\dim N'=2\) and all \(G_i\), \(i>k_1\), have fixed points in \(M'\), we have
  \begin{align*}
    N'&=S^2,&A_i&\neq \emptyset,&B_i&\neq \emptyset.
  \end{align*}
  Because the \(S^1\)-action on \(S^2\) has only two fixed points, \(N\) and \(S\), there are at most two elementary factors isomorphic to \(SU(l_i+1)\).
  The orientation reversing involutions of \(S^2\) which commute with the \(S^1\)-action and have fixed points are given by ``reflections'' at \(S^1\)-orbits.
  Therefore there is at most one elementary factor isomorphic to \(SO(2l_i+1)\).
  If there is such a factor then there is at most one \(G_i\) isomorphic to \(SU(l_i+1)\)  because \(N\) is mapped to \(S\) by such a reflection.
  Let
  \begin{align*}
    \phi_i:S(U(l_i)\times U(1))&\rightarrow U(1) & 
    \left(
    \begin{matrix}
      A & 0\\
      0 & g
    \end{matrix}
    \right)
    &\mapsto g & (A\in U(l_i), g\in U(1)).
  \end{align*}
Then we have the following admissible \(5\)-tuples:
\begin{center}
  \begin{tabular}{|c|c|c|}
    \(\tilde{G}'\)                            & \(5\)-tuple                                                              &\(M'\)\\\hline\hline
    \(S^1\)                                  & \((\emptyset,S^2,\emptyset,\emptyset,\emptyset)\)                        & \(S^2\)\\\hline
    \(S^1\times SU(l_1+1)\)                  & \((\phi_1^{\pm 1},S^2,\{N\},\emptyset,\emptyset)\)                        & \(\C P^{l_1+1}\)\\
                                             & \((\phi_1^{\pm 1},S^2,\{N,S\},\emptyset,\emptyset)\)                      & \(S^{2l_1+2}\)\\\hline
    \(S^1\times SO(2l_1+1)\)                 & \((\emptyset,S^2,\emptyset,S^1,\emptyset)\)                              & \(S^{2l_1+2}\)\\\hline
    \(S^1\times SU(l_1+1)\times SU(l_2+1)\)  & \((\phi_1^{\pm 1}\phi_2^{\pm 1}, S^2, (\{N\},\{S\}), \emptyset,\emptyset)\) & \(\C P^{l_1+l_2+1}\)\\\hline
    \(S^1\times SU(l_1+1)\times SO(2l_2+1)\) & \((\phi_1^{\pm 1},S^2,\{N,S\},S^1,\emptyset)\)                             & \(S^{2l_1+2l_2+2}\)\\
  \end{tabular}
\end{center}
Therefore the statement follows.
\end{proof}

Now we turn to the case, where \(M\) is a torus manifold with \(G\)-action such that \(G\) is semi-simple and has exactly two elementary factors \(G_1,G_2\).
We start with a discussion of the case, where \(G_1\times G_2\neq SO(2l_1)\times SO(2l_2)\).

\begin{cor}
\label{sec:classification-7}
  Let \(\tilde{G}=G_1\times G_2\neq SO(2l_1)\times SO(2l_2)\) with \(G_1\) and \(G_2\) elementary of rank \(l_1,l_2\), respectively, and \(M\) a torus manifold with \(G\)-action. Then \(M\) is one of the following:
  \begin{equation*}
    \C P^{l_1}\times \C P^{l_2}, \C P^{l_1}\times S^{2l_2}, S^{2l_1}\times S^{2l_2},  S_1^{2l_1}\times_{\mathbb{Z}_2} S_1^{2l_2},S_1^{2l_1}\times_{\mathbb{Z}_2} S_2^{2l_2},  S^{2l_1+2l_2}.
  \end{equation*}
Here \(S_1^l\) denotes the \(l\)-sphere together with the \(\mathbb{Z}_2\)-action generated by the antipodal map and \(S_2^l\) the \(l\)-sphere together with the \(\mathbb{Z}_2\)-action generated by a reflection at a hyperplane.

Furthermore, the \(\tilde{G}\)-actions on these spaces is unique up to equivariant diffeomorphism.
\end{cor}
\begin{proof}
  First assume that \(G_1,G_2\neq SO(2l)\).
  Then we have the following possibilities for the admissible \(5\)-tuple of \(M\):
  \begin{center}
     \begin{tabular}{|c|c|c|c|}
       \(G_1\)       &\(G_2\)        &\(5\)-tuple                                               &\(M\)\\\hline\hline
       \(SU(l_1+1)\) &\(SU(l_2+1)\)  &\((\text{const},\text{pt},\emptyset,\emptyset,\emptyset)\)& \(\C P^{l_1}\times \C P^{l_2}\)\\\hline
       \(SU(l_1+1)\) &\(SO(2l_2+1)\) &\((\text{const},\text{pt},\emptyset,\emptyset,\emptyset)\)& \(\C P^{l_1}\times S^{2l_2}\)\\\hline
       \(SO(2l_1+1)\)&\(SO(2l_2+1)\) &\((\emptyset,\text{pt},\emptyset,\emptyset,a_{12}=0)\)     & \(S^{2l_1}\times S^{2l_2}\)\\
                     &               &\((\emptyset,\text{pt},\emptyset,\emptyset,a_{12}=1)\)     & \(S_1^{2l_1}\times_{\mathbb{Z}_2} S_1^{2l_2}\)\\ 
     \end{tabular}
  \end{center}
  
  If \(G_1=SU(l_1+1)\) and \(G_2=SO(2l_2)\), then, by Corollary~\ref{sec:g-action-m-2}, there is one admissible triple for \((G,G_1)\) namely \((\text{const},S^{2l_2},\emptyset)\).
  It corresponds to \(\C P^{l_1}\times S^{2l_2}\).

  Now assume that  \(G_1=SO(2l_1+1)\) and \(G_2=SO(2l_2)\). Let \((N,B)\) be the admissible pair for \((G,G_1)\) corresponding to \(M\).
  Then, by Corollary~\ref{sec:g-action-m-2}, we have \(N=S^{2l_2}\).
  Up to equivariant diffeomorphism there are two orientation reversing involutions on \(S^{2l_2}\) which commute with the action of \(G_2\), the anti-podal map and a reflection at an hyperplane in \(\R^{2l_2+1}\).
  Therefore we have four possibilities for \(M\):
  \begin{equation*}
    S^{2l_1}\times S^{2l_2}, S^{2l_1+2l_2},  S_1^{2l_1}\times_{\mathbb{Z}_2} S^{2l_2}_1,  S_1^{2l_1}\times_{\mathbb{Z}_2} S^{2l_2}_2.
  \end{equation*}
\end{proof}

For the discussion of the case \(G_1\times G_2= SO(2l_1)\times SO(2l_2)\) we need the following lemma.

\begin{lemma}
\label{sec:classification-1}
  Let \(\tilde{G}=SO(2l_1)\times S^1\) and \(M\) a simply connected torus manifold with \(G\)-action such that \(SO(2l_1)\) is an elementary factor of \(\tilde{G}\) and \(S^1\) acts effectively on \(M\) and \(M^{S^1}\) has codimension two in \(M\).
  
  Then \(M\) is equivariantly diffeomorphic to \(\#_i (S^2\times S^{2l_1})_i\) or \(S^{2l_1+2}\).

  Here the action of \(\tilde{G}\) on  \(S^{2l_1+2}\) is given by the restriction of the usual \(SO(2l_1+3)\)-action to \(\tilde{G}\).
  The action of \(\tilde{G}\) on \(S^2\times S^{2l_1}\) is the product action of the usual action of \(S^1\) and \(SO(2l_1)\) on \(S^2\) and \(S^{2l_1}\), respectively.
  Moreover, the connected sums are equivariant.
\end{lemma}
\begin{proof}
  As described in section~\ref{sec:so2l}, we may replace \(\tilde{G}\) by \(SU(l_1)\times S\times S^1\).
  Let \((\psi,N,A)\) be the admissible triple corresponding to \(M\).
  Then \(\psi\) is completely determined by the discussion in section \ref{sec:so2l} and \(A=N^S=M^{SU(l_1)}\).
  Furthermore \(S\) and \(S^1\) act effectively on \(N\).
  All components of \(N^S\) and \(N^{S^1}\) have codimension two in \(N\).

  By Lemma~\ref{sec:case-g_1-=-5}, \(N\) is simply connected.

  Denote by \(\tilde{M}\) the blow up of \(M\) along \(A\).
  Because all \(T\)-fixed points of \(M\) are contained in \(A\) we have \(l_1\#M^T=\#\tilde{M}^T\).
  On the other hand, \(\tilde{M}\) is a fiber bundle with fiber \(N\) over \(\C P^{l_1-1}\).
  Therefore we have \(l_1\#N^{S\times S^1}=\#\tilde{M}^T\).

  From this \(\#M^T=\# N^{S\times S^1}\) follows.

  Because \(S\) and \(S^1\) act both effectively on \(N\) such that their fixed point sets have codimension two, it follows from the classification of simply connected four-dimensional \(T^2\)-manifolds given in \cite[p. 547,549]{0216.20202} that the \(T\)-equivariant diffeomorphism type of \(N\) is determined by \(\#M^T\) and that \(\#M^T\) is even.

  Therefore the \(S\times S^1\times SU(l_1)\)-equivariant diffeomorphism type of \(M\) is uniquely determined by \(\# M^T=\chi(M)\).
  It follows from Theorem~\ref{sec:case-g_1=so2l_1-2} that the \(SO(2l_1)\times S^1\)-equivariant diffeomorphism type of \(M\) is uniquely determined by \(\chi(M)\).
  Because
  \begin{equation*}
    M_k=
    \begin{cases}
      \#_{i=1}^k (S^2\times S^{2l_1})_i & \text{if } k \geq 1\\
      S^{2l_1+2} & \text{if } k=0
    \end{cases}
  \end{equation*}
  possesses an action of \(\tilde{G}\) and \(\chi(M_k)=2k+2\), the statement follows.
\end{proof}

\begin{cor}
\label{sec:classification-6}
    Let \(\tilde{G}=SO(2l_1)\times SO(2l_2)\) and \(M\) a simply connected torus manifold with \(G\)-action such that \(SO(2l_1)\),   \(SO(2l_2)\)  are elementary factors of \(\tilde{G}\).
  
  Then \(M\) is equivariantly diffeomorphic to \(\#_i (S^{2l_1}\times S^{2l_2})_i\) or \(M=S^{2l_1+2l_2}\).

  Here the action of \(\tilde{G}\) on  \(S^{2l_1+2}\) is given by the restriction of the usual \(SO(2l_1+2l_2+1)\)-action to \(\tilde{G}\).
  The action of \(\tilde{G}\) on \(S^{2l_1}\times S^{2l_2}\) is the product action of the usual action of \(SO(2l_1)\) and \(SO(2l_2)\) on \(S^{2l_1}\) and \(S^{2l_2}\), respectively.
  Moreover, the connected sums are equivariant.
\end{cor}
\begin{proof}
  As described in section~\ref{sec:so2l}, we may replace \(\tilde{G}\) by \(SU(l_1)\times S\times SO(2l_2)\).
  Let \((\psi,N,A)\) be the admissible triple for \((SU(l_1)\times S\times SO(2l_2), SU(l_1))\) corresponding to \(M\).
  Then \(\psi\) is completely determined by the discussion in section \ref{sec:so2l} and \(A=N^S\).
  Furthermore, \(S\) acts effectively on \(N\) such that \(N^{S}\) has codimension two.

  By Lemma~\ref{sec:case-g_1-=-5}, \(N\) is simply connected. Therefore, by Lemma~\ref{sec:classification-1}, the equivariant diffeomorphism-type of \(N\) is uniquely determined by \(\chi(N)\in 2\mathbb{Z}\).
  Because all other parts of the triple \((\psi,N,A)\) are determined by the discussion in section~\ref{sec:so2l} and the equivariant diffeomorphism type of \(N\), it follows that the equivariant diffeomorphism type of \(M\) is determined by \(\chi(N)\).
Let \(T_2\) be the maximal torus \(T\cap SO(2l_2)\) of \(SO(2l_2)\).
Then as in the proof of Lemma~\ref{sec:classification-1} one sees that
\begin{equation*}
  \chi(M)=\#M^T=\#N^{S\times T_2}=\chi(N).
\end{equation*}

Therefore the equivariant diffeomorphism type of \(M\) is uniquely determined by \(\chi(M)\in 2\mathbb{Z}\).
 Because
  \begin{equation*}
    M_k=
    \begin{cases}
      \#_{i=1}^k (S^{2l_1}\times S^{2l_2})_i & \text{if } k \geq 1\\
      S^{2l_1+2l_2} & \text{if } k=0
    \end{cases}
  \end{equation*}
  possesses an action of \(\tilde{G}\) and \(\chi(M_k)=2k+2\), the statement follows.
\end{proof}

At the end of this section we give a classification of four dimensional torus manifolds with \(G\)-action.

\begin{cor}
  Let \(M\) be a four dimensional torus manifold with \(G\)-action, \(G\) a non-abelian Lie-group of rank two.
  Then \(M\) is one of the following
  \begin{equation*}
    \C P^2,\;\C P^1 \times \C P^1,\; S^4,\; S_1^2\times_{\mathbb{Z}_2}S_1^2,\;S_1^2\times_{\mathbb{Z}_2}S_2^2
  \end{equation*}
or a \(S^2\)-bundle over \(\C P^1\).
Here \(S_1^2\) denotes the two-sphere together with the \(\mathbb{Z}_2\)-action generated by the antipodal map and \(S_2^2\) the two-sphere together with the \(\mathbb{Z}_2\)-action generated by a reflection at a hyperplane.
\end{cor}
\begin{proof}
  Let \(\tilde{G}\) be a covering group of \(G\). Then there are the following possibilities using Convention~\ref{sec:g-action-m-3}:
  \begin{align*}
    \tilde{G}=&\;SU(3),\; SU(2)\times SU(2),\;SU(2)\times S^1,\\
    &SU(2)\times SO(3),\;SO(3)\times SO(3),\; SO(3)\times S^1,\;\text{Spin}(4).
  \end{align*}
If \(\tilde{G}=\text{Spin}(4)\), we replace it by \(SU(2)\times S^1\) as before.

Then we have the following admissible \(5\)-tuples:
\begin{center}
     \begin{tabular}{|c|c|c|}
       \(\tilde{G}\)               &\(5\)-tuple                                               &\(M\)\\\hline\hline
       \(SU(3)\)                   &\((\text{const},\text{pt},\emptyset,\emptyset,\emptyset)\)&\(\C P^2\)\\\hline
       \(SU(2)\times SU(2)\)       &\((\text{const},\text{pt},\emptyset,\emptyset,\emptyset)\)&\(\C P^1\times \C P^1\)\\\hline
       \(SU(2)\times S^1\)        &\((\psi,S^2,\emptyset,\emptyset,\emptyset)\)              &\(S^2\)-bundle over \( \C P^1\)\\
                                   &\((\psi,S^2,N,\emptyset,\emptyset)\)                      & \( \C P^2\)\\
                                   &\((\psi,S^2,\{N,S\},\emptyset,\emptyset)\)                &\(S^4\)\\\hline
       \(SU(2)\times SO(3)\)       &\((\text{const},\text{pt},\emptyset,\emptyset,\emptyset)\)& \(\C P^1 \times S^2\)\\\hline
       \(SO(3)\times SO(3)\)       &\((\emptyset,\text{pt},\emptyset,\emptyset,a_{12}=1)\)     & \(S^2_1 \times_{\mathbb{Z}_2} S^2_1\)\\
                                   &\((\emptyset,\text{pt},\emptyset,\emptyset,a_{12}=0)\)     & \(S^2 \times S^2\)\\\hline
       \(SO(3)\times S^1\)         &\((\emptyset,S^2,\emptyset,\emptyset,\emptyset)\)         & \(S^2 \times S^2\)\\
                                   &\((\emptyset,S^2_1,\emptyset,\emptyset,\emptyset)\)       & \(S_1^2\times_{\mathbb{Z}_2} S_1^2\)\\
                                   &\((\emptyset,S^2_2,\emptyset,\emptyset,\emptyset)\)       & \(S_1^2\times_{\mathbb{Z}_2} S_2^2\)\\
                                   &\((\emptyset,S^2_2,\emptyset,S^1,\emptyset)\)             & \(S^4\)\\
     \end{tabular}
  \end{center}
Here \(\psi\) is a group homomorphism \(S(U(1)\times U(1))\rightarrow S^1\). 
\end{proof}

\appendix

\section{Lie-groups}
\label{sec:lie-alg_cal}

\begin{lemma}
\label{sec:lie-algebra-calc}
  Let \(l>1\). Then \(S(U(l)\times U(1))\) is a maximal subgroup of \(SU(l+1)\).
\end{lemma}
\begin{proof}
  Let \(H\) be a subgroup of \(SU(l+1)\) with \(S(U(l)\times U(1))\subset H\subsetneq SU(l+1)\).

  Because \(S(U(l)\times U(1))\) is a maximal connected subgroup of \(SU(l+1)\) the identity component of \(H\) has to be
  \(S(U(l)\times U(1))\).
  Therefore \(H\) is contained in the normalizer of \(S(U(l)\times U(1))\).
  Because \(l>1\),
  \begin{multline*}
    N_{SU(l+1)}S(U(l)\times U(1))/S(U(l)\times U(1))\\= \left(SU(l+1)/S(U(l)\times U(1))\right)^{S(U(l)\times U(1))}= \left(\C P^l\right)^{S(U(l)\times U(1))}
  \end{multline*}
  is just one point.
  Therefore \(H=S(U(l)\times U(1))\) follows.
\end{proof}

\begin{Lemma}
  \label{lem:lie-algebra-calc}
  Let \(\psi:S(U(l)\times U(1))\rightarrow S^1\) be a non-trivial  group homomorphism and
  \begin{align*}
    H_0&=SU(l+1)\times S^1,\\
    H_1&=S(U(l)\times U(1)) \times S^1,\\
    H_2&= \{(g,\psi(g)), g\in S(U(l)\times U(1))\}.
  \end{align*}
  Then \(H_1\) is the only connected proper closed subgroup of \(H_0\), which contains \(H_2\) properly.
\end{Lemma}
\begin{proof}
 Let \(H_2\subset H\subset H_0\) be a closed connected subgroup.
 Then we have
 \begin{equation*}
   \rank H_0 \geq \rank H \geq \rank H_2 = \rank H_0 -1.
 \end{equation*}
At first assume that \(\rank H = \rank H_0\).
Then we have by \cite[p. 297]{mimura91:_topol_of_lie_group_i_and_ii}
\begin{equation*}
  H=H'\times S^1,
\end{equation*}
where \(H'\) is a connected subgroup of maximal rank of  \(SU(l+1)\).
Let \(\pi_1:H_0\rightarrow SU(l+1)\) the projection to the first factor.
Because \(H'=\pi_1(H)\supset \pi_1(H_2)=S(U(l)\times U(1))\) and \(S(U(l)\times U(1))\) is a maximal connected subgroup of \(SU(l+1)\), we have by Lemma \ref{sec:lie-algebra-calc} that \(H=H_1\) or \(H=H_0\).

Now assume that \(\rank H=\rank H_2\).
Then there is a non-trivial group homomorphism \(H\rightarrow S^1\).
Therefore locally \(H\) is a product \(H'\times S^1\), where \(H'\) is a simple group which contains \(SU(l)\) as a maximal rank subgroup.
By \cite[p. 219]{0034.30701}, we have
\begin{equation*}
  H' = E_7,E_8,G_2,SU(l).
\end{equation*}
If \(H'=SU(l)\), then we have \(H=H_2\). Therefore we have to show that the other cases do not occur.
 \begin{center}
     \begin{tabular}{|c|c|c|c|}
       \(l\)&\(\dim H_0\)&\(\dim H'\times S^1\)\\\hline\hline
       \(8\)&\(81\)&\(\dim E_7\times S^1=134\)\\\hline
       \(9\)&\(100\)&\(\dim E_8\times S^1=249\)\\\hline
       \(3\)&\(16\)&\(\dim G_2\times S^1=15\)
     \end{tabular}
  \end{center}
Therefore the first two cases do not occur.
Because there is no \(G_2\)-representation of dimension less than seven, the third case does not occur.
\end{proof}

\begin{lemma}
\label{sec:lie-groups}
  Let \(T\) be a torus and \(\psi_1,\psi_2: S(U(l)\times U(1)) \rightarrow T\) be two group homomorphisms.
  Furthermore, let, for \(i=1,2\),
  \begin{align*}
    H_i=\{(g,\psi_i(g))\in SU(l+1)\times T; g \in  S(U(l)\times U(1))\}
  \end{align*}
  be the graph of \(\psi_i\).
  \begin{enumerate}
  \item If \(l>1\), then \(H_1\) and \(H_2\) are conjugated in \(SU(l+1)\times T\) if and only if \(\psi_1=\psi_2\).
  \item If \(l=1\), then \(H_1\) and \(H_2\) are conjugated in \(SU(l+1)\times T\) if and only if \(\psi_1=\psi_2^{\pm 1}\).
  \end{enumerate}
\end{lemma}
\begin{proof}
  At first assume that \(H_1\) and \(H_2\) are conjugated in \(SU(l+1)\times T\).
  Let \(g'\in SU(l+1)\times T\) such that
  \begin{equation*}
    H_1=g' H_2 g'^{-1}.
  \end{equation*}
  Because \(T\) is contained in the center of \(SU(l+1)\times T\), we may assume that \(g'=(g,1) \in SU(l+1)\times\{1\}\).
  Let \(\pi_1:SU(l+1)\times T \rightarrow SU(l+1)\) be the projection on the first factor.
  Then:
  \begin{equation*}
    S(U(l)\times U(1)) = \pi_1(H_1) = g \pi_1(H_2) g^{-1} = g S(U(l)\times U(1)) g^{-1}.
  \end{equation*}
By Lemma~\ref{sec:lie-algebra-calc}, it follows that
\begin{equation*}
  g \in N_{SU(l+1)}S(U(l)\times U(1))=
  \begin{cases}
    S(U(l)\times U(1)) &\text{if } l>1\\
     N_{SU(2)}S(U(1)\times U(1)) &\text{if } l=1.
  \end{cases}
\end{equation*}
Now for \(h\in S(U(l)\times U(1))\) we have
\begin{equation*}
  (h,\psi_1(h))= g'(g^{-1}hg,\psi_1(h))g'^{-1}.
\end{equation*}
Now \((g^{-1}hg,\psi_1(h))\) lies in \(H_2\). Therefore we may write:
\begin{equation*}
  g'(g^{-1}hg,\psi_1(h))g'^{-1}=g'(g^{-1}hg,\psi_2(g^{-1}hg))g'^{-1}=(h,\psi_2(g^{-1}hg))
\end{equation*}
If \(l>1\) we have 
\begin{equation*}
  \psi_2(g^{-1}hg)=\psi_2(g)^{-1}\psi_2(h)\psi_2(g)=\psi_2(h).
\end{equation*}
Otherwise we have
\begin{equation*}
  \psi_2(g^{-1}hg)=\psi_2(h^{\pm 1})=\psi_2(h)^{\pm 1}.
\end{equation*}

The other implications are trivial.
Therefore the statement follows.
\end{proof}

\begin{lemma}
  \label{sec:lie-groups-1}
  Let \(l\geq 1\).
  \(\text{Spin}(2l)\) is a maximal connected subgroup of \(\text{Spin}(2l+1)\).
  Its normalizer consists out of two components.
\end{lemma}
\begin{proof}
  By \cite[p. 219]{0034.30701}, \(\text{Spin}(2l)\) is a maximal connected subgroup of \(\text{Spin}(2l+1)\).
  \begin{equation*}
    N_{\text{Spin}(2l+1)}\text{Spin}(2l)/\text{Spin}(2l)= \left(\text{Spin}(2l+1)/\text{Spin}(2l)\right)^{\text{Spin}(2l)}=\left(S^{2l}\right)^{\text{Spin}(2l)}
  \end{equation*}
  consists out of  two points. Therefore the second statement follows.
\end{proof}

\begin{lemma}
\label{sec:lie-groups-2}
  Let \(G\) be a Lie-group, which acts on the manifold \(M\).
  Furthermore, let \(N\subset M\) be a submanifold.
  If the intersection of \(Gx\) and \(N\) is transverse in \(x\) for all \(x\in N\), then \(GN\) is open in \(M\).
\end{lemma}
\begin{proof}
  We will show that \(f:G\times N \rightarrow M\), \((h,x)\mapsto hx\) is a submersion.
  Because a submersion is an open map, it follows that \(GN=f(G\times N)\) is open in \(M\).
  For \(g \in G\), let
  \begin{align*}
    l_g: G\times N & \rightarrow G\times N\\
    (h,x)&\mapsto (gh,x)
  \end{align*}
and
\begin{align*}
  l_g': M&\rightarrow M\\
  x&\mapsto gx.
\end{align*}

Then we have for all \(g\in G\)
\begin{equation*}
  f=l_g'\circ f \circ l_{g^{-1}}.
\end{equation*}
Now for \((g,x)\in G\times N\) we have
\begin{equation*}
  D_{(g,x)}f=D_xl_g'\, D_{(e,x)}f\, D_{(g,x)}l_{g^{-1}}.
\end{equation*}
Because \(Gx\) and \(N\) intersect transversely in \(x\), the differential \(D_{(e,x)}f\) is surjective.
Because \(l_g'\), \(l_{g^{-1}}\) are diffeomorphisms, it follows that  \(D_{(g,x)}f\) is surjective.
Therefore \(f\) is a submersion.
\end{proof}

\section{Generalities on torus manifolds}
\label{sec:gen-tor-mgf}

\begin{Lemma}
\label{lem:torus-m2}
  Let \(M\) be a torus manifold and \(M_1,\dots,M_k\) pairwise distinct characteristic submanifolds of \(M\) with \(N=M_1\cap\dots\cap M_k\neq \emptyset\).
 Then each \(M_i\) intersects transversely with \(\bigcap_{j=1}^{i-1}M_j\).
 Therefore \(N\) is a submanifold of \(M\) with \(\codim N = 2k\) and \(\dim \langle\lambda(M_1),\dots,\lambda(M_k)\rangle=k\).
Furthermore, \(N\) is the union of some components of \(M^{\langle\lambda(M_1),\dots,\lambda(M_k)\rangle}\).
\end{Lemma}
\begin{proof}
  We prove the lemma by induction on \(k\).
  Let \(k\geq 1\) and \(x\in N\). Then we have
  \begin{equation*}
    T_xM=\bigcap_{i=1}^kT_xM_i \oplus \bigoplus_j V_j,
  \end{equation*}
where the \(V_j\) are one-dimensional complex  \(\langle\lambda(M_1),\dots,\lambda(M_k)\rangle\)-representations.
Since the \(M_i\) have codimension two in \(M\), each \(\lambda(M_i)\) acts non-trivially on exactly one \(V_{j_i}\).

If \(\codim \bigcap_{i=1}^kT_xM_i < 2k\), then there are \(i_1\) and \(i_2\), such that \(V_{j_{i_1}}=V_{j_{i_2}}\).
Therefore
\begin{equation*}
  T_xM_{i_1}=T_xM_{i_2}=T_xM^{\langle\lambda(M_{i_1}),\lambda(M_{i_2})\rangle}
\end{equation*}
has codimension two.

Since \(\langle\lambda(M_{i_1}),\lambda(M_{i_2})\rangle\) has dimension two, it does not act almost effectively on \(M\).
This is a contradiction.
 Therefore \(\bigcap_{i=1}^kT_xM_i\) has codimension \(2k\).
By induction hypothesis \(\bigcap_{i=1}^{k-1}M_i\) is a submanifold of codimension \(2k-2\) and \(T_x\bigcap_{i=1}^{k-1}M_i=\bigcap_{i=1}^{k-1}T_xM_i\).
Thus, \(M_k\) and \(\bigcap_{i=1}^{k-1}M_i\) intersect transversely.
Therefore \(N\) is a submanifold of \(M\) of codimension \(2k\).

If \(\langle\lambda(M_1),\dots,\lambda(M_k)\rangle\) has dimension smaller than \(k\) then the weights of the \(V_j\) are linear dependent.
Therefore there is \((a_1,\dots,a_k)\in \mathbb{Z}^k-\{0\}\), such that
\begin{equation*}
  \C=V_1^{a_1}\otimes\dots\otimes V_k^{a_k},
\end{equation*}
where \(\C\) denotes the trivial \(\langle\lambda(M_1),\dots,\lambda(M_k)\rangle\)-representation.
 This gives a contradiction because each \(\lambda(M_i)\) acts non-trivially on exactly one \(V_{j}\).

Because \(\langle\lambda(M_1),\dots,\lambda(M_k)\rangle\) has dimension \(k\), \(M^{\langle\lambda(M_1),\dots,\lambda(M_k)\rangle}\) has dimension at most \(2n-2k\).
But \(N\) is contained in \(M^{\langle\lambda(M_1),\dots,\lambda(M_k)\rangle}\) and has dimension \(2n-2k\).
Therefore it is the union of some components of \(M^{\langle\lambda(M_1),\dots,\lambda(M_k)\rangle}\).
\end{proof}

\begin{Lemma}
\label{lem:torus-m1}
  Let \(M\) be a torus manifold of dimension \(2n\) and \(N\) a component of the intersection of \(k (\leq n)\) characteristic submanifolds \(M_1,\dots,M_k\)  of \(M\) with \(N^T\neq \emptyset\).
  Then \(N\) is a torus manifold.
  Moreover, the characteristic submanifolds of \(N\) are given by the components of intersections of characteristic submanifolds \(M_i \neq M_1,\dots,M_k\) of \(M\) with \(N\), which contain a \(T\)-fixed point.
\end{Lemma}
\begin{proof}
  Let \(M_i \neq M_1,\dots, M_k\) be a characteristic submanifold of \(M\) with \((M_i\cap N)^T \neq \emptyset\).
  Then, by Lemma~\ref{lem:torus-m2}, each component of \(M_i\cap N\) which contains a \(T\)-fixed point has codimension two in \(N\). 
  That means that they are characteristic submanifolds of \(N\).
  
  Now let \(N_1\subset N\) be a characteristic submanifold and \(x \in N_1^T\).
  Then we have
  \begin{equation*}
    T_xM=T_xN_1\oplus V_0 \oplus N_x(N,M)
  \end{equation*}
  as \(T\)-representations with \(V_0\) a one dimensional complex \(T\)-representation.
  Let \(M_i\) be the characteristic submanifold of \(M\), which corresponds to \(V_0\).
  Then \(N_1\) is the component of the intersection \(M_i\cap N\), which contains \(x\).
\end{proof}

\begin{lemma}
\label{sec:gener-torus-manif}
  Let \(M\) be a \(2n\)-dimensional torus manifold and \(T'\) a subtorus of \(T\).
  If \(N\) is a component of \(M^{T'}\), which contains a \(T\)-fixed point \(x\), then \(N\) is a component of the intersection of some characteristic submanifolds of \(M\).
\end{lemma}
\begin{proof}
  By Lemma~\ref{lem:torus-m2}, the intersection of the characteristic submanifolds \(M_1,\dots M_k\) is a union of some components of \(M^{\langle\lambda(M_1),\dots,\lambda(M_k)\rangle}\).

  Therefore we have to show that there are characteristic submanifolds \(M_1,\dots,M_k\) of \(M\) such that
  \begin{equation*}
    T_xN=T_x\left(M_1\cap\dots\cap M_k\right).
  \end{equation*}
  There are \(n\) characteristic submanifolds \(M_1,\dots,M_n\) which intersect transversely in \(x\).
  Therefore we have
  \begin{equation*}
    T_xM=N_x(M_1,M)\oplus \dots \oplus N_x(M_n,M).
  \end{equation*}
  We may assume that there is a \(1\leq k\leq n\) such that \(T'\) acts trivially on \(N_x(M_i,M)\) for \(i>k\) and non-trivially on \(N_x(M_i,M)\) for \(i\leq k\).
Then we have
\begin{equation*}
  T_xN=\left(T_xM\right)^{T'}= N_x(M_{k+1},M)\oplus\dots\oplus N_x(M_n,M)=T_x\left(M_1\cap\dots\cap M_k\right).
\end{equation*}
\end{proof}

\begin{lemma}
\label{sec:gener-torus-manif-1}
  Let \(M\) be a torus manifold with \(T^n\times \mathbb{Z}_2\)-action, such that \(\mathbb{Z}_2\) acts non-trivially on \(M\).
  Furthermore, let \(B\subset M\) be a submanifold of codimension one on which \(\mathbb{Z}_2\) acts trivially and \(N\) the intersection of characteristic submanifolds \(M_1,\dots,M_k\) of \(M\).
  Then \(B\) and \(N\) intersect transversely. 
\end{lemma}
\begin{proof}
  Let \(x\in B\cap N\) then we have the \(\langle \lambda(M_1),\dots,\lambda(M_k)\rangle\times \mathbb{Z}_2\)-representation \(T_xM\).
  It decomposes as the sum of the eigenspaces of the non-trivial element of \(\mathbb{Z}_2\).
  Because \(B\) has codimension one the eigenspace to the eigenvalue \(-1\) is one dimensional.
  Because the irreducible non-trivial torus representations are two-dimensional, we have
  \begin{align*}
    T_xN&= \left(T_xM\right)^{\langle \lambda(M_1),\dots,\lambda(M_k)\rangle} =T_xM^{\langle \lambda(M_1),\dots,\lambda(M_k)\rangle\times\mathbb{Z}_2}\oplus N_x(B,M)^{\langle \lambda(M_1),\dots,\lambda(M_k)\rangle} \\
    & = T_xM^{\langle \lambda(M_1),\dots,\lambda(M_k)\rangle\times\mathbb{Z}_2}\oplus N_x(B,M).
  \end{align*}
  That means that the intersection is transverse.
\end{proof}

\begin{lemma}
\label{sec:gener-torus-manif-2}
  Let \(M^{2n}\) be a \((\mathbb{Z}_2)_1\times(\mathbb{Z}_2)_2\)-manifold such that \((\mathbb{Z}_2)_i\) acts non-trivially on \(M\).
  Furthermore, let \(B_i\subset M\), \(i=1,2\), be closed connected submanifolds of codimension one such that \((\mathbb{Z}_2)_i\) acts trivially on \(B_i\).
  Then the following statements are equivalent:
  \begin{enumerate}
  \item \(B_1,B_2\) intersect transversely
  \item \(B_1\neq B_2\)
  \item  \((\mathbb{Z}_2)_1\times(\mathbb{Z}_2)_2\) acts effectively on \(M\) or \(B_1\cap B_2= \emptyset\)
  \end{enumerate}
\end{lemma}
\begin{proof}
  Denote by \(V_i\) the non-trivial real irreducible representation of \((\mathbb{Z}_2)_i\).
  Let \(x \in B_1\cap B_2\).
  Then for the  \((\mathbb{Z}_2)_1\times(\mathbb{Z}_2)_2\)-representation \(T_xM\) there are two possibilities:
  \begin{equation*}
    T_xM=
    \begin{cases}
      \mathbb{R}^{2n-1}\oplus V_1 \otimes V_2\\
      \mathbb{R}^{2n-2}\oplus V_1 \oplus V_2
    \end{cases}
  \end{equation*}
  In the first case \(B_i\), \(i=1,2\), is the component of \(M^{(\mathbb{Z}_2)_1\times(\mathbb{Z}_2)_2}\) containing \(x\) and  \((\mathbb{Z}_2)_1\times(\mathbb{Z}_2)_2\) acts non-effectively on \(M\).
  In the second case  \((\mathbb{Z}_2)_1\times(\mathbb{Z}_2)_2\) acts effectively on \(M\) and \(B_1,B_2\) intersect transversely in \(x\).
  
  All conditions given in the lemma imply that we are in the second case or \(B_1\cap B_2=\emptyset\).
  Therefore they are equivalent.
\end{proof}

\begin{remark}
  Lemmas~\ref{lem:torus-m2},~\ref{sec:gener-torus-manif-1} also hold if we do not require that a characteristic manifold contains a \(T\)-fixed point.
\end{remark}

\bibliography{test}{}

\providecommand{\bysame}{\leavevmode\hbox to3em{\hrulefill}\thinspace}
\providecommand{\MR}{\relax\ifhmode\unskip\space\fi MR }
\providecommand{\MRhref}[2]{%
  \href{http://www.ams.org/mathscinet-getitem?mr=#1}{#2}
}
\providecommand{\href}[2]{#2}
\begin{thebibliography}{10}

\bibitem{0225.20020}
M.~Aschbacher, \emph{{On finite groups generated by odd transpositions. I.}},
  Math. Z. \textbf{127} (1972), 45--56 (English).

\bibitem{0034.30701}
A.~Borel and J.~de~Siebenthal, \emph{{Les sous-groupes ferm\'es de rang maximum
  des groupes de Lie clos.}}, Comment. Math. Helv. \textbf{23} (1949), 200--221
  (French).

\bibitem{0246.57017}
G.~E. Bredon, \emph{{Introduction to compact transformation groups.}}, {Pure
  and Applied Mathematics, 46. New York-London: Academic Press.}, 1972.

\bibitem{0581.22009}
T.~Br\"ocker and T.~tom Dieck, \emph{{Representations of compact Lie groups.}},
  {Graduate Texts in Mathematics, 98. New York etc.: Springer-Verlag.}, 1985
  (English).

\bibitem{1012.52021}
V.M. Buchstaber and T.E. Panov, \emph{{Torus actions and their applications in
  topology and combinatorics.}}, {University Lecture Series. 24. American
  Mathematical Society (AMS)}, 2002 (English).

\bibitem{davis91:_convex_coxet}
M.~Davis and T.~Januszkiewicz, \emph{{Convex polytopes, Coxeter orbifolds and
  torus actions}}, Duke Math. J. \textbf{62} (1991), no.~2, 417--451 (English).

\bibitem{0408.14001}
P.~Griffiths and J.~Harris, \emph{{Principles of algebraic geometry.}}, {Pure
  and Applied Mathematics. A Wiley-Interscience Publication. New York etc.:
  John Wiley \& Sons. XII }, 1978 (English).

\bibitem{0712.57010}
F.~Hirzebruch and P.~Slodowy, \emph{{Elliptic genera, involutions, and
  homogeneous spin manifolds.}}, Geom. Dedicata \textbf{35} (1990), no.~1-3,
  309--343 (English).

\bibitem{0153.53703}
K.~J\"{a}nich, \emph{{Differenzierbare Mannigfaltigkeiten mit Rand als
  Orbitr\"aume differenzierbarer G-Mannigfaltigkeiten ohne Rand.}}, Topology
  \textbf{5} (1966), 301--320 (German).

\bibitem{pre05136053}
M.~Kankaanrinta, \emph{{Equivariant collaring, tubular neighbourhood and gluing
  theorems for proper Lie group actions.}}, Algebr. Geom. Topol. \textbf{7}
  (2007), 1--27 (English).

\bibitem{kuroki_pre_2_2009}
S.~Kuroki, \emph{Classification of quasitoric manifolds with codimension one
  extended actions}, Preprint (2009) (English).

\bibitem{kuroki_pre_3_2009}
\bysame, \emph{Classification of torus manifolds with codimension one extended
  actions}, Preprint (2009) (English).

\bibitem{kuroki_pre_1_2009}
\bysame, \emph{Characterization of homogeneous torus manifolds}, Osaka J. Math.
  \textbf{47} (2010), no.~1, 285--299 (English).

\bibitem{0940.57037}
M.~Masuda, \emph{{Unitary toric manifolds, multi-fans and equivariant index.}},
  Tohoku Math. J., II. Ser. \textbf{51} (1999), no.~2, 237--265 (English).

\bibitem{1111.57019}
M.~Masuda and T.~Panov, \emph{{On the cohomology of torus manifolds.}}, Osaka
  J. Math. \textbf{43} (2006), no.~3, 711--746 (English).

\bibitem{0567.53031}
D.~McDuff, \emph{{Examples of simply-connected symplectic non-K\"ahlerian
  manifolds.}}, J. Differ. Geom. \textbf{20} (1984), 267--277 (English).

\bibitem{0298.57008}
J.~W. Milnor and J.~D. Stasheff, \emph{{Characteristic classes.}}, {Annals of
  Mathematics Studies. No.76. Princeton, N.J.: Princeton University Press and
  University of Tokyo Press}, 1974 (English).

\bibitem{mimura91:_topol_of_lie_group_i_and_ii}
M.~Mimura and H.~Toda, \emph{{Topology of Lie Groups I and II}}, Translations
  of Mathematical Monographs; Volume 91, AMS, 1991 (English).

\bibitem{0796.57001}
A.~L. Onishchik, \emph{{Topology of transitive transformation groups.}},
  {Leipzig: Johann Ambrosius Barth. xv }, 1994 (English).

\bibitem{0216.20202}
P.~Orlik and F.~Raymond, \emph{{Actions of the torus on 4-manifolds. I.}},
  Trans. Am. Math. Soc. \textbf{152} (1970), 531--559 (English).

\end{thebibliography}


\bibliographystyle{amsplain}

\end{document}